\documentclass[11pt,a4paper]{article}

\usepackage[USenglish]{babel}
\usepackage[latin1]{inputenc}
\usepackage{microtype}
\usepackage{amsmath} 
\usepackage{amsthm,amssymb,mathrsfs} 
\usepackage[T1]{fontenc}
\usepackage{ae,aecompl}
\usepackage{times}
\usepackage[colorlinks,pagebackref, hypertexnames=true, unicode, pdfpagelabels]{hyperref}  
\usepackage[alphabetic,backrefs]{amsrefs}
\usepackage{enumerate}
\usepackage{commath} 
\usepackage{tensor} 

\newtheorem{theorem}{Theorem}[section]

\newtheorem{corollary}[theorem]{Corollary}
\newtheorem{lemma}[theorem]{Lemma}
\newtheorem{conjecture}[theorem]{Conjecture}

\theoremstyle{remark}
\newtheorem{remark}[theorem]{Remark}

\theoremstyle{definition}
\newtheorem{definition}[theorem]{Definition}
\newtheorem{example}[theorem]{Example}

\newtheorem*{thm}{Theorem}

\numberwithin{equation}{section}

\newcommand{\hk}{\mathbin{\! \hbox{\vrule height0.3pt width5pt 
depth 0.2pt \vrule height5pt width0.4pt depth 0.2pt}}}
\newcommand{\ws}{\ensuremath{w^{\flat}}}
\begin{document}

\title{${G}_2$ holonomy manifolds are superconformal}
\author{L\'azaro O. Rodr\'iguez D\'iaz\thanks{IMECC-UNICAMP, S\~ao Paulo. Supported by S\~ao Paulo State Research Council (Fapesp) grant 2014/13357-7.}.}
\date{}
\maketitle

\begin{abstract}
We study the chiral de Rham complex (CDR) over a manifold $M$ with
holonomy $G_2$. We prove that the vertex algebra of global sections of the
CDR associated to $M$ contains two commuting copies of the Shatashvili-Vafa $G_2$ superconformal algebra. Our proof is a \textit{tour de force}, based on explicit computations.   
\end{abstract}

\tableofcontents
\newpage

\section{Introduction}
\label{cap:introduction}

The chiral de Rham complex (CDR) was introduced in \cite{Malikov-Schechtman-Vaintrob99} by  Malikov, Schechtman and Vaintrob. It is a sheaf of supersymmetric vertex algebras over a smooth manifold $M$. Locally, over a coordinate chart, it is simply $n$ copies of the $bc-\beta\gamma$ system ($n$ is the dimension of $M$), i.e., a tensor product of $n$ copies of the Clifford vertex algebra and $n$ copies of the Weyl vertex algebra, and then extended to $M$ by gluing on the intersections of these coordinate charts.

One of the most important facts about the CDR proved in the seminal work \cite{Malikov-Schechtman-Vaintrob99} is that there exists an embedding of the $N=2$ superconformal vertex algebra into the vertex algebra of global sections of CDR when $M$ is a Calabi-Yau manifold. This idea of looking for special vertex subalgebras of the vertex algebra of global sections of the CDR was further investigated in \cite{Ben-Zvi-Heluani-Szczesny} where it was proved that when $M$ is a hyperk\"{a}lher manifold, the $N=4$ superconformal vertex algebra appears as a subalgebra of global sections of CDR. Subsequently in \cite{Heluani09} it was shown that in fact there are two commuting copies of the $N=2$ superconformal algebra ($N=4$ superconformal algebra) of half the central charge when $M$ is Calabi-Yau (respectively hyperk\"{a}lher).

It is by now well known in the physics literature (\cite{Odake89}, \cite{HowePapadopoulos91}, \cite{Howe-Papadopoulos93}, \cite{Shatashvili-Vafa95}) that even though we can define a sigma model on an arbritary target space, in order for the theory to be supersymmetric, the target space manifold must be of special holonomy. This implies the existence of covariantly constant $p$-forms,  and the existence of such forms on the target space manifold implies the existence of extra elements in the chiral algebra.

In \cite{Ekstrand-Heluani-Kallen-Zabzine13} a program was launched in order to understand the facts mentioned above as a relation between special holonomy of $M$ and the existence of certain subalgebras of CDR. To pursue this objective they introduced an embedding (Theorem \ref{theorem Reimundo}) different than the one in \cite{Malikov-Schechtman-Vaintrob99}, of the space of differential forms $\Omega^{*}(M)$ into global sections of CDR (in fact, they introduced two different embeddings). When the manifold M has special holonomy it admits covariantly constant forms, so they obtain corresponding sections of CDR and the subalgebra generated by them. In particular, they recover the result of \cites{Ben-Zvi-Heluani-Szczesny, Heluani09} when M is a Calabi-Yau manifold or hyperk\"{a}lher. They also constructed two commuting copies of the Odake algebra on the space of global sections of CDR of a Calabi-Yau threefold, and conjectured a similar result for $G_{2}$ holonomy manifolds.

More precisely they conjectured \cite[Conjecture 7.3]{Ekstrand-Heluani-Kallen-Zabzine13} (Conjecture \ref{conjecture}) that: \emph{if $M$ is a manifold with $G_{2}$ holonomy, the vertex algebra of global sections of CDR contains two commuting copies of the Shatashvili-Vafa $G_{2}$ superconformal algebra}, each of these copies should be generated by the global section that comes from the covariantly constant $3$-form (that defines the geometry) using the two different embeddings of $\Omega^{*}(M)$ that they defined. 

The \emph{Shatashvili-Vafa $G_{2}$ superconformal algebra} appeared as the chiral algebra associated to the sigma model with target a manifold with $G_{2}$ holonomy in \cite{Shatashvili-Vafa95}, its classical counterpart had been studied in \cite{Howe-Papadopoulos93}. It is a superconformal vertex algebra with six generators $\{ L, G, \Phi, K, X, M \}$. It is an extension of the $N=1$ superconformal algebra of central charge $c=21/2$ (formed by the super-partners $\{L,G\}$) by two fields $\Phi$ and $K$, primary of conformal weight $\frac{3}{2}$ and $2$ respectively, and their superpartners $X$ and $M$ (of conformal weight $2$ and $\frac{5}{2}$ respectively). Their OPEs can be found in subsection \ref{Shatashvili-Vafa algebra} in the language of lambda brackets of \cite{dandrea}. 

This algebra is a member of a two-parameter family $SW(\frac{3}{2}, \frac{3}{2}, 2)$ of non-linear $W$-algebras previously studied in \cite{Blumenhagen92}, it is parametrized by $(c,\varepsilon)$ ($c$ is the central charge and $\varepsilon$ the coupling constant). The Shatashvili-Vafa $G_{2}$ algebra is a quotient of $SW(\frac{3}{2},\frac{3}{2}, 2)$ with $c=\frac{21}{2}$ and $\varepsilon=0$, in other words is the only one among this family which has central charge $c=\frac{21}{2}$ and contains the tri-critical Ising model as a subalgebra.

The above conjecture was checked in \cite{Ekstrand-Heluani-Kallen-Zabzine13} in the case when the manifold $M=\mathbb{R}^{7}$ is the flat space, and when $M=CY_{3}\times S^{1}$, where $CY_{3}$ is a compact Calabi-Yau threefold and $S^{1}$ is a circle, using the above mentioned result about the Odake algebra.

In this paper we prove the conjecture when $M$ is an arbitrary non-flat $G_{2}$ manifold. Our approach is a \textit{tour de force}, based on explicit computations, some of them really long, and some abstract algebraic manipulations. But the beauty of these after all is that we need to use many of the known identities in $G_{2}$ geometry, all the computations are tied in a non-trivial way to the geometry of the manifold. To perform the computations we have used the \ttfamily{Mathematica} \normalfont package OPEdefs created by Kris Thielemans \cite{Thielemans91} for symbolic computation of operator product expansion, and the computer algebra system \ttfamily{Cadabra} \normalfont \cite{Cadabra07} created by Kasper Peeters.

In section \ref{sec-vertex superalgebras} we recap the basic facts about vertex superalgebras and we introduce the main examples that are used in the paper. In section \ref{sec-the Chiral de Rham Complex} we review the construction of the CDR as well as the main tool (Theorem~\ref{theorem Reimundo}) used in the paper: the embedding of the space of differential forms $\Omega^{*}(M)$ into global sections of CDR, introduced by \cite{Ekstrand-Heluani-Kallen-Zabzine13}. In section \ref{sec-G2 holonomy manifolds} we recall some background material on $G_{2}$ geometry as well as many of the known identities necessary for the computations in the next section. In section \ref{CDR of a G2 manifold} we prove the main result of this paper:
\begin{thm}
Let $M$ be a $G_{2}$ holonomy manifold, then the space of global sections of the CDR associated to $M$ contains two commuting copies of the Shatashvili-Vafa $G_2$ superconformal algebra of central charge $\frac{21}{2}$.
\end{thm}
The reader is referred to Theorem~\ref{maintheorem} for a more explicit description of the generators of these pairs of algebras. It is important to note that the Conjecture~\ref{conjecture} provides the explicit candidates (global sections of the CDR) for the generators, the challenge is to verify that they actually generate the Shatashvili-Vafa $G_2$ superconformal algebra and that the two copies commute. 

We develop the proof of Theorem~\ref{maintheorem} in four steps. First we prove in Theorem \ref{maintheorem2} that the space of global sections of CDR contains two commuting pairs of $N=1$ superconformal algebras, at central charge $21/2$ and $7/10$ respectively, the last one is precisely the tri-critical Ising model. Secondly in subsection \ref{subs-linearbrackets} we check the linear $\lambda$-brackets that were not already computed in the way of proving Theorem~\ref{maintheorem2}. In third place we verify the non-linear $\lambda$-brackets under the assumption a non-linear identity (\ref{ideal in the definition of the algebra}) among the generators is satisfied. Finally in \ref{subs-checkingtherelation} we check the non-linear identity (\ref{ideal in the definition of the algebra}), see Remark \ref{rm-discussion about non-linearities}. Along the way of proving Theorem~\ref{maintheorem} we perform a lot of computations in local coordinates, for this we make extensive use of properties of the manifold $M$ such as: the Ricci flatness, the contractions between the $3$-form and its Hodge dual (see page~\pageref{ContractionsThree-FourForm}), the $3$-form is parallel, the symmetries of the Riemann curvature tensor, etc. To avoid unnecessary calculations we take advantage of the space of global sections of the CDR is a vertex algebra, to use the axioms as well as the identities satisfied by a vertex algebra, e.g., the Jacobi identity, the Borcherds identity, etc, to derive $\lambda$-brackets from the ones already computed. 

It is remarkable that Borcherds identity was proved very useful computing the non-linear $\lambda$-brackets in subsection \ref{subsubsec-non linear lambda brackets}. It is also engaging the appearance of the first Pontryagin class $p_{1}(M)$ of the manifold $M$ in the proof of the non-linear identity (\ref{ideal in the definition of the algebra}) (see Remark \ref{rm-Pontryagin}), taking into account that the only oriented characteristic class of interest for a $G_{2}$-manifold is $p_{1}(M)$ \cite{CHNP15}.

A new feature of the Shatashvili-Vafa $G_2$ superconformal algebra not shared by their cousins the $N=2$ superconformal algebra (the $N=4$ superconformal algebra) in the Calabi-Yau manifold case (respectively in the hyperk\"{a}lher case) is the existence of non-linear $\lambda$-brackets among the generators, i.e., in some $\lambda$-brackets appear the normally ordered product of generators. This is a direct consequence of the geometry of the manifolds, in an almost Hermitian manifold even if the metric and the complex structure satisfied some compatibility condition they are essentially independent, though in the $G_{2}$ case the metric and the cross product are both determined in a highly nonlinear way from the $3$-form. 

It is worth to mention that our proof of the conjecture is more tortuous than the proofs of previous results for Calabi-Yau and hyperk\"{a}lher manifolds \cites{Ben-Zvi-Heluani-Szczesny, Heluani09, Ekstrand-Heluani-Kallen-Zabzine13}, this is due to the lack of `good' coordinates systems in the $G_2$-holonomy case, it would be interesting to explore if there exist some coordinate system in which the quantum corrections to the fields (\ref{correction3form}) and (\ref{correction4form}) get simplified or even vanish as happens when $M$ is Calabi-Yau or hyperk\"{a}lher.

Another related question is the following. There is a proposal of how to define a topological twist of the Shatashvili-Vafa $G_2$ superconformal algebra in \cite{Shatashvili-Vafa95}, in fact, there is a more recent approach \cite{BoerNaqviShomer}. There are known a few necessary topological conditions a manifold with holonomy $G_2$ should satisfied, but we are far from knowing sufficient conditions. Question: can we put the `topological' Shatashvili-Vafa $G_2$ superconformal algebra locally inside the $bc-\beta\gamma$ systems in such a way the obstruction to be globally well defined as a subalgebra of the CDR be sufficient conditions for having holonomy $G_2$? In the Calabi-Yau case this is exactly what happens, the vanishing of the first Chern class of the manifold is the obstruction to having the topological $N=2$ superconformal algebra globally well defined \cite[Section 4]{Malikov-Schechtman-Vaintrob99}.

Among the special holonomy manifolds we are missing the $Spin(7)$ case in dimension $8$. In view of the present work is reasonable to expect that an analogous theorem can be proved, i.e., the space of global sections of the CDR associated to a $Spin(7)$-holonomy manifold contains two commuting copies of the Shatashvili-Vafa $Spin(7)$ superconformal algebra (subsection \ref{Shatashvili-Vafa Spin(7) algebra}) of central charge $12$ (see Remark \ref{rm-commentaboutSpin(7)case}). It follows easily from our result that for a $Spin(7)$-manifold of the form $M\times\mathbb{R}$ where $M$ is a $G_{2}$-holonomy manifold the above claim is true.

\emph{Acknowledgements:} The author would like to thank Reimundo Heluani for generously sharing his insights and ideas and for the encouragement despite the long computations; the present paper is influenced by his views about the subject.

\section{Vertex superalgebras}\label{sec-vertex superalgebras}

In this section we recall some facts about vertex superalgebras, for details the reader is referred to \cites{Kac96, DeSole-Kac06}.

Let $V$ be a vector superspace, i.e., a vector space with a decomposition $V=V_{\bar{0}}\oplus V_{\bar{1}}$ for $\bar{0},\bar{1}\in \mathbb{Z}/2\mathbb{Z}$. We call $V_{\overline{0}}$ the even space and $V_{\overline{1}}$ the odd space. If $v\in V_{\alpha}$, one writes $p(v)=\alpha$ and calls it the parity of $v$.

The algebra $End V$ acquires a $\mathbb{Z}/2\mathbb{Z}$ grading by letting

$$(End V)_{\bar{\alpha}}=\{A\in End V: A(V_{\bar{\beta}})\subset V_{\bar{\alpha}+\bar{\beta}}\}$$

for $\bar{\alpha},\bar{\beta}\in \mathbb{Z}/2\mathbb{Z}$.

An $End V$-\emph{valued field} is a formal distribution of the form

$$a(z)=\sum_{n\in\mathbb{Z}}a_{(n)}z^{-n-1}, \hspace{0.3in} a_{(n)}\in EndV$$

such that for every $v\in V$, we have $a_{(n)}v=0$ for large enough $n$.

We now proceed to give two different definitions of a vertex algebra that we use freely through the text, for a proof of the equivalence of these definitions we refer to \cite{DeSole-Kac06}.

\begin{definition}\label{def-vertex1}

A vertex superalgebra consists of the data of a  vector superspace $V$ (the space of states) , $\left|0\right>\in V_{\bar{0}}$ is a vector (the vacuum vector), an even endomorphism $T$, and a parity preserving linear map $a\rightarrow Y(a,z)$ from $V$ to $EndV$-valued fields (the state-field correspondence). This data should satisfy the following set of axioms:

\begin{itemize}
\item Vacuum axioms:
\begin{gather*}
Y(\left|0\right>,z)=Id,\qquad Y(a,z)\left|0\right>|_{z=0}=a, \qquad T\left|0\right>=0.
\end{gather*}

\item Translation invariance:
\begin{equation}
[T,Y(a,z)]=\partial_{z}Y(a,z). \nonumber
\end{equation}

\item Locality:
\begin{equation}
(z-w)^{n}[Y(a,z),Y(b,w)]=0 \hspace{0.3in} n\gg 0. \nonumber
\end{equation}
\end{itemize}

\end{definition}

\fussy
Given a vertex superalgebra $V$ and $a\in V$ we expand the fields
\begin{align}
Y(a,z)=\sum_{n\in \mathbb{Z}}a_{(n)}z^{-n-1}
\end{align}
and we call the endomorphisms $a_{(n)}$ the \emph{Fourier modes} of $Y(a,z)$.
\sloppy

In any vertex algebra we have the following identity, known as the \emph{Borcherds identity}:
\small
\begin{align}\label{Borcherds identity}
&\sum_{j\in\mathbb{Z_{+}}}(-1)^{j}\binom{n}{j}\left(a_{(m+n-j)}\left(b_{(k+j)}c\right)-(-1)^{n}(-1)^{p(a)p(b)}b_{(n+k-j)}\left(a_{(m+j)}c\right)\right) \nonumber\\ 
&=\sum_{j\in\mathbb{Z_{+}}}\binom{m}{j}\left(a_{(n+j)}b\right)_{(m+k-j)}c, 
\end{align}
\normalsize
for all $a, b, c\in V$, $m, n, k\in \mathbb{Z}$. In fact this identity appeared in the original Borcherds definition of vertex algebras \cite{Borcherds86}.

Given a vertex algebra $V$ we define two operations:
\begin{gather}\label{def-lambda bracket and normally ordered product}
[a_{\lambda} b]=\sum_{j\geq 0}\frac{\lambda^{j}}{j!}a_{(j)}b, \qquad :ab:=a_{(-1)}b.
\end{gather}
The first is called the $\lambda$-bracket and the second is called the normally ordered product. These operations, the vacuum vector $\left|0\right>$ and the derivation $T$ determine the structure of the vertex algebra as follows from the next definition of vertex algebra.

We need to introduce the definition of a Lie conformal algebra.
\begin{definition}
A Lie conformal superalgebra $R=R_{\bar{0}}\oplus R_{\bar{1}}$ is a $\mathbb{Z}/2\mathbb{Z}$-graded $\mathbb{C}[\partial]$-module endowed with a parity preserving $\mathbb{C}$-bilinear $\lambda$-bracket
\begin{equation}
[._{\lambda}.]: R\otimes R\rightarrow R\otimes\mathbb{C}[\lambda]
\end{equation}
which satisfies:

\begin{itemize}
\item[(i)]sesquilinearity:
\begin{equation}\label{sesquilinearity}
[\partial a_{\lambda} b]=-\lambda[a_{\lambda}b],\;\; [ a_{\lambda}\partial b]=(\lambda+\partial)[a_{\lambda}b],
\end{equation} 
\item[(ii)]skewsymmetry:
\begin{equation}\label{skewsymmetry}
[b_{\lambda} a]=-(-1)^{p(a)p(b)}[a_{-\partial-\lambda}b], 
\end{equation}
\item[(iii)] Jacobi identity:
\begin{equation}\label{Jacobi identity}
[a_{\lambda}[b_{\mu}c]]=(-1)^{p(a)p(b)}[b_{\mu}[a_{\lambda}c]]+[[a_{\lambda}b]_{\lambda+\mu}c],
\end{equation} 
\end{itemize}
\end{definition}

\begin{definition}
A vertex algebra is a quintuple $(V,\left|0\right>,\partial,[._{\lambda}.],::)$ which satisfies the following three properties:
\item[(i)] $(V,\partial,[._{\lambda}.])$ is a Lie conformal algebra,
\item[(ii)] $(V,\left|0\right>,\partial,::)$ is a unital differential algebra satisfying
\begin{itemize}
\item quasi-commutativity:
\begin{align}\label{quasi-commutativity}
:ab:-(-1)^{p(a)p(b)}:ba:=\int_{-\partial}^{0}[a_{\lambda}b]d\lambda
\end{align}
for any $a,b\in V,$
\item quasi-associativity:
\small
\begin{align}\label{quasi-associativity}
::ab:c:-:a:bc::=:\left(\int_{0}^{\partial}d\lambda a\right)[b_{\lambda}c]:+(-1)^{p(a)p(b)}:\left(\int_{0}^{\partial}d\lambda b\right)[a_{\lambda}c]:
\end{align}
\normalsize
for any $a,b,c\in V,$
\end{itemize}
The integrals in the expression above should be interpreted in the following manner. First, expand the $\lambda$-bracket. Second, put the powers of $\lambda$ on the left, under the sign of integral. Finally, take the definite integral by the usual rules inside the parenthesis.\
The binary operation $::$ is called the normally ordered product of $V$.
\item[(iii)] The $\lambda$-bracket and the normally ordered product are related by the non-commutative Wick formula:

\begin{equation}\label{Wickformula}
[a_{\lambda}:bc:]=:[a_{\lambda}b]c:+(-1)^{p(a)p(b)}:b[a_{\lambda}c]:+\int_{0}^{\lambda}[[a_{\lambda}b]_{\mu}c]d\mu,
\end{equation}

for any $a,b,c\in V.$

\end{definition}

In this paper we are concerned with a special class of vertex superalgebras, known as conformal vertex algebras. We say that a vertex algebra $V$ is conformal if there is a $L\in V$ such that the Fourier modes of the corresponding field $Y(L,z)=\sum_{n\in\mathbb{Z}}L_{n}z^{-n-2}$ satisfy:
\begin{itemize}
\item[(i)] the operators $L_{n}$ form the Virasoro algebra
\begin{equation}\label{Virasoroalgebra}
[L_{n},L_{m}]=(m-n)L_{m+n}+\frac{c}{12}(m^{3}-m)\delta_{m,-n}\mathrm{Id}_{V}
\end{equation}
where $c\in\mathbb{C}$ is called the central charge,
\item[(ii)] $L_{-1}=T$,
\item[(iii)] $L_{0}$ is a diagonalizable operator with eigenvalues bounded below.
\end{itemize}

\begin{remark}
Using  (\ref{def-lambda bracket and normally ordered product}) note that (\ref{Virasoroalgebra}) is equivalent to:
\begin{equation}\label{lambda bracket of the Virasoro element}
[L_{\lambda}L]=(\partial+2\lambda)L+\frac{c\lambda^{3}}{12}.
\end{equation}
A field $L(z)$ satisfying this is called a Virasoro field. Note also that the condition (ii) above implies that $L(z)$ is an even field.
\end{remark}

If $V$ is conformal vertex algebra and $a\in V$ is an eigenvector of $L_{0}$, its eigenvalue  is called the conformal weight of $a$ and is denoted by $\Delta(a)=\Delta$. Moreover $a$ has conformal weight $\Delta$ if and only if 
\begin{equation*}
[L_{\lambda}a]=(\partial+\Delta\lambda)a+\mathcal{O}(\lambda^{2}).
\end{equation*}  
\noindent
In the case when $[L_{\lambda}a]=(\partial+\Delta\lambda)a$ the vector $a$ and the corresponding vertex operator $Y(a,z)$ are called primary. This is equivalent to $L_{n}a=\delta_{n,0}\Delta a$ for all $n\geq 0$.

\begin{remark}\label{remark-universal enveloping algebra}
An important feature of defining vertex algebras using Lie conformal algebras is that: vertex algebras are to conformal Lie algebras what associative algebras are to Lie
algebras in the following sense. For any Lie conformal algebra $R$ there exists a vertex
algebra $U(R)$ with an embedding of conformal algebras $\pi: R\hookrightarrow U(R)$ satisfying the usual universal property: for any other vertex algebra $V$ and a map $f: R \hookrightarrow V$, there exists a morphism of vertex algebras $g:U(R)\hookrightarrow V$ such that $f=g\circ \pi$. Moreover, the algebra $U(R)$ is constructed very similar as in the Lie algebra situation, any vector of $U(R)$ can be obtained by products of elements of $R$.

This parallel with the Lie algebra case justifies the presentation of the examples below. When we say that a certain set $A$ of vectors satisfying some prescribed $\lambda$-brackets generate a vertex algebra, we first construct the corresponding Lie conformal algebra and then we consider its universal enveloping vertex algebra, in particular any vector of this algebra is a combination of products of elements of $A$ and their derivatives.  There
are other situations (see subsections \ref{Shatashvili-Vafa algebra} and \ref{Shatashvili-Vafa Spin(7) algebra} ) when the $\lambda$-bracket of elements in $A$ is not linear in the elements
of $A$ but can be expressed as combinations of products of elements of $A$ and their
derivatives. In this case we say that the vertex algebra is non-linearly generated \cite{DeSoleKac05}.
\end{remark}

\begin{remark}\label{remark-OPE}
Note that the $\lambda$-bracket notation encodes the same information as the singular part of the operator product expansion (OPE) of fields in a quantum field theory in dimension two. In other words when we write $[{A}_{\lambda}B]=\sum_{j=0}^{N-1}\frac{\lambda^{j}}{j!}A_{(j)}B$, in physics notation this means that: 
\begin{align*}
A(z)B(w)\sim \sum_{j=0}^{N-1}\frac{(A_{(j)}B)(w)}{(z-w)^{j+1}},
\end{align*}
where $A(z)$, $B(w)$ and $(A_{(j)}B)(w)$ are the fields corresponding to the vectors $A$, $B$ and $A_{(j)}B$ under the state-field correspondence.
\end{remark}

\begin{example}[\cite{Kac96}]\label{Neveu-Schwarz}The $N=1$ (Neveu-Schwarz) superconformal vertex algebra.\\
The $N=1$ superconformal vertex algebra of central charge $c$ is an extension of the Virasoro algebra of central charge $c$ by an odd primary field $G$ of conformal weight $3/2$ i.e  it is generated by $L$ and $G$ with $\lambda$-brackets:
\begin{gather*}
[L_{\lambda}L]=(\partial+2\lambda)L+\frac{c}{12}\lambda^{3}, \quad [L_{\lambda}G]=(\partial+\frac{3}{2}\lambda)G, \quad [G_{\lambda}G]=2L+\frac{c}{3}\lambda^{2}.
\end{gather*}
\end{example}

\begin{example}[\cite{Kac96}]\label{N=2 superconformal algebra}The $N=2$ superconformal vertex algebra.\\
The $N=2$ superconformal vertex algebra of central charge $c$ is generated by the Virasoro field $L$ with $\lambda$-bracket (\ref{lambda bracket of the Virasoro element}), an even primary field $J$ of conformal weight $1$, and two odd primary fields $G^{\pm}$ of conformal weight $\tfrac{3}{2}$, with the $\lambda$-brackets given by:
\begin{gather*}
[{L}_{\lambda}J]=\left(\partial+\lambda\right)J,\quad [{L}_{\lambda}G^{\pm}]=\left(\partial+\frac{3}{2}\lambda\right)G^{\pm},\\
[{J}_{\lambda}G^{\pm}]=\pm G^{\pm},\quad [{J}_{\lambda}J]=\frac{c}{3}\lambda,\quad [{G^{\pm}}_{\lambda}G^{\pm}]=0,\\
[{G^{+}}_{\lambda}G^{-}]=L+\frac{1}{2}\partial J+J\lambda +\frac{c}{6}\lambda^{2}.
\end{gather*}
\end{example}

\begin{example}[\cite{KacWakimoto04}] The $N=4$ superconformal vertex algebra.\\
The $N=4$ superconformal vertex algebra is generated by a Virasoro field $L$, three primary even fields $J^{0}, J^{+}$ and $J^{-}$  of conformal weight 1, and four primary odd fields $G^{\pm}, \bar{G^{\pm}}$ of conformal weight $\frac{3}{2}$. The remaining non-zero $\lambda$-brackets are:
\begin{gather*}
[{J^{0}}_{\lambda}J^{\pm}]=\pm 2J^{\pm},\quad [{J^{0}}_{\lambda}J^{0}]=\frac{c}{3}\lambda, \quad [{J^{+}}_{\lambda}J^{-}]=J^{0}+\frac{c}{6}\lambda,\\
[{J^{0}}_{\lambda}G^{\pm}]=\pm G^{\pm},\quad [{J^{0}}_{\lambda}\bar{G}^{\pm}]=\pm\bar{G}^{\pm},\quad [{J^{+}}_{\lambda}G^{-}]=G^{+},\\
[{J^{-}}_{\lambda}G^{+}]=G^{-},\quad [{J^{+}}_{\lambda}\bar{G}^{-}]=-\bar{G}^{+},\quad [{J^{-}}_{\lambda}\bar{G}^{+}]=-\bar{G}^{-},\\
[{G^{\pm}}_{\lambda}\bar{G}^{\pm}]=\left(\partial+2\lambda\right)J^{\pm},\quad [{G^{\pm}}_{\lambda}\bar{G}^{\mp}]=L\pm\frac{1}{2}\partial J^{0}\pm J^{0}\lambda+\frac{c}{6}\lambda^{2}.
\end{gather*} 

\end{example}

\begin{example}\label{bc-betagamma system} The $bc-\beta\gamma$ system.\\
This vertex algebra is generated by four fields: $b$ and $c$ are odd fields, $\beta$ and $\gamma$ are even fields, and the non-trivial $\lambda$-brackets between the generators are:
\begin{gather*}
[{\beta}_{\lambda}\gamma]=1, \qquad [{b}_{\lambda}c]=1.
\end{gather*}
If we define $G=:c\beta:+:(\partial\gamma) b:$, $L=:(\partial\gamma)\beta:-\frac{1}{2}:c(\partial b):+\tfrac{1}{2}:(\partial c)b:$, $\{L,G\}$ generate an $N=1$ superconformal algebra of central charge $\frac{3}{2}$ (example \ref{Neveu-Schwarz}).  With respect to the Virasoro $L$ the conformal weight of the generators is as follows: 
\begin{equation*}
\Delta b=\frac{1}{2},\;\;\Delta c=\frac{1}{2},\;\;\Delta \gamma=0,\;\;\Delta\beta=1.
\end{equation*}
\end{example}

\subsection{The Shatashvili-Vafa \texorpdfstring{$G_{2}$}{} superconformal algebra} \label{Shatashvili-Vafa algebra}

The Shatashvili-Vafa $G_{2}$ superconformal algebra \cite{Shatashvili-Vafa95} is an extension of the $N=1$ superconformal algebra $\{L,G\}$ (Example \ref{Neveu-Schwarz}) by four fields $\{\Phi,K,X,M\}$ such that: $\Phi$ and $M$ are odd, $K$ and $X$ are even, $\Delta(\Phi)=\frac{3}{2}$, $\Delta(K)=2$, $\Delta(X)=2$ and $\Delta(M)=\frac{5}{2}$. Furthermore $\Phi$ and $K$ are primary fields. The $\lambda$-brackets are given by:
\begin{equation}
[{\Phi}_{\lambda}\Phi]=(-\frac{7}{2}) \lambda^{2}+6X,\;\;\; [{\Phi}_{\lambda}X]=-\frac{15}{2}\Phi\lambda-\frac{5}{2}\partial\Phi,\nonumber
\end{equation}
\begin{equation}
[{X}_{\lambda}X]=\frac{35}{24}\lambda^{3}-10X\lambda-5\partial X,\;\;\; [{G}_{\lambda}\Phi]=K,\nonumber
\end{equation}
\begin{equation}
[{G}_{\lambda}X]=-\frac{1}{2}G\lambda+M,\;\;\;[{G}_{\lambda}K]=3\Phi\lambda+\partial\Phi,\nonumber
\end{equation}
\begin{equation}
[{G}_{\lambda}M]=-\frac{7}{12}\lambda^{3}+\left(L+4X\right)\lambda+\partial X,\;\;\; [{\Phi}_{\lambda}K]=-3G\lambda-3\left(M+\frac{1}{2}\partial G\right),\nonumber
\end{equation}
\begin{equation}
[{\Phi}_{\lambda}M]=\frac{9}{2}K\lambda-\left(3:G\Phi:-\frac{5}{2}\partial K\right),\;\;\;[{X}_{\lambda}K]=-3K\lambda+3\left(:G\Phi:-\partial K\right),\nonumber
\end{equation}
\begin{equation}
[{X}_{\lambda}M]=-\frac{9}{4}G\lambda^{2}-\left(5M+\frac{9}{4}\partial G\right)\lambda+\left(4:GX:-\frac{7}{2}\partial M-\frac{3}{4}\partial^{2}G\right),\nonumber
\end{equation}
\begin{equation}
[{K}_{\lambda}K]=-\frac{21}{6}\lambda^{3}+6\left(X-L\right)\lambda+3\partial\left(X-L\right),\nonumber
\end{equation}
\begin{equation}
[{K}_{\lambda}M]=-\frac{15}{2}\Phi\lambda^{2}-\frac{11}{2}\partial\Phi\lambda+3\left(:GK:-2:L\Phi:\right),\nonumber
\end{equation}
\begin{eqnarray}
[{M}_{\lambda}M]&=&-\frac{35}{24}\lambda^{4}+\frac{1}{2}\left(20X-9L\right)\lambda^{2}+\left(10\partial X-\frac{9}{2}\partial L\right)\lambda+\left(\frac{3}{2}\partial^{2}X\right.\nonumber\\
&&\left.-\frac{3}{2}\partial^{2}L-4:GM:+8:LX:\right),\nonumber
\end{eqnarray}
\begin{equation}
[{L}_{\lambda}X]=-\frac{7}{24}\lambda^{3}+2X\lambda+\partial X,\;\;\; [{L}_{\lambda}M]=-\frac{1}{4}G\lambda^{2}+\frac{5}{2}M\lambda+\partial M. \nonumber
\end{equation}
\noindent
And the generators satisfy the following relation:
\begin{align}\label{ideal in the definition of the algebra}
0=4:GX:-2:\Phi K:-4\partial M-\partial^{2}G.
\end{align}

\begin{remark}
Unlike the previous examples the right hand side of some $\lambda$-brackets is non-linear in the generators. This is an important feature of the algebra.
\end{remark}
This superconformal algebra appeared as the chiral algebra associated to the sigma model with target a manifold with $G_{2}$ holonomy in \cite{Shatashvili-Vafa95} Its classical counterpart had been studied by Howe and Papadopoulos in \cite{Howe-Papadopoulos93}. In fact this algebra is a member of a two-parameter family $SW(\frac{3}{2}, \frac{3}{2}, 2)$ previously studied in \cite{Blumenhagen92} where the author found the family of all superconformal algebras which are extension of the super-Virasoro algebra, i.e., the $N=1$ superconformal algebra, by two primary supercurrents of conformal weights $\frac{3}{2}$ and $2$ respectively. It is a family parametrized by $(c,\varepsilon)$ ($c$ is the central charge and $\varepsilon$ the coupling constant) of non-linear $W$-algebras.

The Shatashvili-Vafa $G_{2}$ algebra is a quotient of $SW(\frac{3}{2},\frac{3}{2}, 2)$ with $c=\frac{21}{2}$ and $\varepsilon=0$, the relation (\ref{ideal in the definition of the algebra}) is precisely the one imposed by the quotient.

In \cite{Shatashvili-Vafa95}, this algebra was obtained as a free field realization in terms of seven free Bosons and seven free Fermions, and the relation (\ref{ideal in the definition of the algebra}) was trivially satisfied. The first to note that we should impose this relation if we define the algebra abstractly (it is necessary for the Jacobi identities to be checked) was Figueroa-O'Farrill \cite{OFarrill97}.

The Shatashvili-Vafa $G_{2}$ superconformal algebra can be obtained also as a quantum Hamiltonian reduction of $\operatorname{osp}\left(4|2\right)$ \cite{HeluaniRodriguez15}.

\begin{remark}\label{Tri-critical Ising model}
Note that if we define $\tilde{X}:=-\frac{1}{5}X$ and $\tilde{\Phi}:=\frac{i}{\sqrt{15}}$, then $\{\tilde{X},\tilde{\Phi}\}$ generate an $N=1$ superconformal algebra of central charge $c=\frac{7}{10}$. This $N=1$ superconformal algebra at this value of the central charge is known as the tri-critical Ising model. Therefore the Shatashvili-Vafa algebra contains two $N=1$ superconformal subalgebras, the original one generated by $\{L,G\}$ and the subalgebra generated by $\{X,\Phi\}$.
\end{remark}

\begin{remark}\label{remark-the SV algebra is generated in conf weight 3/2}
Note that all the generators of the Shatashvili-Vafa $G_{2}$ algebra can be obtained as the $\lambda$-brackets of the generators in conformal weight $\frac{3}{2}$ i.e., $\Phi$ and $G$.
\end{remark}
The Shatashvili-Vafa $G_{2}$ superconformal algebra has an automorphism given by $L\rightarrow L$, $G\rightarrow G$, $\Phi\rightarrow -\Phi$, $K\rightarrow -K$, $X\rightarrow X$, $M\rightarrow M$. It is interesting to note that the fixed vectors $\{L, G, X, M\}$ generate a closed subalgebra, this algebra is a member of a one parameter family  $SW(\tfrac{3}{2},2)$ of superconformal algebras, this one corresponds to the parameter $\tfrac{21}{2}$. Another member of this family is the next example.

\subsection{The Shatashvili-Vafa \texorpdfstring{$\mathrm{Spin}(7)$}{} superconformal algebra} \label{Shatashvili-Vafa Spin(7) algebra}
The Shatashvili-Vafa $\mathrm{Spin}(7)$ superconformal algebra \cite{Shatashvili-Vafa95} is an extension of the $N=1$ superconformal algebra $\{L,G\}$ (Example \ref{Neveu-Schwarz}) by an even field $\bar{X}$ of conformal weight $2$ (not primary) and an odd field $\bar{M}$ of conformal weight $\frac{5}{2}$  (not primary). The $\lambda$-bracket are given by:

\begin{equation*}
[{\bar{X}}_{\lambda}\bar{X}]=\frac{8}{3}\lambda^{3}+16\bar{X}\lambda+8\partial \bar{X}, \quad [{L}_{\lambda}\bar{X}]=\frac{1}{3}\lambda^{3}+\left(\partial+2\lambda\right)\bar{X},
\end{equation*}
\begin{equation*}
[{G}_{\lambda}\bar{X}]=\frac{1}{2}G\lambda+\bar{M}, \quad [{G}_{\lambda}\bar{M}]=\frac{2}{3}\lambda^{3}-\left(L-4\bar{X}\right)\lambda+\partial\bar{X},
\end{equation*}
\begin{equation*}
[{\bar{X}}_{\lambda}\bar{M}]=-\frac{15}{4}G\lambda^{2}-\left(\frac{15}{4}\partial G-8\bar{M}\right)\lambda+\frac{11}{2}\partial\bar{M}-\frac{5}{4}\partial^{2}G-6:G\bar{X}:,
\end{equation*}
\begin{align*}
[{\bar{M}}_{\lambda}\bar{M}]=&-\frac{8}{3}\lambda^{4}-\left(\frac{15}{2}L+16\bar{X}\right)\lambda^{2}-\left(\frac{15}{2}\partial L+16\partial\bar{X}\right)\lambda\\
&-\left(\frac{5}{2}\partial^{2}\bar{X}+\frac{5}{2}\partial^{2}L+12:L\bar{X}:-6:G\bar{M}:\right).
\end{align*}
Similar to the example above this algebra appeared as the chiral algebra associated to the sigma model with target a manifold with $\mathrm{Spin}(7)$ holonomy in \cite{Shatashvili-Vafa95} and its classical counterpart had been studied by Howe and Papadopoulos in \cite{Howe-Papadopoulos93}. This algebra also belongs to a family $SW(\frac{3}{2},2)$ of superconformal algebras with one parameter $c$ (the central charge), the Shatashvili-Vafa $\mathrm{Spin}(7)$ algebra corresponds to $c=12$, see (\cite{OFarrillSchrans91}, \cite{OFarrillSchrans92}, \cite{OFarrill97}).

\section{The Chiral de Rham Complex}\label{sec-the Chiral de Rham Complex}

In this section we review the construction of the chiral de Rham complex \cite{Malikov-Schechtman-Vaintrob99} of a manifold $M$ and remember some of its properties. We also recall a theorem from \cite{Ekstrand-Heluani-Kallen-Zabzine13} that provides two different embeddings of the space of differential forms $\Omega^{*}(M)$ into global sections of CDR in the case when $M$ is a smooth Riemannian manifold. The chiral de Rham complex can be defined in the algebraic, analytic and smooth category, in this paper we are only concerned with the smooth setting.

Let $M$ be a $C^{\infty}$ manifold, the CDR of $M$ is a sheaf of vertex algebras, let us begin with a local description. On a coordinate chart of $M$ with coordinates $\{x^{i}\}$ the sections of CDR are described as follows: for each coordinate $x^{i}$ we have an even field $\gamma^{i}$, corresponding to each  vector field $\frac{\partial}{\partial x^{i}}$ and each  differential form $dx^{i}$ we have an odd fields $b_{i}$ and $c^{i}$ respectively, we also have an even field $\beta_{i}$ for each $\gamma^{i}$, and $\{\gamma^{i},\beta_{i},b_{i},c^{i}\}$ form a $bc-\beta\gamma$ system (see example \ref{bc-betagamma system}).

Now we need to take care of what happens if we change coordinates. A crucial observation in \cite{Malikov-Schechtman-Vaintrob99} is that for any change of coordinates $x^{i}\rightarrow y^{i}(x^{j})$ there exists an automorphism of CDR on the intersection of the coordinate charts. Using this it is possible to glue on intersections and construct a global sheaf. More precisely if $\tilde{\gamma}^{i}$, $\tilde{\beta}_{i}$, $\tilde{c}^{i}$ and $\tilde{b}_{i}$ are the fields associated to the coordinates $y^{i}$, they are expressed in terms of the fields in the coordinates $x^{i}$ as follows: 
\begin{align}\label{change of coordinates bc-betagamma}
\tilde{\gamma}^{i}=&y^{i}(\gamma),\nonumber\\
\tilde{c}^{i}=&\frac{\partial y^{i}}{\partial x^{j}}(\gamma)c^{j},\nonumber\\
\tilde{b}_{i}=&\frac{\partial x^{j}}{\partial y^{i}}(y(\gamma))b_{j},\\
\tilde{\beta}_{i}=&\beta_{j}\frac{\partial x^{j}}{\partial y^{i}}(y(\gamma))+\frac{\partial^{2}x^{k}}{\partial y^{i}\partial y^{l}}(y(\gamma))\frac{\partial y^{l}}{\partial x^{r}}c^{r}b_{k}.\nonumber
\end{align}
We see that the $\gamma^{i}$ transform as coordinates do, the $b_{i}$ transform as vector fields, the $c^{i}$ change as differential forms, however the $\beta_{i}$ change in a non-tensorial manner. In the remaining of this section we collect some results about the existence of global sections of CDR. 

\begin{remark}
The first thing to note is that the multiplication (normally ordered product) in CDR is neither associative nor commutative by the own nature of CDR: CDR is a sheaf of vertex algebras. Even though we can define a multiplication map $\mathcal{O}_{M}\times CDR\rightarrow CDR$, it is not associative. This implies that it is very difficult to construct global sections of CDR.\\
From (\ref{change of coordinates bc-betagamma}) follows that functions and vector fields of $M$ give rise to global sections. However, trying to construct sections of CDR from other tensors on $M$ is not trivial because of the terms on the RHS of the quasi-associativity
(\ref{quasi-associativity}) appearing under a change of coordinates.\\
In \cite{Ben-Zvi-Heluani-Szczesny} it was noticed that one can use the Levi-Civita connection on M to counteract these quasi-associativity terms in order to construct sections of CDR associated to differential two-forms. In \cite{Ekstrand-Heluani-Kallen-Zabzine13} this was generalized to higher order forms.
\end{remark}

We can define in a local chart the following field $G=(\partial\gamma^{i})b_{i}+\beta_{i}c^{i}$ (see example \ref{bc-betagamma system}) and ask if this field is globally defined. The answer \cite{Malikov-Schechtman-Vaintrob99}  is the following, if we change coordinates  $x^{i}\rightarrow y^{i}(x^{j})$, and let $\tilde{G}$ be the field in the coordinates $y^{i}$ we obtain:
\begin{equation*}
\tilde{G}=G+\partial\left(\frac{\partial}{\partial y^{r}}\left[Tr\log\left(\frac{\partial x^{i}}{\partial y^{j}}\right)\right]\tilde{c}^{r}\right).
\end{equation*}
Note that if $M$ is orientable $G$ is globally defined. In general the field $G$ is not globally defined however the Fourier modes $G_{(0)}$ is globally defined because $\tilde{G}$ differs from $G$ by a derivative of a field. Moreover we have $(G_{(0)})^{2}=\partial$. This endomorphism $G_{(0)}$ is called the supersymmetric generator, and it is an odd derivation of all the $n$-products i.e., $G_{(0)}({a}_{(n)}b)={(G_{(0)}a)}_{(n)}b+(-1)^{p(a)}{a}_{(n)}{(G_{(0)}b)}$, where $a,b$ are two sections of the CDR.

In \cite{Ekstrand-Heluani-Kallen-Zabzine13} were produced two different embeddings of the space of differential forms $\Omega^{*}(M)$ into global sections of CDR in the case when $M$ is a smooth Riemannian manifold. These embeddings depend on a choice of a metric and are explicitly given in terms of the corresponding Levi-Civita connection. We proceed to recall their construction:

Let $g$ be a Riemannian metric on $M$ and let $\{x^{i}\}$ be a local coordinate system. Define local sections $e^{i}_{\pm}$ by the equations:

\begin{equation*}
e^{i}_{+}:=\dfrac{g^{ij}b_{j}+c^{i}}{\sqrt{2}},\;\;\;\;  i e^{i}_{-}:=\dfrac{i\left(g^{ij}b_{j}-c^{i}\right)}{\sqrt{2}}.
\end{equation*}

Note that

\begin{equation}\label{commutation rules of the e's}
[{e^{i}_{+}}_{\lambda}e^{j}_{+}]=g^{ij},\;\;[{e^{i}_{-}}_{\lambda}e^{j}_{-}]=-g^{ij},\;\;[{e^{i}_{+}}_{\lambda}e^{j}_{-}]=0.
\end{equation}

Define 

\begin{equation*}
F_{\pm (0)}:=1,\;\;\;\; F_{+(k)}^{i_{1}\dots i_{k}}:=e^{i_{1}}_{+}\dots e^{i_{k}}_{+},\;\;\;\; F_{-(k)}^{i_{1}\dots i_{k}}:=i^{k}e^{i_{1}}_{-}\dots e^{i_{k}}_{-}. 
\end{equation*}

(Here and further we adopt the following convention: when a product of fields appears without any colon or parenthesis, we read the normal product from right to left, recalling that it is not associative.)

Define $G_{\pm (n,n)}=F_{\pm(n)}$ and for each $1\leq s\leq \lfloor \frac{n}{2} \rfloor$ we define

\begin{equation*}
G^{i_{1}\dots i_{n}}_{\pm(n,n-2s)}=\Gamma^{i_{1}}_{k_{1}l_{1}}g^{i_{2}k_{1}}\partial \gamma^{l_{1}}\dots\Gamma^{i_{2s-1}}_{k_{2s-1}l_{2s-1}}g^{i_{2s}k_{2s-1}}\partial\gamma^{l_{2s-1}}F^{i_{2s+1}\dots i_{n}}_{\pm(n-2s)},
\end{equation*} 
\noindent
where $\Gamma^{i}_{jk}$ are the Christoffel symbols of the Levi-Civita connection associated to $g$, $\lfloor \cdot \rfloor$ denotes the integer part, the subscript between parenthesis denotes how many $e$'s are present in the expression.

Define the numbers $T_{r,s}$ as the coefficients of the Bessel polynomials:
\begin{equation*}
y_{r}(x)=\sum^{r}_{s=0}T_{r,s}x^{s}=\sum^{r}_{s=0}\dfrac{(r+s)!}{(r-s)!s!2^{s}}x^{s},
\end{equation*}
\noindent
and let $T_{r,s}:=0$ when $s<0$ or $s>t$. The following theorem was proved in \cite{Ekstrand-Heluani-Kallen-Zabzine13}.
\begin{theorem}\cite[Theorem 6.1]{Ekstrand-Heluani-Kallen-Zabzine13}\label{theorem Reimundo}
Let $(M, g)$ be a Riemannian manifold. For any differential form $w\in\Omega^{*}(M)$ locally described by $w_{i_{1}}\dots i_{n}dx^{i_{1}}\wedge\dots\wedge dx^{i_{n}}$ define
\begin{equation*}
J_{\pm q}:=\frac{1}{n!}w_{i_{1}\dots i_{n}}E^{i_{1}\dots i_{n}}_{\pm(n)},
\end{equation*}
\begin{equation*}
E_{\pm(n)}:=\sum_{s=0}^{[\frac{n}{2}]}T_{n-s,s}G_{\pm(n,n-2s)}.
\end{equation*}
\noindent
Then $J_{\pm q}$ are well defined sections of CDR.
\end{theorem}
\begin{corollary}
$G_{(0)}\left(J_{\pm q}\right)$ are also well defined sections of CDR.
\end{corollary}
This follows from the discussion above about the supersymmetric generator $G_{(0)}$. In fact this corollary was implicit in the original form of the Theorem \ref{theorem Reimundo} in \cite{Ekstrand-Heluani-Kallen-Zabzine13} where the authors consider the CDR as a supersymmetric sheaf of vertex superalgebras.

As an exemplification of the Theorem \ref{theorem Reimundo} and because they will be used in section \ref{CDR of a G2 manifold} we write the sections of CDR that correspond to $2$, $3$ and $4$ forms:

if $w$ is a $2$-form then
\begin{align}\label{correction2form}
J_{\pm q}=\pm\frac{1}{2}w_{ij}e^{i}_{\pm}e^{j}_{\pm}+\frac{1}{2}\Gamma^{i}_{jk}g^{jl}w_{il}\partial\gamma^{k},
\end{align}
if $w$ is a $3$-form then
\begin{align}\label{correction3form}
J_{+q}=\frac{1}{6}w_{ijk}e^{i}_{+}e^{j}_{+}e^{k}_{+}+\frac{1}{2}w_{ijk}\Gamma_{mn}^{i}g^{jm}\partial \gamma^{n}e^{k}_{+},
\end{align} 
\begin{align*}
J_{-q}=\frac{(-i)}{6}w_{ijk}e^{i}_{-}e^{j}_{-}e^{k}_{-}+\frac{i}{2}w_{ijk}\Gamma_{mn}^{i}g^{jm}\partial \gamma^{n}e^{k}_{-},
\end{align*} 
if $w$ is a $4$-form then 
\begin{align}\label{correction4form}
J_{+q}=&\frac{1}{24}w_{ijkl}e^{i}_{+}e^{j}_{+}e^{k}_{+}e^{l}_{+} + \frac{1}{4}w_{ijkl}\Gamma_{mn}^{i}g^{jm}\partial\gamma^{n}e^{k}_{+}e^{l}_{+}\nonumber\\
&+\frac{1}{8}w_{ijkl}\Gamma_{m_{1}n_{1}}^{i}g^{jm_{1}}\partial\gamma^{n_{1}}\Gamma_{m_{2}n_{2}}^{k}g^{lm_{2}}\partial \gamma^{n_{2}},
\end{align}
\begin{align*}
J_{-q}=&\frac{1}{24}w_{ijkl}e^{i}_{-}e^{j}_{-}e^{k}_{-}e^{l}_{-} + \frac{(-1)}{4}w_{ijkl}\Gamma_{mn}^{i}g^{jm}\partial\gamma^{n}e^{k}_{-}e^{l}_{-}\nonumber\\
&+\frac{1}{8}w_{ijkl}\Gamma_{m_{1}n_{1}}^{i}g^{jm_{1}}\partial\gamma^{n_{1}}\Gamma_{m_{2}n_{2}}^{k}g^{lm_{2}}\partial \gamma^{n_{2}}.
\end{align*}

\begin{theorem}\cite[section 7.1 and 7.2]{Ekstrand-Heluani-Kallen-Zabzine13}
Let $M$ be a Calabi-Yau manifold, $\omega$ the K\"ahler form, let $J_{\pm q}$ be the global sections given by Theorem~\ref{theorem Reimundo} corresponding to $\omega$, then the two sets of sections of CDR $\{J_{\pm q}, G_{(0)}(J_{\pm q})\}$ generate two commuting copies of the $N=2$ superconformal algebra of central charge $\tfrac{3}{2} \dim M$. Furthermore if we assume that $M$ is a Calabi-Yau threefold, let $\Omega$ and $\bar\Omega$ denote the holomorphic volume form and its complex conjugate respectively. Let $X_{\pm}$ and $\bar X_{\pm}$ be the global sections given by Theorem~\ref{theorem Reimundo} corresponding to $\Omega$ and $\bar\Omega$ respectively, then the two sets of sections of CDR $\{J_{\pm q}, G_{(0)}(J_{\pm q}),X_{\pm}, G_{(0)}(X_{\pm}), \bar X_{\pm}, G_{(0)}(\bar X_{\pm})\}$ generate two commuting copies of the Odake algebra \cite{Odake89}.
\end{theorem}
 
\section{\texorpdfstring{$G_{2}$}{} holonomy manifolds}\label{sec-G2 holonomy manifolds}

In this section we review some properties of the manifolds with $G_{2}$ holonomy that will be used extensively in the computations. The references for this section are \cites{Joyce07,salamon89}.

The Lie group $G_{2}$ can be defined as the group of linear automorphism of $\mathbb{R}^{7}$ that preserves the cross product, where we identify $\mathbb{R}^{7}=Im(\mathbb{O})$, and the cross product operation is induced from the octonion multiplication $u\times v=im(\bar{v}.u)$, $u,v\in \mathbb{O}$. The cross product operation is encoded by a 3-form $\varphi_{0}\in\wedge^{3}(\mathbb{R}^{7})^{*}$ defined as 
\begin{align*}
\varphi_{0}(u,v,w)=<u\times v,w>.
\end{align*}

If we choose coordinates $\left(x_{1},\dots,x_{7}\right)$ on $\mathbb{R}^{7}$ and denote by $dx_{ij\dots k}$ the exterior form $dx_{i}\wedge dx_{j}\wedge\dots\wedge dx_{k}$ on $\mathbb{R}^{7}$,  $\varphi_{0}$ can be written as:

\begin{align*}
\varphi_{0}=dx_{123}+dx_{145}+dx_{167}+dx_{246}-dx_{257}-dx_{347}-dx_{356}.
\end{align*}
In the other way, the identity
\begin{equation}\label{innerproduct}
6<u,v>dvol_{7}=(u\lrcorner \varphi_{0})\wedge(v\lrcorner \varphi_{0})\wedge\varphi_{0},
\end{equation}
\noindent
expresses the inner product of $\mathbb{R}^{7}$ in terms of the volume form of $\mathbb{R}^{7}$ and the $3$-form $\varphi_{0}$. Observe that the Lie algebra $\rm{g}_{2}$ of the Lie group $G_{2}$ can be described as:
\begin{align}\label{definition of g2}
\rm{g}_{2}=\{(a_{ij})\in \mathfrak{so}(7) | a_{ij}{\varphi_{0}}_{ijk}=0 \;\;\forall\; k\}.
\end{align}

\begin{definition}
Let $M$ be a smooth $7$-dimensional manifold, and $F$ the frame bundle of $M$. We say that $M$ has a $G_{2}$-structure if there is a principal subbundle $P$ of $F$ with fibre $G_{2}$. 
\end{definition}

Equivalently $M$ has a $G_{2}$-structure if there exists $\varphi\in\Omega^{3}(M)$ that is pointwise modelled in $\varphi_{0}$, i.e., for each $x\in M$, $\varphi_{0}$ is the image of $\varphi(x)$  under a linear oriented isomorphism $T_{x}M\rightarrow \mathbb{R}^{7}$. Thus, we have a 1-1 correspondence between 3-forms $\varphi$ pointwise modelled in $\varphi_{0}$ and $G_{2}$-structures $P$ on $M$.

Given a $G_{2}$-structure in $M$ this determines in particular an orientation and a metric $g$ in $M$ because any $G_{2}$-structure is an instance of an $SO(7)$-structure. Furthermore, let $\varphi$ be the 3-form that corresponds to this $G_{2}$-structure. Then, the identity (\ref{innerproduct}) determines the metric in terms of $\varphi$ up to a conformal factor. Let $Hol(g)$ denote the holonomy group of $M$, then $Hol(g)\subseteq G_{2}$ if and only if $\varphi$ is parallel with respect to the Levi-Civita connection of $g$, i.e., $\nabla\varphi=0$.

\begin{lemma}\cite{salamon89}
The holonomy group of the Riemannian metric $g$ induced by $\varphi$ is contained in $G_{2}$ if and only if $d\varphi=0$ and $d*\varphi=0$, where $*$ denotes the Hodge star for $g$.
\end{lemma}

We say that a manifold with a $G_{2}$-structure is a $G_{2}$-manifold if $Hol(g)\subseteq G_{2}$.

If we write $\varphi$ and $*\varphi$ in local coordinates: $\varphi=\varphi_{ijk}dx_{ijk}$, $*\varphi=\psi_{ijkl}dx_{ijkl}$ then $d\varphi=0$ and $d*\varphi=0$ are equivalent to: 
\begin{equation}\label{Phicovariantlyconstant}
0=\frac{\partial}{\partial x^{m}}\varphi_{ijk}-\Gamma_{mi}^{s}\varphi_{sjk}-\Gamma_{mj}^{s}\varphi_{isk}-\Gamma_{mk}^{s}\varphi_{ijs},
\end{equation}
\begin{equation}\label{Psicovariantlyconstant}
0=\frac{\partial}{\partial x^{m}}\psi_{ijkl}-\Gamma_{mi}^{s}\psi_{sjkl}-\Gamma_{mj}^{s}\psi_{iskl}-\Gamma_{mk}^{s}\psi_{ijsl}-\Gamma_{ml}^{s}\psi_{ijks},
\end{equation}
for all $m=1,\dots ,7$.

\begin{lemma}\cite{salamon89}\label{lemmaG2Ricciflat}
Let $(M,g)$ be a Riemannian 7-manifold. If $Hol(g)\subseteq G_{2}$, then $g$ is Ricci-flat.
\end{lemma}

Let $M$ be a $7$-dimensional manifold with $G_{2}$--structure given by the $3$-form $\varphi$, denote by $g$ and $\psi=*\varphi$ the associated metric and $4$-form respectively. Then the following identities are satisfied, they are collected from \cite{Karigiannis09}:\\
\noindent
Contractions of $\varphi$ and $\psi$\label{ContractionsThree-FourForm}
\begin{align}
\varphi_{ijk}\varphi_{abc}g^{ia}g^{jb}g^{kc}=&42,\label{dosPhitresg}\\
\varphi_{ijk}\varphi_{abc}g^{jb}g^{kc}=&6g_{ia},\label{dosPhidosg}\\
\varphi_{ijk}\varphi_{abc}g^{kc}=&g_{ia}g_{jb}-g_{ib}g_{ja}-\psi_{ijab}\label{dosPhiunag},
\end{align}

\begin{align}
\varphi_{ijk}\psi_{abcd}g^{ib}g^{jc}g^{kd}=&0,\label{PhiPsitresg}\\
\varphi_{ijk}\psi_{abcd}g^{jc}g^{kd}=&-4\varphi_{iab},\label{PhiPsidosg}\\
\varphi_{ijk}\psi_{abcd}g^{kd}=& g_{ia}\varphi_{jbc}+g_{ib}\varphi_{ajc}+g_{ic}\varphi_{abj}\nonumber\\
&-g_{aj}\varphi_{ibc}-g_{bj}\varphi_{aic}-g_{cj}\varphi_{abi},
\end{align}

\begin{align}
\psi_{ijkl}\psi_{abcd}g^{ld}=&-\varphi_{ajk}\varphi_{ibc}-\varphi_{iak}\varphi_{jbc}-\varphi_{ija}\varphi_{kbc}\nonumber\\
&+g_{ia}g_{jb}g_{kc}+g_{ib}g_{jc}g_{ka}+g_{ic}g_{ja}g_{kb}-g_{ia}g_{jc}g_{kb}\nonumber\\
&-g_{ib}g_{ja}g_{kc}-g_{ic}g_{jb}g_{ka}-g_{ia}\psi_{jkbc}-g_{ja}\psi_{kibc}\nonumber\\
&-g_{ka}\psi_{ijbc}+g_{ab}\psi_{ijbc}-g_{ac}\psi_{ijkb}.
\end{align}

Another useful identity is:
\begin{align}\label{identityvarphi-wedge-varphi}
\varphi \wedge \left(\omega \hk \varphi \right)=&(-2)\left(\psi\wedge \ws \right), \qquad \text{or in coordinates}\nonumber\\
\varphi_{[ijk}\varphi_{mn]l}=&(-2)g_{l[m}\varphi_{nijk]}  
\end{align}
where $w$ is a vector field in $M$ and \ws\ be the $1$--form dual to $w$;  $[\;\;]$ denotes the anti--symmetrization of the indices.

\begin{lemma} (Lemma 4.9 in \cite{Karigiannis09})\label{lemmaG2Curvatureidentity}
Let $M$ be a manifold with a $G_{2}$-structure and let $R_{ijkl}$ be the Riemann curvature tensor of $M$. Then we have $R_{ijkl}\psi^{ijkm}=0$. 
\end{lemma}
The next lemma follows easily from the fact that if $Hol(g)\subseteq G_{2}$ then $R_{ijkl} \in \rm{Sym}^{2}(\rm{g}_2)$, and the above definition (\ref{definition of g2}) of $\rm{g}_{2}$
\begin{lemma}\label{lemmaG2Curvatureidentitywiththreeform}
Let $M$ be a $G_{2}$-manifold and let $R_{ijkl}$ be the Riemann curvature tensor of $M$. Then we have $R_{ijkl}\varphi^{ijm}=0$. 
\end{lemma}

\section{The Chiral de Rham Complex over a \texorpdfstring{$G_{2}$}{}-manifold}\label{CDR of a G2 manifold}

In this section we prove the main results of this paper. We begin remembering the conjecture that is the leitmotiv of this paper.

Let $(M,g)$ be a $G_{2}$-manifold, let $\varphi$ be the corresponding defining 3-form. Let $\Phi_{\pm}$ be the global sections of CDR that correspond to $\varphi$ by Theorem \ref{theorem Reimundo}, and let $K_{\pm}:=G_{(0)}\left(\Phi_{\pm}\right)$. In \cite{Ekstrand-Heluani-Kallen-Zabzine13} the following conjecture was stated:

\begin{conjecture}\cite[Conjecture 7.3]{Ekstrand-Heluani-Kallen-Zabzine13}\label{conjecture}
The sections pairs $\{\Phi_{+},K_{+}\}$ and $\{\Phi_{-},K_{-}\}$ generate two commuting copies of the Shatashvili-Vafa $G_{2}$ superconformal algebra given in subsection \ref{Shatashvili-Vafa algebra} (see Remark \ref{remark-the SV algebra is generated in conf weight 3/2}).
\end{conjecture} 

\subsection{\texorpdfstring{$G_{2}$}{} holonomy manifolds are superconformal}\label{G2 holonomy manifolds are superconformal}
Before stating our main result we define some sections:
\begin{gather*}
X_{\pm}:=\dfrac{1}{6}\left({\Phi_{\pm}}_{(0)}\Phi_{\pm}\right),\quad G_{\pm}:=(-\dfrac{1}{3})\left({\Phi_{\pm}}_{(1)}K_{\pm}\right),\quad L_{\pm}:=\dfrac{1}{2}\left({G_{\pm}}_{(0)}G_{\pm}\right).
\end{gather*}
Note that these are well defined global sections because they are defined in terms of $\Phi_{\pm}$ and $K_{\pm}$.

\begin{theorem}\label{maintheorem2}
The pairs of sections  $\{G_{\pm},L_{\pm}\}$ and $\{\Phi_{\pm},X_{\pm}\}$ generate two $N=1$ superconformal algebras of central charge $\frac{21}{2}$  and $\frac{7}{10}$ respectively. Furthermore the plus signed sections commute with the minus signed sections. 
\end{theorem}

\begin{proof}
The proof is based on explicit computations and some abstract manipulations. To make the proof clear, every time that an explicit computation should be made we will indicate the appropriate subsection of the section \ref{app-computations} where the computation is performed. Note that to compute a $\lambda$-bracket we can choose a local coordinate chart and perform the computation in this chart because we are working with well defined global sections.\\
We begin proving that $\{\Phi_{\pm},X_{\pm}\}$ generate two $N=1$ superconformal algebras of the desired central charges respectively. We work the plus case (the minus case is similar).  

We have by (\ref{correction3form})

\begin{equation}
\Phi_{+}=\tfrac{1}{6}\varphi_{ijk}e^{i}_{+}e^{j}_{+}e^{k}_{+} + \tfrac{1}{2}\varphi_{ijk}\Gamma^{i}_{mn}g^{jm}\partial \gamma^{n}e^{k}_{+}.
\end{equation}

In \ref{PhiwithPhi} it is shown that 

\begin{equation}\label{PhilambdaPhi}
[{\Phi_{+}}_{\lambda}\Phi_{+}]=(-\tfrac{7}{2}) \lambda^{2}+6X_{+},
\end{equation}

\noindent where $X_{+}$ is given by:

\begin{eqnarray}\label{explicitformforX+}
X_{+}&=&-\tfrac{1}{24}\psi_{ijkl}e^{i}_{+}e^{j}_{+}e^{k}_{+}e^{l}_{+}-\tfrac{1}{4}\psi_{ijkl}\Gamma^{i}_{mn}g^{jm}\partial\gamma^{n}e^{k}_{+}e^{l}_{+}\nonumber\\
&&-\tfrac{1}{8}\psi_{ijkl}\Gamma^{i}_{m_{1}n_{1}}g^{jm_{1}}\partial\gamma^{n_{1}}\Gamma^{k}_{m_{2}n_{2}}g^{lm_{2}}\partial\gamma^{n_{2}}-\tfrac{1}{2}g_{ij}\partial(e^{i}_{+})e^{j}_{+}\nonumber\\
&&-\tfrac{1}{2}g_{ij}\Gamma^{j}_{kl}\partial\gamma^{k}e^{l}_{+}e^{i}_{+}-\tfrac{1}{4}\Gamma^{i}_{jk}\Gamma^{k}_{il}\partial\gamma^{j}\partial\gamma^{l}.
\end{eqnarray}

In \ref{XwithPhi} we see that:

\begin{equation}
{X_{+}}_{(2)}\Phi_{+}=0, \;\;\; {X_{+}}_{(1)}\Phi_{+}=-\tfrac{15}{2}\Phi_{+}.
\end{equation}

\noindent
Then, using skewsymmetry (\ref{skewsymmetry}), we get ${\Phi_{+}}_{(2)}X_{+}=0$, ${\Phi_{+}}_{(1)}X_{+}=-\tfrac{15}{2}\Phi_{+}$. To compute ${\Phi_{+}}_{(0)}X_{+}$ we observe that:

\begin{equation}\label{PhizeroPhi}
[{\Phi_{+}}_{\lambda}[{\Phi_{+}}_{\mu}\Phi_{+}]]=6[{\Phi_{+}}_{\lambda}X_{+}]
\end{equation}
\noindent
by (\ref{PhilambdaPhi}). Using the Jacobi identity (\ref{Jacobi identity}) on the left side of (\ref{PhizeroPhi}) we have:

\begin{eqnarray}\label{PhizeroPhiauxiliar}
[{\Phi_{+}}_{\lambda}[{\Phi_{+}}_{\mu}\Phi_{+}]]&=&(-1)[{\Phi_{+}}_{\mu}[{\Phi_{+}}_{\lambda}\Phi_{+}]]+[[{\Phi_{+}}_{\lambda}\Phi_{+}]_{\lambda+\mu}\Phi_{+}]\nonumber\\
&=&(-1)[{\Phi_{+}}_{\mu}6X_{+}]+[{6X_{+}\;}_{\lambda+\mu}\;\Phi_{+}]\nonumber\\
&=&(-6)\left({\Phi_{+}}_{(0)}X_{+}+(-\tfrac{15}{2})\Phi_{+}\mu\right)+(-1)6[{\Phi_{+}\;}_{-\lambda-\mu-\partial}\;X_{+}]\nonumber\\
&=&(-6)\left({\Phi_{+}}_{(0)}X_{+}+(-\tfrac{15}{2})\Phi_{+}\mu\right)\\ \nonumber
&&+(-6)\left({\Phi_{+}}_{(0)}X_{+}+\tfrac{15}{2}(\lambda+\mu+\partial)\Phi_{+}\right).
\end{eqnarray} 

Equating (\ref{PhizeroPhi})  and (\ref{PhizeroPhiauxiliar}) we obtain a polynomial equality in $\lambda$ and $\mu$. Taking the coefficient of $\lambda^{0}\mu$ we get ${\Phi_{+}}_{(0)}X_{+}=-\frac{5}{2}\partial\Phi_{+}$. Then

\begin{equation}\label{PhilambdaX}
[{\Phi_{+}}_{\lambda}X_{+}]=-\tfrac{15}{2}\Phi_{+}\lambda-\tfrac{5}{2}\partial\Phi_{+}.
\end{equation} 

To compute $[{X_{+}}_{\lambda}X_{+}]$ observe that, using (\ref{PhilambdaPhi}) and the Jacobi identity, we have:

\begin{equation}
6[{X_{+}}_{\lambda}X_{+}]=[{X_{+}}_{\lambda}[{\Phi_{+}}_{\mu}\Phi_{+}]]=[{\Phi_{+}}_{\mu}[{X_{+}}_{\lambda}\Phi_{+}]]+[[{X_{+}}_{\lambda}\Phi_{+}]_{\lambda+\mu}\Phi_{+}].\nonumber
\end{equation}

Then 

\begin{equation}\label{XwithXauxiliar}
6[{X_{+}}_{\lambda}X_{+}]=[{\Phi_{+}}_{\mu}-\tfrac{10}{2}\partial\Phi_{+}-\tfrac{15}{2}\Phi_{+}\lambda]+[[{X_{+}}_{\lambda}\Phi_{+}]_{\lambda+\mu}\Phi_{+}].
\end{equation}

This is a polynomial identity in $\lambda$ and $\mu$. Putting $\mu=0$ we get:

\begin{eqnarray}
6[{X_{+}}_{\lambda}X_{+}]&=&-30\partial X_{+}-45X_{+}\lambda+[[{X_{+}}_{\lambda}\Phi_{+}]_{\lambda}\Phi_{+}]\nonumber\\
&=&-30\partial X_{+}-45X_{+}\lambda+[{\tfrac{10}{2}\partial\Phi_{+}-\tfrac{15}{2}\Phi_{+}\lambda\;}_{\lambda}\Phi_{+}]\nonumber\\
&=&-30\partial X_{+}-45X_{+}\lambda+[{-\tfrac{10}{2}\partial\Phi_{+}}_{\lambda}\Phi_{+}]+[{-\tfrac{15}{2}\Phi_{+}}_{\lambda}\Phi_{+}]\lambda,\nonumber
\end{eqnarray}

Using sesquilinearity (\ref{sesquilinearity}) it follows that:

\begin{equation}\label{XlambdaX}
[{X_{+}}_{\lambda}X_{+}]=\tfrac{35}{24}\lambda^{3}-10X_{+}\lambda-5\partial X_{+}.
\end{equation}

\noindent
We have proved that $\{\Phi_{\pm},X_{\pm}\}$ satisfy the $\lambda$-brackets (\ref{PhilambdaPhi}), (\ref{PhilambdaX}) and (\ref{XlambdaX}) and therefore, by Remark \ref{Tri-critical Ising model} we conclude that $\{\Phi_{+},X_{+}\}$ and $\{\Phi_{-},X_{-}\}$ generate two $N=1$ superconformal algebras of central charge $\tfrac{7}{10}$ respectively. Furthermore, as the $e_{+}^{i}$'s commute with the $e_{-}^{i}$'s
(\ref{commutation rules of the e's}), we have that $\Phi_{+}$ commutes with $\Phi_{-}$ and, consequently, the algebra generated by $\{\Phi_{+},X_{+}\}$ commutes with the algebra generated by $\{\Phi_{-},X_{-}\}$.
  
Now we are going to prove that $\{G_{\pm},L_{\pm}\}$ generate two $N=1$ superconformal algebras of central charge $\tfrac{21}{2}$. 
We compute explicitly $G_{\pm}:=(-\dfrac{1}{3})\left({\Phi_{\pm}}_{(1)}K_{\pm}\right)$ (see \ref{Phi-lambda-K}) and find that:

\begin{eqnarray}
G_{+}&=&\tfrac{1}{2}c^{i}\beta_{i}+\tfrac{1}{2}\partial\gamma^{i}b_{i}+\tfrac{1}{2}g^{ij}b_{i}\beta_{j}+\tfrac{1}{2}g^{ij}\Gamma^{l}_{ik}c^{k}b_{j}b_{l}+\tfrac{1}{2}g^{ij}\Gamma^{k}_{ij}\partial b_{k}\nonumber\\
&&+\tfrac{1}{2}g_{ij}\partial\gamma^{i}c^{j}+g^{ij}\Gamma^{k}_{il}\Gamma^{l}_{jm}\partial\gamma^{m}b_{k},\nonumber
\end{eqnarray}
\begin{eqnarray}
G_{-}&=&\tfrac{1}{2}c^{i}\beta_{i}+\tfrac{1}{2}\partial\gamma^{i}b_{i}+(-\tfrac{1}{2})g^{ij}b_{i}\beta_{j}+(-\tfrac{1}{2})g^{ij}\Gamma^{l}_{ik}c^{k}b_{j}b_{l}+(-\tfrac{1}{2})g^{ij}\Gamma^{k}_{ij}\partial b_{k}\nonumber\\
&&+(-\tfrac{1}{2})g_{ij}\partial\gamma^{i}c^{j}+(-1)g^{ij}\Gamma^{k}_{il}\Gamma^{l}_{jm}\partial\gamma^{m}b_{k}.\nonumber
\end{eqnarray}

Note that if we define $G:=G_{+}+G_{-}$, we have $G=c^{i}\beta_{i}+\partial\gamma^{i}b_{i}$, and $\{G,L:=\dfrac{1}{2}{G}_{(0)}G\}$ generate an $N=1$ superconformal algebra of central charge $c=21$, and with respect to $L$ the conformal weight of the generators is as follows (see example \ref{bc-betagamma system} and the definition of $G_{(0)}$ in section \ref{sec-the Chiral de Rham Complex}):
\begin{equation}
\Delta b_{i}=\frac{1}{2},\;\;\Delta c^{i}=\frac{1}{2},\;\;\Delta \gamma^{i}=0,\;\;\Delta\beta_{i}=1,\;\;\;\; i\in\{1,\dots ,7\}.
\end{equation}
\noindent
This implies that the conformal weights of the local sections in the coordinate chart are always positive. We will use this observation below.

In \ref{G+withG-}, we prove that $[{G_{+}}_{\lambda}G_{-}]=0$. By definition, $L_{+}=\dfrac{1}{2}{G_{+}}_{(0)}G_{+}$ and $L_{-}=\dfrac{1}{2}{G_{-}}_{(0)}G_{-}$. Then we have that $\{G_{+},L_{+}\}$ commute with $\{G_{-},L_{-}\}$, and that $L=L_{+}+L_{-}$. We will show that $\{G_{+},L_{+}\}$ is an $N=1$ superconformal algebra (the same proof works for $\{G_{-},L_{-}\}$). First we take care of $[{L_{+}}_{\lambda}L_{+}]$. As $\Delta L_{+}=2$, the observation above about conformal weights implies:
\begin{equation}
[{L_{+}}_{\lambda}L_{+}]=[L_{\lambda}L_{+}]=\partial L_{+}+2 L_{+}\lambda+ A \lambda^{2}+B\lambda^{3},\nonumber
\end{equation}
\noindent
where $A$ and $B$ are two fields with $\Delta A=1$ and $\Delta B=0$.\\
Consequently:
\begin{equation}
[L_{\lambda}A]=\partial A+A\lambda +C\lambda^{2},\;\; [L_{\lambda}B]=\partial B,\nonumber
\end{equation}
where $C$ is a field with $\Delta C=0$, i.e., $[L_{\lambda}C]=\partial C$.
Using the Jacobi identity (\ref{Jacobi identity}), we obtain
\begin{eqnarray}\label{auxiliary1L+lambdaL+}
[{L_{+}}_{\lambda}[L_{\mu}L]]&=&[{L_{+}}_{\mu}[{L_{+}}_{\lambda}L_{+}]]+[{[{L_{+}}_{\lambda}L_{+}]}_{\lambda+\mu}L_{+}].
\end{eqnarray}
Also:
\begin{eqnarray}\label{auxiliary2L+lambdaL+}
[{L_{+}}_{\lambda}[L_{\mu}L]]&=&[{L_{+}}_{\lambda} \partial L_{+}+2L_{+}\mu].
\end{eqnarray}
Working the right side of (\ref{auxiliary2L+lambdaL+}) :
\begin{align}
[{L_{+}}_{\lambda} \partial L_{+}+2L_{+}\mu]=&\left(\partial +\lambda\right)\left(\partial L_{+}+2L_{+}\lambda+A\lambda^{2}+B\lambda^{3}\right)\\\nonumber
&+2\mu\left(\partial L_{+}+2 L_{+}\lambda+ A \lambda^{2}+B\lambda^{3}\right).\nonumber
\end{align}
Working separately the summands on the right side of (\ref{auxiliary1L+lambdaL+}) and putting $\mu=0$:
\begin{eqnarray}
[{L_{+}}_{\mu}[{L_{+}}_{\lambda}L_{+}]]_{{|}_{\mu=0}}&=&[{L}_{\mu}\partial L_{+}+2 L_{+}\lambda+ A \lambda^{2}+B\lambda^{3}]_{{|}_{\mu=0}}\nonumber\\
&=&\partial^{2}L_{+}+2\partial L_{+}\lambda + \partial A\lambda^{2}+\partial B\lambda^{3},\nonumber
\end{eqnarray}
\begin{eqnarray}
[{[{L_{+}}_{\lambda}L_{+}]}_{\lambda}L_{+}]&=&[{\partial L_{+}+2 L_{+}\lambda+ A \lambda^{2}+B\lambda^{3}}_{\lambda} L]\nonumber\\
&=&(-\lambda)\left(\partial L_{+}+2 L_{+}\lambda+ A \lambda^{2}+B\lambda^{3}\right)\nonumber\\
&&+2\lambda\left(\partial L_{+}+2 L_{+}\lambda+ A \lambda^{2}+B\lambda^{3}\right)+A\lambda^{3}\nonumber\\
&&+(-C)\lambda^{4}+(-2)\partial C\lambda^{3}+(-1)\partial^{2}C\lambda^{2}+(-1)\partial B\lambda^{3}.\nonumber
\end{eqnarray}

\noindent
Equating the right sides of (\ref{auxiliary1L+lambdaL+}) and (\ref{auxiliary2L+lambdaL+}), putting $\mu=0$, and looking at the coefficient of $\lambda^{4}$ and $\lambda^{3}$ we get that $C=0$ and $\partial B=A$ respectively.\\
Working separately the summands on the right side of (\ref{auxiliary1L+lambdaL+}) and putting $\lambda=0$:
\begin{equation}
[{L_{+}}_{\mu}[{L_{+}}_{\lambda}L_{+}]]_{{|}_{\lambda=0}}=[{L_{+}}_{\mu}\partial L_{+}]=(\partial+\mu)\left(\partial L_{+}+2 L_{+}\mu+ A \mu^{2}+B\mu^{3}\right),\nonumber
\end{equation}
\begin{equation}
[{[{L_{+}}_{\lambda}L_{+}]}_{\lambda+\mu}L_{+}]_{{|}_{\lambda=0}}=[{\partial L_{+}}_{\mu} L_{+}]=(-1)\mu\left(\partial L_{+}+2 L_{+}\mu+ A \mu^{2}+B\mu^{3}\right).\nonumber
\end{equation}
Equating the right sides of (\ref{auxiliary1L+lambdaL+}) and (\ref{auxiliary2L+lambdaL+}), putting $\lambda=0$, and looking at the coefficient of $\mu^{3}$ we get that $\partial B=0.$ 
But this implies that $A =0$. Also, as $\Delta B=0$, we know that $B$ is a function, therefore $\partial B=0$ implies that $B$ is a constant.\\
Hence $[{L_{+}}_{\lambda}L_{+}]=\left(\partial+2\lambda\right)L_{+}+\frac{c_{+}}{12}\lambda^{3},$ where $c_{+}$ is a constant.\\
Now we compute $[{L_{+}}_{\lambda}G_{+}]$. Proceeding as above we get that  $\Delta G_{+}=\frac{3}{2}$ and, hence:
\begin{equation}
[{L_{+}}_{\lambda}G_{+}]=[L_{\lambda}G_{+}]=\partial G_{+}+\dfrac{3}{2}G_{+}\lambda+C\lambda^{2},\nonumber
\end{equation}
\noindent
where $C$ is a field with $\Delta C=\dfrac{1}{2},$ and 
\begin{equation}
[{L}_{\lambda}C]=\partial C+\dfrac{1}{2}C\lambda.\nonumber
\end{equation}
Using the Jacobi identity (\ref{Jacobi identity}), we obtain
\begin{equation}\label{auxiliary1L+withG+}
[{G_{+}}_{\lambda}[{L_{+}}_{\mu}L_{+}]]=[L_{\mu}[{G_{+}}_{\lambda}L_{+}]]+[{[{G_{+}}_{\lambda}L_{+}]}_{\lambda+\mu}L]. 
\end{equation}
We also have:
\begin{equation}\label{auxiliary2L+withG+}
[{G_{+}}_{\lambda}[{L_{+}}_{\mu}L_{+}]]=[{G_{+}}_{\lambda}(\partial+2\mu)L_{+}].
\end{equation}
Working the right side of (\ref{auxiliary2L+withG+}):
\begin{eqnarray}
[{G_{+}}_{\lambda}(\partial+2\mu)L_{+}]&=&(-1)(\lambda+\partial)\left(\partial G_{+}+\dfrac{3}{2}(-\lambda-\partial)G_{+}+(\lambda+\partial)^{2}C\right)\nonumber\\
&&+2\mu(-1)\left(\partial G_{+}+\dfrac{3}{2}(-\lambda-\partial)G_{+}+(\lambda+\partial)^{2}C\right).\nonumber
\end{eqnarray}
Working separately the summands on the right side of (\ref{auxiliary1L+withG+}) and putting $\mu=0$:
\begin{eqnarray}
[L_{\mu}[{G_{+}}_{\lambda}L_{+}]]_{{|}_{\mu=0}}&=&[L_{\mu}\dfrac{1}{2}\partial G_{+}+\dfrac{3}{2}G_{+}\lambda+(-C)\lambda^{2}]_{{|}_{\mu=0}}\nonumber\\
&=&\dfrac{1}{2}(\partial^2G_{+})+\dfrac{3}{2}\partial G_{+}\lambda+(-1)\partial C\lambda^{2},\nonumber
\end{eqnarray}
\begin{eqnarray}
[{[{G_{+}}_{\lambda}L_{+}]}_{\lambda+\mu}L]_{{|}_{\mu=0}}&=&[\dfrac{1}{2}\partial G_{+}+\dfrac{3}{2}G_{+}\lambda+{(-C)\lambda^{2}}_{\lambda}L]\nonumber\\
&=&(-\dfrac{1}{2})\lambda\left(\dfrac{1}{2}\partial G_{+}+\dfrac{3}{2}G_{+}\lambda+(-C)\lambda^{2}\right)\nonumber\\
&&+\dfrac{3}{2}\lambda\left(\dfrac{1}{2}\partial G_{+}+\dfrac{3}{2}G_{+}\lambda+(-C)\lambda^{2}\right)+(-\dfrac{1}{2})C\lambda^{3}.\nonumber
\end{eqnarray}
Equating the right sides of (\ref{auxiliary1L+withG+}) and (\ref{auxiliary2L+withG+}) putting $\mu=0$, and looking at the coefficient of $\lambda^{3}$ we get that $C=0$ which implies $[{L_{+}}_{\lambda}G_{+}]=\partial G_{+}+\dfrac{3}{2}G_{+}\lambda.$\\
Now we compute $[{G_{+}}_{\lambda}G_{+}]$. By definition of $L_{+}$ and conformal weights positivity we have:
\begin{equation}
[{G_{+}}_{\lambda}G_{+}]=[{G}_{\lambda}G_{+}]=2L_{+}+D\lambda+E\lambda^{2},\nonumber 
\end{equation}
with $D$ and $E$ two fields of conformal weights $1$ and $0$ respectively.

Using the Jacobi identity (\ref{Jacobi identity}) we obtain
\begin{equation}\label{auxiliary1G+withG+}
[{L_{+}}_{\lambda}[G_{\mu}G]]=[{G_{+}}_{\mu}[{L_{+}}_{\lambda}G_{+}]]+[[{L_{+}}_{\lambda}G_{+}]_{\lambda+\mu}G_{+}].
\end{equation}
We also have:
\begin{equation}\label{auxiliary2G+withG+}
[{L_{+}}_{\lambda}[G_{\mu}G]]=[{L_{+}}_{\lambda}2L]=[{L_{+}}_{\lambda}2L_{+}]=2(\partial+2\lambda)L_{+}+2\dfrac{c_{+}}{12}\lambda^{3}.
\end{equation}
Working separately the summands on the right side of (\ref{auxiliary1G+withG+}):
\begin{align}
[{G_{+}}_{\mu}[{L_{+}}_{\lambda}G_{+}]]=&[{G_{+}}_{\mu}\partial G_{+}+\dfrac{3}{2}G_{+}\lambda]\nonumber\\
=&(\partial+\mu)\left(2L_{+}+D\mu+E\mu^{2}\right)+\dfrac{3}{2}\lambda\left(2L_{+}+D\mu+E\mu^{2}\right),\nonumber
\end{align}
\begin{align}
[{[{L_{+}}_{\lambda}G_{+}]}_{\lambda+\mu}G_{+}]=&[{\partial G_{+}+\dfrac{3}{2}G_{+}\lambda}_{\lambda+\mu}G_{+}]\nonumber\\
=&(-1)(\lambda+\mu)\left(2L_{+}+D(\lambda+\mu)+E(\lambda+\mu)^{2}\right)\nonumber\\
&+\dfrac{3}{2}\lambda\left(2L_{+}+D(\lambda+\mu)+E(\lambda+\mu)^{2}\right).\nonumber
\end{align}
Equating both the right sides of (\ref{auxiliary1G+withG+}) and (\ref{auxiliary2G+withG+}) and looking at the coefficient of $\lambda^{2}$ with $\mu=0$ we get that $D=0$ while looking at the coefficient of $\lambda^{3}$ with $\mu=0$ we get that $E=\dfrac{c_{+}}{3}$. Therefore $[{G_{+}}_{\lambda}G_{+}]=2L_{+}+\dfrac{c_{+}}{3}\lambda^{2}$.\\
We have proved that  $\{G_{+},L_{+}\}$ and $\{G_{-},L_{-}\}$ are two commuting  $N=1$ superconformal algebras of central charge $c_{+}$ and $c_{-}$ respectively, such that $c_{+}+c_{-}=21.$\\
We can calculate the central charge by computing explicitly the lambda bracket $[{G_{+}}_{\lambda}G_{+}]=[G_{\lambda}G_{+}]$ and looking at the coefficient of $\lambda^{2}$.\\
Computing the coefficient of $\lambda^{2}$:
\small
\begin{align}
\frac{1}{2}\left(\frac{\partial^{2}}{\partial \lambda}\left(\left[{G_{+}}_{\lambda}G_{+}\right]\right)\right)=&\dfrac{7}{4}+(-\dfrac{1}{4})\left(g^{lm}_{,l}\right)_{,m}+\dfrac{1}{4}\left(g^{lm}\Gamma^{n}_{lm}\right)_{,n}+\dfrac{1}{2}g^{lm}\Gamma^{n}_{la}\Gamma^{a}_{mn}+\dfrac{7}{4}\nonumber\\ 
=&\dfrac{7}{2}+\dfrac{1}{2}g^{lm}\Gamma^{n}_{la}\Gamma^{a}_{mn}+\dfrac{1}{2}\left(g^{lm}\Gamma^{n}_{lm}\right)_{,n}\nonumber\\ 
=&\dfrac{7}{2}+\left(-\dfrac{1}{2}\right)g^{lm}\Gamma^{n}_{la}\Gamma^{a}_{mn}+\dfrac{1}{2}g^{lm}\left(\Gamma^{n}_{lm}\right)_{,n}\nonumber\\
=&\dfrac{7}{2}+\dfrac{1}{2}g^{lm}R_{lm}\nonumber\\
=&\dfrac{7}{2}.\nonumber
\end{align}
\normalsize
The last equality follows because $G_{2}$-manifolds are Ricci flat, Lemma \ref{lemmaG2Ricciflat}. It is interesting to note that, as in other places,  the scalar curvature of the manifold, i.e., $g^{ij}R_{ij}$, appears here explicitly. Then, the central charge of the $N=1$ superconformal algebra $\{G_{+},L_{+}\}$ is $c_{+}=\frac{21}{2}$ and consequently as $c_{+}+c_{-}=21$, the central charge of $\{G_{-},L_{-}\}$ is $c_{-}=\frac{21}{2}$.

To conclude the proof of the theorem we only need to check that $G_{+}$(resp. $G_{-}$) commutes with $\Phi_{-}$(resp. $\Phi_{+}$). This is accomplished by performing an explicit computation in \ref{G+withPhi-}.
\end{proof}

\begin{remark}\label{rm-commentaboutSpin(7)case}
Note that $X_{+}$ (\ref{explicitformforX+}) is not exactly the section produced by Theorem \ref{theorem Reimundo} using the 4-form, because, besides the correction (\ref{correction4form}), we need to add other terms. That this would happen was already observed in \cite{Ekstrand-Heluani-Kallen-Zabzine13} while working the flat case. For this reason, Conjecture \ref{conjecture} was formulated using only the $3$-form. The expression (\ref{explicitformforX+}) that we have obtained here works in any coordinate system.
\end{remark}

\begin{remark}
It should be noticed that the complexity of the explicit computations performed is greater than in the other holonomy cases. This is due to the lack of special coordinate systems simplifying the corrections to the fields. This in turn is a reflection of our lack of knowledge about the geometry of manifolds with $G_{2}$ holonomy.  
\end{remark}

\subsection{Proof of the Conjecture}\label{subsec-Towards a proof of the Conjecture}
In this subsection, we prove Conjecture \ref{conjecture}.
\begin{theorem}\label{maintheorem}
The sections pairs $\{\Phi_{+},K_{+}\}$ and $\{\Phi_{-},K_{-}\}$ generate two commuting copies of the Shatashvili-Vafa $G_{2}$ superconformal algebra.
\end{theorem}
In the above subsection \ref{G2 holonomy manifolds are superconformal}, besides the global sections $\Phi_{\pm}$ and $K_{\pm}=G_{(0)}\Phi_{\pm}$, we introduced the global sections:
\begin{gather*}
X_{\pm}:=\dfrac{1}{6}\left({\Phi_{\pm}}_{(0)}\Phi_{\pm}\right),\quad G_{\pm}:=(-\dfrac{1}{3})\left({\Phi_{\pm}}_{(1)}K_{\pm}\right),\quad L_{\pm}:=\dfrac{1}{2}\left({G_{\pm}}_{(0)}G_{\pm}\right).
\end{gather*}
Now we introduce $M_{\pm}:=G_{(0)}X_{\pm}={G_{\pm}}_{(0)}X_{\pm}$. Note that, again, these sections are globally well defined because they are defined in terms of $X_{\pm}$ using the supersymmetric generator $G_{(0)}$.
\begin{remark}\label{rm-somebracketsalreadychecked}
The reader should note that Theorem \ref{maintheorem} not only proves that we have two pair of commuting copies of the $N=1$ superconformal algebra. Namely it is implicit in the proof that $\{\Phi_{\pm}, X_{\pm}\}$ satisfy the commutation rules of the same named fields of the Shatashvili-Vafa $G_{2}$ superconformal algebra (subsection \ref{Shatashvili-Vafa algebra}), i.e.,:
\begin{gather*}
[{\Phi_{\pm}}_{\lambda}\Phi_{\pm}]=(-\frac{7}{2}) \lambda^{2}+6X_{\pm},\;\; [{\Phi_{\pm}}_{\lambda}X_{\pm}]=-\frac{15}{2}\Phi_{\pm}\lambda-\frac{5}{2}\partial\Phi_{\pm},\\
[{X_{\pm}}_{\lambda}X_{\pm}]=\frac{35}{24}\lambda^{3}-10X_{\pm}\lambda-5\partial X_{\pm}.
\end{gather*}
It follows from the way we define the global sections (see also Remark \ref{remark-the SV algebra is generated in conf weight 3/2}) that  $\{L_{+}, X_{+},K_{+},M_{+}\}$ (resp. $\{L_{-},X_{-}, K_{-}, M_{-}\}$)  can be expressed in terms of $\Phi_{+}$  and $G_{+}$ (resp. $\Phi_{-}$  and $G_{-}$ ). Remember also that in Theorem \ref{maintheorem} we proved that $\{\Phi_{+},G_{+}\}$ commute with $\{\Phi_{-},G_{-}\}$. Therefore $\{L_{+}, G_{+},\Phi_{+}, X_{+},K_{+},M_{+}\}$ commute with $\{L_{-}, G_{-}\Phi_{-},X_{-},K_{-},M_{-}\}$.
\end{remark}
By the remark above, to prove the conjecture we only need to check the $\lambda$--brackets between the fields $\{L_{\pm}, G_{\pm}, \Phi_{\pm}, X_{\pm}, K_{\pm}, M_{\pm}\}$ satisfy the same $\lambda$-brackets of the Shatashvili-Vafa $G_{2}$ algebra (ruling out the ones have already been checked in the way of proving Theorem \ref{maintheorem}). We work the plus case, the minus case is similar.\\ 
We prove first the linear $\lambda$-brackets (subsection \ref{subs-linearbrackets}); and then the non-linear $\lambda$-brackets (\ref{subsubsec-non linear lambda brackets}) under the assumption that we have proved the relation (\ref{ideal in the definition of the algebra}):
\begin{align}\label{assumption of the relation}
0=4:G_{\pm}X_{\pm}:-2:\Phi_{\pm} K_{\pm}:-4\partial M_{\pm}-\partial^{2}G_{\pm}.
\end{align}
Finally (subsection \ref{subs-checkingtherelation}) we prove the above relation.
\begin{remark}\label{rm-discussion about non-linearities}
We need the assumption (\ref{assumption of the relation}) only to prove the non-linear $\lambda$-brackets, the linear ones follow without the assumption. In fact, the relation (\ref{assumption of the relation}) is used only to prove the $\lambda$-bracket $[{\Phi_{+}}_{\lambda}M_{+}]$, the others non-linear $\lambda$-brackets are deduced from this one. Similarly we could have computed first another non-linear $\lambda$-bracket using the relation, and then deduce the others ones, that is, there is nothing special in $[{\Phi_{+}}_{\lambda}M_{+}]$. Even more, if we are able to check a non-linear $\lambda$-bracket without the assumption, then we can prove the relation (\ref{assumption of the relation}) using the Jacobi identity (\ref{Jacobi identity}), remenber that the space of global sections of the CDR is a vertex algebra. The moral here is that we need to verify at least one non-linear identity among the fields. We opted here to check the relation (\ref{assumption of the relation}).
\end{remark}
To simplify the notation we denote the sypersymmetric generator $G_{(0)}$ by $D$. Remember that $D$ is an odd derivation of all the $n$-products; and that $[D,\partial]=0$ because $D^{2}=\partial$.
 
\subsubsection{Linear \texorpdfstring{$\lambda$}{}-brackets}\label{subs-linearbrackets}
\begin{itemize}
\item[1)] $[{\Phi_{+}}_{\lambda}K_{+}]$\\
By definition of $G_{+}$ we have that ${\Phi_{+}}_{(1)}K_{+}=(-3)G_{+}$. We are left with ${\Phi_{+}}_{(2)}K_{+}$ and ${\Phi_{+}}_{(0)}K_{+}$, because ${\Phi_{+}}_{(n)}K_{+}=0$ for $n\ge 3$ by the positivity of the conformal weight.

Computing ${\Phi_{+}}_{(2)}K_{+}$:\\
By skewsymmetry (\ref{skewsymmetry}), we have $[{\Phi_{+}}_{\lambda}K_{+}]=(-1)[{K_{+}}_{-\partial-\lambda}\Phi_{+}]$. Then,  
\begin{align*}
&{\Phi_{+}}_{(2)}K_{+}\tfrac{\lambda^{2}}{2}+{\Phi_{+}}_{(1)}K_{+}\lambda+{\Phi_{+}}_{(0)}K_{+}\\
=&-{K_{+}}_{(2)}\Phi_{+}\tfrac{\left(-\partial-\lambda\right)^{2}}{2}-{K_{+}}_{(1)}\Phi_{+}\left(-\partial-\lambda\right)-{K_{+}}_{(0)}\Phi_{+}.
\end{align*}
Expanding the right side of the last equality and looking at the coefficient of $\lambda^{2}$ we get that ${\Phi_{+}}_{(2)}K_{+}=-{K_{+}}_{(2)}\Phi_{+}$. On the other hand, as $D$ is an odd derivation of the $n$-products, we have
\begin{align*}
0=D({\Phi_{+}}_{(2)}\Phi_{+})={D\Phi_{+}}_{(2)}\Phi_{+}+(-1){\Phi_{+}}_{(2)}D\Phi_{+}={K_{+}}_{(2)}\Phi_{+}-{\Phi_{+}}_{(2)}K_{+}.
\end{align*}
This implies ${\Phi_{+}}_{(2)}K_{+}=0$.

Computing ${\Phi_{+}}_{(0)}K_{+}$:\\
Applying Borcherds identity (\ref{Borcherds identity}) with $a=\Phi_{+}, b=K_{+}, c=\left|0\right>$,  $n=-1, m=2, k=-2$, we obtain
\begin{align*}
&\binom{2}{1}\left({\Phi_{+}}_{(0)}K_{+}\right)_{(-1)}\left|0\right>+\binom{2}{2}\left({\Phi_{+}}_{(1)}K_{+}\right)_{(-2)}\left|0\right>\\
=&\binom{-1}{0}{\Phi_{+}}_{(1)}\left({K_{+}}_{(-2)}\left|0\right>\right)+(-1)\binom{-1}{1}{\Phi_{+}}_{(0)}\left({K_{+}}_{(-1)}\left|0\right>\right).
\end{align*}
Then
\begin{align}\label{auxiliary1 Phi-0-K}
{\Phi_{+}}_{(0)}K_{+}=3\partial G_{+}+{\Phi_{+}}_{(1)}\partial K_{+}.
\end{align}
As we have already computed $[{\Phi_{+}}_{\lambda}\Phi_{+}]$, using sesquilinearity (\ref{sesquilinearity}), we have that ${\Phi_{+}}_{(1)}\partial\Phi_{+}=6X_{+}$. Using that $D$ is an odd derivation we get $6M_{+}=D({\Phi_{+}}_{(1)}\partial\Phi_{+})={K_{+}}_{(1)}\partial\Phi_{+}-{\Phi_{+}}_{(1)}\partial K_{+}$. Then,
\begin{align}\label{auxiliary2 Phi-0-K}
{\Phi_{+}}_{(1)}\partial K_{+}={K_{+}}_{(1)}\partial\Phi_{+}-6M_{+}.
\end{align}
We also have $\partial\left({\Phi_{+}}_{(1)}K_{+}\right)={\partial\Phi_{+}}_{(1)}K_{+}+{\Phi_{+}}_{(1)}\partial K_{+}$. Which implies
\begin{align}\label{auxiliary3 Phi-0-K}
(-3)\partial G={\partial\Phi_{+}}_{(1)}K_{+}+{\Phi_{+}}_{(1)}\partial K_{+}.
\end{align} 
As ${\Phi_{+}}_{(n)}K_{+}=0$ for $n\ge 2$, by skewsymmetry (\ref{skewsymmetry}) we have
\small 
\begin{align*}
[{\partial\Phi_{+}}_{\lambda}K_{+}]&=(-1)[{K_{+}}_{-\lambda-\partial}\partial\Phi_{+}]\\
\quad&=(-1)\left({K_{+}}_{(2)}\partial\Phi_{+}\left(-\lambda-\partial\right)^{2}+{K_{+}}_{(1)}\partial\Phi_{+}\left(-\lambda-\partial\right)+{K_{+}}_{(0)}\partial\Phi_{+}\right).
\end{align*}
\normalsize
Then
\begin{align}\label{auxiliary4 Phi-0-K}
{\partial\Phi_{+}}_{(1)}K_{+}={K_{+}}_{(1)}\partial\Phi_{+}-\partial\left({K_{+}}_{(2)}\partial\Phi_{+}\right).
\end{align}
Applying Borcherds identity (\ref{Borcherds identity}) with $a=\Phi_{+}, b=K_{+}, c=\left|0\right>$, $n=1, m=-2, k=1$ we obtain
\small
\begin{align*}
\binom{-2}{0}\left({\Phi_{+}}_{(1)}K_{+}\right)_{(-1)}\left|0\right>=&{K_{+}}_{(2)}\left({\Phi_{+}}_{(-2)}\left|0\right>\right)+(-1){K_{+}}_{(1)}\left({\Phi_{+}}_{(-1)}\left|0\right>\right)\\
{\Phi_{+}}_{(1)}K_{+}=&{K_{+}}_{(2)}\partial\Phi_{+}+(-1){K_{+}}_{(1)}\Phi_{+}.
\end{align*}
\normalsize
Then using skewsymmetry (\ref{skewsymmetry}) we get:
\begin{align}\label{auxiliary5 Phi-0-K}
{K_{+}}_{(2)}\partial\Phi_{+}=2{\Phi_{+}}_{(1)}K_{+}=(-6)G_{+}.
\end{align}
Finally substituting the equations (\ref{auxiliary5 Phi-0-K}), (\ref{auxiliary4 Phi-0-K}), (\ref{auxiliary3 Phi-0-K}), (\ref{auxiliary2 Phi-0-K}) and (\ref{auxiliary1 Phi-0-K}), we get: ${\Phi_{+}}_{(0)}K_{+}=(-3)M_{+}+(-\frac{3}{2})\partial G_+{}$.
Then we have proved that
\begin{align*}
[{\Phi_{+}}_{\lambda}K_{+}]=-3G_{+}\lambda-3\left(M_{+}+\frac{1}{2}\partial G_{+}\right).
\end{align*} 
\item[2)]$[{K_{+}}_{\lambda}K_{+}]$\\
As $D$ is an odd derivation of the $n$-products we have $D\left({\Phi_{+}}_{(n)}K_{+}\right)={K_{+}}_{(n)}K_{+}-{\Phi_{+}}_{(n)}\partial \Phi_{+}$ for all $n\in \mathbb{Z}$. Then 
\begin{align}\label{auxiliary K-lambda-K}
{K_{+}}_{(n)}K_{+}=D\left({\Phi_{+}}_{(n)}K_{+}\right)+{\Phi_{+}}_{(n)}\partial \Phi_{+}.
\end{align}
We already know $[{\Phi_{+}}_{\lambda}K_{+}]$ and $[{\Phi_{+}}_{\lambda}\Phi_{+}]$, specialising equation (\ref{auxiliary K-lambda-K}) for $n=3,2,1,0$ we conclude that:
\begin{align*}
[{K_{+}}_{\lambda}K_{+}]=-\frac{21}{6}\lambda^{3}+6\left(X_{+}-L_{+}\right)\lambda+3\partial\left(X_{+}-L_{+}\right).
\end{align*}
\item[3)]$[{G_{+}}_{\lambda} \Phi_{+}]$\\
By positivity of the conformal weight and the definition of $K_{+}:={G}_{(0)}\Phi_{+}={G_{+}}_{(0)}\Phi_{+}$ we have that 
\begin{align*}
[{G_{+}}_{\lambda} \Phi_{+}]=A\dfrac{\lambda^{2}}{2}+B\lambda+K_{+},
\end{align*}
where $A$ and $B$ are two fields.
To find $A$ and $B$ we perform an explicit computation \ref{G+ lambda Phi+}, and found that $A=B=0$. Then we conclude:
\begin{align*}
[{G_{+}}_{\lambda} \Phi_{+}]=K_{+}.
\end{align*}
\item[4)]$[{L_{+}}_{\lambda}\Phi_{+}]$\\
As $\Delta \Phi_{+}=\tfrac{3}{2}$ we have
\begin{align*}
[{L_{+}}_{\lambda}\Phi_{+}]=[L_{\lambda}\Phi_{+}]=A\lambda^{2}+(\partial+\frac{3}{2}\lambda)\Phi_{+},
\end{align*}
where $A$ is a field. To find $A$ we perform and explicit computation \ref{L+ lambda Phi+}  and found that $A$=0. Then we conclude:
\begin{align*}
[{L_{+}}_{\lambda}\Phi_{+}]=(\partial+\frac{3}{2}\lambda)\Phi_{+}.
\end{align*}
\item[6)]$[{L_{+}}_{\lambda}X_{+}]$\\
Applying the Jacobi identity (\ref{Jacobi identity}) we have
\begin{align*}
[{L_{+}}_{\lambda}X_{+}]=[{L_{+}}_{\lambda}\tfrac{1}{6}[{\Phi_{+}}_{\mu}\Phi_{+}]]=\frac{1}{6}\left([{\Phi_{+}}_{\mu}[{L_{+}}_{\lambda}\Phi_{+}]]+[{[{L_{+}}_{\lambda}\Phi_{+}]}_{\lambda+\mu}\Phi_{+}]\right).
\end{align*}
From where
\begin{align*}
[{L_{+}}_{\lambda}X_{+}]=-\frac{7}{24}\lambda^{3}+2X_{+}\lambda+\partial X_{+}.
\end{align*}
\item[7)]$[{L_{+}}_{\lambda}K_{+}]$\\
Applying the Jacobi identity (\ref{Jacobi identity}) we have
\begin{align*}
[{L_{+}}_{\lambda}K_{+}]=[{L_{+}}_{\lambda}[{G_{+}}_{\mu}\Phi_{+}]]=[{G_{+}}_{\mu}[{L_{+}}_{\lambda}\Phi_{+}]]+[{[{L_{+}}_{\lambda}G_{+}]}_{\lambda+\mu}\Phi_{+}].
\end{align*}
We conclude that:
\begin{align*}
[{L_{+}}_{\lambda}K_{+}]=(\partial+2\lambda)K_{+}.
\end{align*}
\item[8)]$[{G_{+}}_{\lambda}X_{+}]$\\
Computing in much the same way as $[{L_{+}}_{\lambda}X_{+}]$, we obtain:
\begin{align*}
[{G_{+}}_{\lambda}X_{+}]=-\frac{1}{2}G_{+}\lambda+M_{+}.
\end{align*}
\item[9)]$[{G_{+}}_{\lambda}K_{+}]$\\
In the same way as $[{L_{+}}_{\lambda}K_{+}]$, we obtain:
\begin{align*}
[{G_{+}}_{\lambda}K_{+}]=3\Phi_{+}\lambda+\partial\Phi_{+}.
\end{align*}
\item[10)]$[{L_{+}}_{\lambda}M_{+}]$
\begin{align}\label{auxiliary1 L+ lambda M+}
[{L_{+}}_{\lambda}[{G_{+}}_{\mu}X_{+}]]=&[{L_{+}}_{\lambda}(-\tfrac{1}{2})G_{+}\mu+M_{+}]\\
=&(-\tfrac{1}{2})[{L_{+}}_{\lambda}G_{+}]\mu+[{L_{+}}_{\lambda}M_{+}].
\end{align}
On the other hand using the Jacobi identity (\ref{Jacobi identity}) we have:
\begin{align}\label{auxiliary2 L+ lambda M+}
[{L_{+}}_{\lambda}[{G_{+}}_{\mu}X_{+}]]=[{G_{+}}_{\mu}[{L_{+}}_{\lambda}X_{+}]]+[{[{L_{+}}_{\lambda}G_{+}]}_{\lambda+\mu}X_{+}].
\end{align}
Equating the right sides of the equations (\ref{auxiliary1 L+ lambda M+}) and (\ref{auxiliary2 L+ lambda M+}), expanding using the already known $\lambda$-brackets and then looking at the coefficients of the $\lambda$'s we get:
\begin{align*}
[{L_{+}}_{\lambda}M_{+}]=-\frac{1}{4}G_{+}\lambda^{2}+\frac{5}{2}M_{+}\lambda+\partial M_{+}.
\end{align*}

\item[11)]$[{G_{+}}_{\lambda}M_{+}]$\\
Computing in much the same way as $[{L_{+}}_{\lambda}M_{+}]$, we obtain:
\begin{align*}
[{G_{+}}_{\lambda}M_{+}]=-\frac{7}{12}\lambda^{3}+\left(L_{+}+4X_{+}\right)\lambda+\partial X_{+},
\end{align*}

\end{itemize}

\subsubsection{Non linear \texorpdfstring{$\lambda$}{}-brackets}\label{subsubsec-non linear lambda brackets}
\begin{itemize}
\item[1)]$[{\Phi_{+}}_{\lambda}M_{+}]$\\
Computing ${\Phi_{+}}_{(3)}M_{+}$:
\begin{align*}
{\Phi_{+}}_{(3)}M_{+}={\Phi_{+}}_{(3)}\left({G_{+}}_{(0)}X_{+}\right).
\end{align*}
Using the Borcherds identity (\ref{Borcherds identity}) with $m=3$, $k=0$, $n=0$, $a=G_{+}$, $b=X_{+}$, and $c=\Phi_{+}$, we obtain:
\begin{align*}
\binom{3}{0}\left({G_{+}}_{(0)}X_{+}\right)_{(3)}\Phi_{+}+\binom{3}{1}\left({G_{+}}_{(1)}X_{+}\right)_{(2)}\Phi_{+}\\
={G_{+}}_{(3)}\left({X_{+}}_{(0)}\Phi_{+}\right)+(-1){X_{+}}_{(0)}\left({G_{+}}_{(3)}\Phi_{+}\right).
\end{align*}
implying that ${M_{+}}_{(3)}\Phi_{+}=0$ and, by skewsymmetry (\ref{skewsymmetry}), that ${\Phi_{+}}_{(3)}M_{+}=0$.\\
Computing ${\Phi_{+}}_{(2)}M_{+}$:
\begin{align*}
{\Phi_{+}}_{(2)}M_{+}={\Phi_{+}}_{(2)}\left({G_{+}}_{(0)}X_{+}\right).
\end{align*}
Using the Borcherds identity (\ref{Borcherds identity}) with $m=2$, $k=0$, $n=0$, $a=G_{+}$, $b=X_{+}$, and $c=\Phi_{+}$, we obtain:
\small
\begin{align*}
\binom{2}{0}\left({G_{+}}_{(0)}X_{+}\right)_{(2)}\Phi_{+}+\binom{2}{1}\left({G_{+}}_{(1)}X_{+}\right)_{(1)}\Phi_{+}\\
={G_{+}}_{(2)}\left({X_{+}}_{(0)}\Phi_{+}\right)+(-1){X_{+}}_{(0)}\left({G_{+}}_{(2)}\Phi_{+}\right).
\end{align*}
\normalsize
Then ${M_{+}}_{(2)}\Phi_{+}=0$ and by skewsymmetry (\ref{skewsymmetry}) ${\Phi_{+}}_{(2)}M_{+}=0$.\\
Computing ${\Phi_{+}}_{(1)}M_{+}$:
\begin{align*}
{\Phi_{+}}_{(1)}M_{+}={\Phi_{+}}_{(1)}\left({G_{+}}_{(0)}X_{+}\right).
\end{align*}
Using the Borcherds identity (\ref{Borcherds identity}) with $m=1$, $k=0$, $n=0$, $a=G_{+}$, $b=X_{+}$, and $c=\Phi_{+}$, we obtain:
\begin{align*}
\left({G_{+}}_{(0)}X_{+}\right)_{(1)}\Phi_{+}+(-\tfrac{1}{2})K_{+}&={G_{+}}_{(1)}\left((-\tfrac{10}{2})\partial\Phi_{+}\right).
\end{align*}
Then ${M_{+}}_{(1)}\Phi_{+}=-\tfrac{9}{2}K_{+}$ and by skewsymmetry (\ref{skewsymmetry}) ${\Phi_{+}}_{(1)}M_{+}=\tfrac{9}{2}K_{+}$.\\
Computing ${\Phi_{+}}_{(0)}M_{+}$:\\
By the assumption (\ref{assumption of the relation}) we have
\begin{align*}
0=&4:G_{+}X_{+}:-2:\Phi_{+} K_{+}:-4\partial M_{+}-\partial^{2}G_{+}.
\end{align*}
Using quasi-commutativity (\ref{quasi-commutativity}) we have
\begin{align*}
:G_{+}X_{+}:-:X_{+}G_{+}:=\int_{-\partial}^{0}[{G_{+}}_{\lambda}X_{+}]d\lambda=\tfrac{1}{4}\partial^{2}G_{+}+\partial M_{+},
\end{align*}
and
\begin{align*}
:\Phi_{+} K_{+}:-:K_{+}\Phi_{+}:=\int_{-\partial}^{0}[{\Phi_{+}}_{\lambda}K_{+}]d\lambda=(-3)\partial M_{+}.
\end{align*}
Thus,
\small 
\begin{align*}
0=&4:X_{+}G_{+}:-2:K_{+}\Phi_{+}:+6\partial M_{+}.
\end{align*}
\normalsize
It follows that
\begin{align*}
[{\Phi_{+}}_{\lambda}\,4:X_{+}G_{+}:-2:K_{+}\Phi_{+}:+6\partial M_{+}]=0.
\end{align*}
The non-commutative Wick formula (\ref{Wickformula}) implies:
\small
\begin{align*}
[{\Phi_{+}}_{\lambda}4:X_{+}G_{+}:]=&4:[{\Phi_{+}}_{\lambda}X_{+}]G_{+}:+4:X_{+}[{\Phi_{+}}_{\lambda}G_{+}]:\\
&+4\int_{0}^{\lambda}\left[[{\Phi_{+}}_{\lambda}X_{+}]_{\mu}G_{+}\right]d\mu\\
=&(-25)K_{+}\lambda^{2}+(-30):\Phi_{+} G_{+}:\lambda+(-10):\partial\Phi_{+} G_{+}:\\
&+4:X_{+}K_{+}:.
\end{align*}
\normalsize
We already know that $[{\Phi_{+}}_{\lambda}M_{+}]=\tfrac{9}{2}K_{+}\lambda+{\Phi_{+}}_{(0)}M_{+}$.
The non-commutative Wick formula (\ref{Wickformula}) then gives:
\small
\begin{align*}
[{\Phi_{+}}_{\lambda}(-2):K_{+}\Phi_{+}:]=&(-2)\left(:[{\Phi_{+}}_{\lambda}K_{+}]\Phi_{+}:+:K_{+}[{\Phi_{+}}_{\lambda}\Phi_{+}]:\right.\\
&\left.+\int_{0}^{\lambda}\left[[{\Phi_{+}}_{\lambda}K_{+}]_{\mu}\Phi_{+}\right]d\mu\right)\\
=&(-2)K_{+}\lambda^{2}+\left(6{\Phi_{+}}_{(0)}M_{+}+6:G_{+}\Phi_{+}:+(-27)\partial K_{+}\right)\lambda\\
&+6:M_{+}\Phi_{+}:+3:\partial G_{+}\Phi_{+}:+(-12):K_{+}X_{+}:.
\end{align*}
\normalsize
By sesquilinearity (\ref{sesquilinearity}), we have
\begin{align*}
[{\Phi_{+}}_{\lambda}6\partial M_{+}]=27K_{+}\lambda^{2}+\left(6{\Phi_{+}}_{(0)}M_{+}+27\partial K_{+}\right)\lambda+6\partial\left({\Phi_{+}}_{(0)}M_{+}\right).
\end{align*}
Looking at the coefficient of $\lambda$ in $[{\Phi_{+}}_{\lambda}\,4:X_{+}G_{+}:-2:K_{+}\Phi_{+}:+6\partial M_{+}]=0$
and using that $:\Phi_{+}G_{+}:+:G_{+}\Phi_{+}:=\partial K_{+}$ (by quasi-commutativity (\ref{quasi-commutativity})), we obtain
\begin{align*}
{\Phi_{+}}_{(0)}M_{+}=(-3):G_{+}\Phi_{+}:+\tfrac{5}{2}\partial K_{+},
\end{align*}
from where we conclude that $[{\Phi_{+}}_{\lambda}M_{+}]=\tfrac{9}{2}K_{+}\lambda-\left(3:G_{+}\Phi_{+}:-\tfrac{5}{2}\partial K_{+}\right)$.

\item[2)]$[{K_{+}}_{\lambda}M_{+}]$\\
Using that $D$ is an odd derivation we get:
\begin{align*}
D\left({\Phi_{+}}_{(n)}M_{+}\right)={K_{+}}_{(n)}M_{+}-{\Phi_{+}}_{(n)}\partial X_{+}.
\end{align*}
In others words, ${K_{+}}_{(n)}M_{+}=D\left({\Phi_{+}}_{(n)}M_{+}\right)+{\Phi_{+}}_{(n)}\partial X_{+}$ and, we can compute $[{K_{+}}_{\lambda}M_{+}]$ in terms of already known data. We conclude that:
\begin{align*}
[{K_{+}}_{\lambda}M_{+}]=-\frac{15}{2}\Phi_{+}\lambda^{2}-\frac{11}{2}\partial\Phi_{+}\lambda+3\left(:G_{+}K_{+}:-2:L_{+}\Phi_{+}:\right).
\end{align*}
\item[3)]$[{X_{+}}_{\lambda}K_{+}]$\\
Using that $D$ is an odd derivation we have:
\begin{align*}
D\left({K_{+}}_{(n)}M_{+}\right)={\partial\Phi_{+}}_{(n)}M_{+}+{K_{+}}_{(n)}\partial X_{+}.
\end{align*}
Then ${K_{+}}_{(n)}\partial X_{+}=D\left({K_{+}}_{(n)}M_{+}\right)-{\partial\Phi_{+}}_{(n)}M_{+}$, that is, we can compute $[{K_{+}}_{\lambda}X_{+}]$ in terms of already known data. We conclude that:
\begin{align*}
[{X_{+}}_{\lambda}K_{+}]=-3K_{+}\lambda+3\left(:G_{+}\Phi_{+}:-\partial K_{+}\right).
\end{align*}

\item[4)]$[{X_{+}}_{\lambda}M_{+}]$\\
We have $[{M_{+}}_{\lambda}X_{+}]=[{M_{+}}_{\lambda}\tfrac{1}{6}[{\Phi_{+}}_{\mu}\Phi_{+}]]$. Using the Jacobi identity (\ref{Jacobi identity}) we can express 
\begin{align*}
[{M_{+}}_{\lambda}[{\Phi_{+}}_{\mu}\Phi_{+}]]=(-1)[{\Phi_{+}}_{\mu}[{M_{+}}_{\lambda}\Phi_{+}]]+[[{M_{+}}_{\lambda}\Phi_{+}]_{\lambda+\mu}\Phi_{+}].
\end{align*}
Putting $\mu=0$ in the above equation we can compute $[{M_{+}}_{\lambda}X_{+}]$ in terms of already known data. We obtain:
\small
\begin{align*}
[{X_{+}}_{\lambda}M_{+}]=-\frac{9}{4}G_{+}\lambda^{2}-\left(5M_{+}+\frac{9}{4}\partial G_{+}\right)\lambda+\left(4:G_{+}X_{+}:-\frac{7}{2}\partial M_{+}-\frac{3}{4}\partial^{2}G_{+}\right).
\end{align*}
\normalsize

\item[5)]$[{M_{+}}_{\lambda}M_{+}]$\\
Using that $D$ is an odd derivation we have:
\begin{align*}
D\left({X_{+}}_{(n)}M_{+}\right)={M_{+}}_{(n)}M_{+}+{X_{+}}_{(n)}\partial X_{+}.
\end{align*}
Then ${M_{+}}_{(n)}M_{+}=D\left({X_{+}}_{(n)}M_{+}\right)-{X_{+}}_{(n)}\partial X_{+}$, that is, we can compute $[{M_{+}}_{\lambda}M_{+}]$ in terms of already known data. We conclude that:
\begin{align*}
[{M_{+}}_{\lambda}M_{+}]=&-\frac{35}{24}\lambda^{4}+\frac{1}{2}\left(20X_{+}-9L_{+}\right)\lambda^{2}+\left(10\partial X_{+}-\frac{9}{2}\partial L_{+}\right)\lambda\\
\quad&+\left(\frac{3}{2}\partial^{2}X_{+}-\frac{3}{2}\partial^{2}L_{+}-4:G_{+}M_{+}:+8:L_{+}X_{+}:\right).
\end{align*}

\end{itemize}

\subsubsection{Checking the relation}\label{subs-checkingtherelation}
To conlude the proof of the Conjecture \ref{conjecture} we need to prove the relation (\ref{assumption of the relation})
\begin{align*}
0=4:G_{+}X_{+}:-2:\Phi_{+} K_{+}:-4\partial M_{+}-\partial^{2}G_{+}.
\end{align*}
\noindent
used in the computation of the non-linear $\lambda$-brackets $[{\Phi_{+}}_{\lambda}M_{+}]$. We check this relation among the fields performing again an explicit computation. This is a long calculation, even longer than the ones already performed in the proof of Theorem~\ref{maintheorem}. Besides the \ttfamily{Mathematica} \normalfont package OPEdefs \cite{Thielemans91}, we have used the computer algebra system \ttfamily{Cadabra} \normalfont \cite{Cadabra07}, this last software was proved very useful for simplifying tensorial expressions with many terms. Despite its length this is a straightforward computation.\\
For these reasons we do not present here the computations, they are show up online at the URL:\\ 
\nolinkurl{http://www.ime.unicamp.br/~lazarord/checking_the_relation.pdf}
\begin{remark}\label{rm-Pontryagin}
As in Theorem~\ref{maintheorem} we make extensive use of: the Ricci flatness, the contractions on page~\pageref{ContractionsThree-FourForm}, the fact that $\varphi$, $\psi$ and $g$ are parallel and the symmetries of the Riemann curvature tensor. Unlike the computations in Theorem~\ref{maintheorem}, we should point out that we need to use the identity (\ref{identityvarphi-wedge-varphi}), in fact this one appears many times and is a key identity. It is also interesting to note the appearance of the first Pontryagin class $p_{1}(M)=\tfrac{1}{8\pi^{2}}\operatorname{Tr}\left(R\wedge R\right)$ of the manifold, it appears as one of the coefficients of the term $\partial\gamma^{i}\partial\gamma^{j}\partial\gamma^{k}c^{l}$, i.e, $\tfrac{1}{8\pi^{2}}\tensor{R}{_i_j_m_n}\tensor{R}{_k_l^m^n}\partial\gamma^{i}\partial\gamma^{j}\partial\gamma^{k}c^{l}$, nevertheless this expression is identically zero because $R$ is antisymmetric in $i$ and $j$ while $\partial\gamma^{i}$ commutes with $\partial\gamma^{j}$. 
\end{remark}

\section{Computations}\label{app-computations}  
To perform all the computations below we are assuming that we are working in a local coordinate chart where the volume form is constant, then we can assume that: 
\begin{equation}\label{localcoordinateassumption}
\Gamma^{i}_{ij}=\frac{\partial\log\sqrt{\abs{g}}}{\partial x^{j}}=0,
\end{equation}
\noindent
where $\abs{g}$ denotes the absolute value of the determinant of the metric tensor $g$.\\
For convenience of the reader we recall the expression of the Riemann curvature tensor $\tensor{R}{^l_i_j_k}$ and the Ricci tensor $R_{ij}$ in terms of the  Christoffel symbols. We also recall some of its symmetries.

\begin{equation}\label{RiemanncurvatureChristoffel}
\tensor{R}{^l_i_j_k}={(\Gamma^{l}_{ik})}_{,j}-{(\Gamma^{l}_{ij})}_{,k}+\Gamma^{l}_{js}\Gamma^{s}_{ik}-\Gamma^{l}_{ks}\Gamma^{s}_{ij},
\end{equation}

\begin{equation}\label{symmetriesoftheRiemannTensor}
\tensor{R}{_{iklm}}=\tensor{R}{_{lmik}},\;\;\;\tensor{R}{_{iklm}}=-\tensor{R}{_{kilm}}=-\tensor{R}{_{ikml}},
\end{equation}

\begin{equation}\label{firstBianchiidentity}
\tensor{R}{_{iklm}}+\tensor{R}{_{imkl}}+\tensor{R}{_{ilmk}}=0\;\;\; \text{(first Bianchi identity).}
\end{equation}

\begin{equation}\label{RiccitensorChristoffel}
R_{ij}=\tensor{R}{^l_i_l_k}=g^{lm}R_{iljm}={(\Gamma^{l}_{ij})}_{,l}-{(\Gamma^{l}_{il})}_{,j}+\Gamma^{l}_{ij}\Gamma^{m}_{lm}-\Gamma^{m}_{il}\Gamma^{l}_{jm}.
\end{equation}
\noindent
The Ricci tensor is symmetric.

\subsection{\texorpdfstring{$[{\Phi_{+}}_{\lambda}\Phi_{+}]$}{}}\label{PhiwithPhi}

We have
\begin{equation}
\Phi_{+}=\tfrac{1}{6}\varphi_{ijk}e^{i}_{+}e^{j}_{+}e^{k}_{+} + \tfrac{1}{2}\varphi_{ijk}\Gamma^{i}_{mn}g^{jm}\partial \gamma^{n}e^{k}_{+}. \nonumber
\end{equation}
To simplify the notation we drop the plus subscript.
\begin{itemize}
\item[a)]Computing\;\;\; $[\tfrac{1}{6}\varphi_{ijk}{e^{i}e^{j}e^{k}}_{\lambda}\tfrac{1}{6}\varphi_{mnl}e^{m}e^{n}e^{l}]$

For this we first compute $[{e^{m}e^{n}e^{l}}_{\lambda}e^{i}e^{j}e^{k}]$ and we can assume anti--symmetrization in the indices $\{i,j,k\}$ and $\{m,n,l\}$ due to the future contraction with the three form $\varphi$.
\small
\begin{align*}
[{e^{m}e^{n}e^{l}}_{\lambda}e^{i}e^{j}e^{k}]=&(-3)g^{im}g^{jn}g^{kl}\lambda^{2}+\left(18 g^{im}g^{jn}e^{k}e^{l}-6\partial(g^{im}g^{jn}g^{kl})\right)\lambda\nonumber\\
&+9g^{im}e^{j}e^{k}e^{n}e^{l}-18g^{im}g^{jn}\partial(e^{l})e^{k}+36\partial(g^{im})g^{jn}e^{k}e^{l}\nonumber\\
&-18\partial(g^{im})\partial(g^{jn})g^{kl}.\nonumber
\end{align*}
\normalsize
Then using the non-commutative Wick formula (\ref{Wickformula}):
\small
\begin{align*}
[\frac{1}{6}\varphi_{ijk}{e^{i}e^{j}e^{k}}_{\lambda}\frac{1}{6}\varphi_{mnl}e^{m}e^{n}e^{l}]=&\frac{1}{36}:\varphi_{mnl}[{e^{m}e^{n}e^{l}}_{-\lambda-\partial}\;\varphi_{ijk}e^{i}e^{j}e^{k}]:\\
=&\frac{1}{36}:\varphi_{mnl}{\left(:\varphi_{ijk}[{e^{m}e^{n}e^{l}}_{\mu}e^{i}e^{j}e^{k}]:\right)}_{\rvert_{-\lambda-\partial}}:.
\end{align*}
\normalsize
Collecting terms first by the order of $\lambda$ and then by the number of factors $e$'s:\\
\fbox{$\mathbf{\lambda^{2}}$}
\begin{align*}
\frac{1}{36}(-3)\varphi_{mnl}\varphi_{ijk}g^{im}g^{jn}g^{kl}=\frac{1}{36}(-3)42=-\frac{7}{2}.
\end{align*}
Here was used the identity (\ref{dosPhitresg}).\\

\fbox{$\mathbf{\lambda}$}
\begin{align*}
&\quad\frac{1}{36}\left(-6\varphi_{mnl}\partial(\varphi_{ijk}g^{im}g^{jn}g^{kl})-18\varphi_{ijk}\varphi_{mnl}g^{im}g^{jn}e^{k}e^{l}\right.\\
&\quad+\left.6\varphi_{mnl}\varphi_{ijk}\partial(g^{im}g^{jn}g^{kl})\right)\\
&=-\frac{1}{6}\varphi_{mnl}\partial(\varphi_{ijk})g^{im}g^{jn}g^{kl}-\frac{1}{2}\varphi_{ijk}\varphi_{mnl}g^{im}g^{jn}e^{k}e^{l}\\
&=-\frac{1}{6}\varphi^{ijk}\partial(\varphi_{ijk})-3g_{kl}e^{k}e^{l}\\
&=0.
\end{align*}
Here was used the identity (\ref{dosPhidosg}) and that $g_{ij}e^{i}e^{j}=-g_{ij}e^{j}e^{i}+g_{ij}\partial(g^{ij})$ by quasi-commutativity (\ref{quasi-commutativity}), then

\begin{equation}\label{trick1}
g_{ij}e^{i}e^{j}=\frac{1}{2}g_{ij}\partial(g^{ij})=\frac{1}{2}g_{ij}\left(-\Gamma^{i}_{ab}g^{aj}-\Gamma^{j}_{ab}g^{ia}\right)\partial\gamma^{b}=-\Gamma^{i}_{ib}\partial\gamma^{b}.
\end{equation}

We also have by (\ref{Phicovariantlyconstant}) and (\ref{dosPhidosg}) that

\begin{equation}\label{trick2}
\varphi^{ijk}\partial(\varphi_{ijk})=\varphi^{ijk}\left(\Gamma^{a}_{ib}\varphi_{ajk}+\Gamma^{a}_{jb}\varphi_{iak}+\Gamma^{a}_{kb}\varphi_{ija}\right)\partial\gamma^{b}=18\Gamma^{i}_{ib}\partial\gamma^{b}.
\end{equation}
\fbox{without $\lambda$}\\
\hspace{0.2in}\fbox{$e^{i}e^{j}e^{k}e^{l}$}
\small
\begin{align*}
\frac{9}{36}\varphi_{ijk}\varphi_{mnl}g^{im}e^{j}e^{k}e^{n}e^{l}=&\frac{1}{4}\left(-\psi_{jknl}+g_{jn}g_{kl}-g_{jl}g_{kn}\right)e^{j}e^{k}e^{n}e^{l}\nonumber\\
=&-\frac{1}{4}\psi_{ijkl}e^{i}e^{j}e^{k}e^{l}+\frac{1}{8}g_{ij}g_{kl}\partial(g^{ik})\partial(g^{jl})\nonumber\\
&-\frac{1}{8}g_{ij}g_{kl}\partial(g^{ij})\partial(g^{kl})\\
=&-\frac{1}{4}\psi_{ijkl}e^{i}e^{j}e^{k}e^{l}+\frac{1}{4}\Gamma^{c}_{ib}\Gamma^{i}_{cd}\partial\gamma^{b}\partial\gamma^{d}\nonumber\\
&+\frac{1}{4}g_{jm}g^{ic}\Gamma^{m}_{ib}\Gamma^{j}_{cd}\partial\gamma^{b}\partial\gamma^{d}-\frac{1}{8}g_{ij}g_{kl}\partial(g^{ij})\partial(g^{kl}).\nonumber
\end{align*}
\normalsize
Here was used the identity (\ref{dosPhiunag}), quasi-associativity (\ref{quasi-associativity}), the equality $g_{ij}e^{i}e^{j}=\frac{1}{2}g_{ij}\partial(g^{ij})$ proved above and that metric $g$ is covariantly constant.

\hspace{0.2in}\fbox{$\partial(e^{i})e^{j}$}
\begin{eqnarray}
&&\frac{1}{36}\left(-18\varphi_{mnl}\varphi_{ijk}g^{im}g^{jn}\partial(e^{k})e^{l}-18\varphi_{mnl}\varphi_{ijk}g^{im}g^{jn}e^{k}\partial(e^{l})\right.\nonumber\\
&&\left.-18\varphi_{mnl}\varphi_{ijk}g^{im}g^{jn}\partial(e^{l})e^{k}\right)\nonumber\\
&=&-3g_{lk}\partial(e^{l})e^{k}-\frac{3}{2}g_{lk}\partial^{2}g^{lk}.\nonumber
\end{eqnarray}
Here we used the identity (\ref{dosPhidosg}) and quasi-commutativity (\ref{quasi-commutativity}):
$$e^{i}\partial(e^{j})=-\partial(e^{j})e^{i}+\int_{-\partial}^{0}[{e^{i}}_{\lambda}\partial e^{j}]d\lambda=-\partial(e^{j})e^{i}+\frac{1}{2}\partial^{2}(g^{ij}).$$
\hspace{0.2in}\fbox{$e^{i}e^{j}$}
\small
\begin{align*}
&\quad-\frac{1}{2}\varphi_{mnl}\partial(\varphi_{ijk}g^{im}g^{jn})e^{k}e^{l}+\varphi_{mnl}\varphi_{ijk}\partial(g^{im})g^{jn}e^{k}e^{l}\\
&=-\frac{1}{2}\varphi_{mnl}\partial(\varphi_{ijk})g^{im}g^{jn}e^{k}e^{l}\\
&=-\frac{1}{2}\varphi_{mnl}\varphi_{ajk}\Gamma^{a}_{ib}g^{im}g^{jn}\partial\gamma^{b}e^{k}e^{l}-\frac{1}{2}\varphi_{mnl}\varphi_{iak}\Gamma^{a}_{jb}g^{im}g^{jn}\partial\gamma^{b}e^{k}e^{l}\\
&\quad-\frac{1}{2}\varphi_{mnl}\varphi_{ija}\Gamma^{a}_{kb}g^{im}g^{jn}\partial\gamma^{b}e^{k}e^{l}\\
&=-\varphi_{mnl}\varphi_{ajk}\Gamma^{a}_{ib}g^{im}g^{jn}\partial\gamma^{b}e^{k}e^{l}-\frac{1}{2}\varphi_{mnl}\varphi_{ija}\Gamma^{a}_{kb}g^{im}g^{jn}\partial\gamma^{b}e^{k}e^{l}\\
&=(-1)\left(g_{am}g_{lk}-g_{mk}g_{la}-\psi_{mlak}\right)\Gamma^{a}_{ib}g^{im}\partial\gamma^{b}e^{k}e^{l}+(-\frac{1}{2})6g_{la}\Gamma^{a}_{kb}\partial\gamma^{b}e^{k}e^{l}\\
&=\psi_{mlak}\Gamma^{a}_{ib}g^{im}\partial\gamma^{b}e^{k}e^{l}+(-1)g_{lk}\Gamma^{i}_{ib}\partial\gamma^{b}e^{k}e^{l}+(-2)g_{la}\Gamma^{a}_{kb}\partial\gamma^{b}e^{k}e^{l},\\
&=\psi_{mlak}\Gamma^{a}_{ib}g^{im}\partial\gamma^{b}e^{k}e^{l}+(-2)g_{la}\Gamma^{a}_{kb}\partial\gamma^{b}e^{k}e^{l}+(-\frac{1}{2})g_{lk}\partial(g^{lk})\Gamma^{i}_{ib}\partial\gamma^{b},\\
&=\psi_{mlak}\Gamma^{a}_{ib}g^{im}\partial\gamma^{b}e^{k}e^{l}+(-2)g_{la}\Gamma^{a}_{kb}\partial\gamma^{b}e^{k}e^{l}+\frac{1}{4}g_{ij}\partial(g^{ij})g_{lk}\partial(g^{lk}).
\end{align*}
\normalsize

Here was used that $d\varphi=0$, i.e., (\ref{Phicovariantlyconstant}), the identities (\ref{dosPhiunag}) and (\ref{dosPhidosg}), and (\ref{trick1}).

\hspace{0.2in}\fbox{without $e$'s}
\small
\begin{align*}
&-\frac{1}{12}\varphi_{mnl}\partial^{2}\left(\varphi_{ijk}g^{im}g^{jn}g^{kl}\right)-\frac{1}{2}\varphi_{mnl}\varphi_{ijk}\partial(g^{im})\partial(g^{jn})g^{kl}\\
&+\frac{1}{6}\varphi_{mnl}\partial\left(\varphi_{ijk}\partial\left(g^{im}g^{jn}g^{kl}\right)\right)\\
=&-\frac{1}{12}\varphi_{mnl}\partial^{2}(\varphi_{ijk})g^{im}g^{jn}g^{kl}+\frac{3}{2}g_{lk}\partial^{2}(g^{lk})\\
=&-\frac{1}{12}\varphi_{ijk}\partial^{2}(\varphi_{ijk})+\frac{3}{2}g_{lk}\partial^{2}(g^{lk})\\
=&\frac{9}{12}\partial\left(g_{ij}\partial(g^{ij})\right)+\frac{1}{12}\partial(\varphi^{ijk})\partial(\varphi_{ijk})+\frac{3}{2}g_{lk}\partial^{2}(g^{lk})\\
=&\frac{3}{4}\partial(g_{ij})\partial(g^{ij})+\frac{1}{12}\partial(\varphi^{ijk})\partial(\varphi_{ijk})+\frac{9}{4}g_{lk}\partial^{2}(g^{lk})\\
=&-\frac{3}{2}\Gamma^{c}_{bi}\Gamma^{i}_{cd}\partial\gamma^{b}\partial\gamma^{d}-\frac{3}{2}g_{aj}g^{ic}\Gamma^{a}_{bi}\Gamma^{j}_{cd}\partial\gamma^{b}\partial\gamma^{d}-\Gamma^{c}_{bi}\Gamma^{i}_{cd}\partial\gamma^{b}\partial\gamma^{d}\\
&-\frac{1}{2}\Gamma^{i}_{bi}\Gamma^{j}_{jd}\partial\gamma^{b}\partial\gamma^{d}-\frac{1}{2}\psi_{amjn}\Gamma^{a}_{bi}g^{mi}\Gamma^{i}_{cd}g^{nc}\partial\gamma^{b}\partial\gamma^{d}+\frac{9}{4}g_{lk}\partial^{2}(g^{lk})\\
=&-\frac{5}{2}\Gamma^{c}_{bi}\Gamma^{i}_{cd}\partial\gamma^{b}\partial\gamma^{d}-\frac{3}{2}g_{aj}g^{ic}\Gamma^{a}_{bi}\Gamma^{j}_{cd}\partial\gamma^{b}\partial\gamma^{d}-\frac{1}{2}\psi_{amjn}\Gamma^{a}_{bi}g^{mi}\Gamma^{j}_{cd}g^{nc}\partial\gamma^{b}\partial\gamma^{d}\\
=&-\frac{1}{8}g_{ij}\partial(g^{ij})g_{lk}\partial(g^{lk})+\frac{9}{4}g_{lk}\partial^{2}g^{lk}.
\end{align*}
\normalsize
In this chain of equalities was used that 
$$\partial\left(\varphi^{ijk}\partial(\varphi_{ijk})\right)=\partial(\varphi^{ijk})\partial(\varphi_{ijk})+\varphi^{ijk}\partial^{2}(\varphi_{ijk})$$
\noindent
and $\varphi^{ijk}\partial(\varphi_{ijk})=9g^{ij}\partial(g_{ij})$; this last identity is proved combining (\ref{trick1}) and (\ref{trick2}). We also used that $d\varphi=0$ i.e., (\ref{Phicovariantlyconstant}), that $g$ is covariantly constant and the identity (\ref{dosPhiunag}). 

\item[b)] Computing\\
\small $[\frac{1}{6}{\varphi_{ijk}e^{i}e^{j}e^{k}}_{\lambda}\frac{1}{2}\varphi_{mnl}\Gamma^{m}_{ab}g^{na}\partial\gamma^{b}e^{l}]+[\frac{1}{2}{\varphi_{ijk}\Gamma^{i}_{ab}g^{ja}\partial\gamma^{b}e^{k}}_{\lambda}\frac{1}{6}\varphi_{mnl}e^{m}e^{n}e^{l}]$. \normalsize
\begin{equation}
[\frac{1}{6}{\varphi_{ijk}e^{i}e^{j}e^{k}}_{\lambda}\frac{1}{2}\varphi_{mnl}\Gamma^{m}_{ab}g^{na}\partial\gamma^{b}e^{l}]=\frac{1}{4}\varphi_{ijk}\varphi_{mnl}g^{il}g^{na}\Gamma^{m}_{ab}\partial\gamma^{b}e^{j}e^{k}.\nonumber
\end{equation}
Using skew-symmetry (\ref{skewsymmetry}) and arranging the indices we have:
\begin{equation}
[\frac{1}{2}{\varphi_{ijk}\Gamma^{i}_{ab}g^{ja}\partial\gamma^{b}e^{k}}_{\lambda}\frac{1}{6}\varphi_{mnl}e^{m}e^{n}e^{l}]=[\frac{1}{6}{\varphi_{ijk}e^{i}e^{j}e^{k}}_{\lambda}\frac{1}{2}\varphi_{mnl}\Gamma^{m}_{ab}g^{na}\partial\gamma^{b}e^{l}],\nonumber
\end{equation}
then 
\begin{align*}
&[\frac{1}{6}{\varphi_{ijk}e^{i}e^{j}e^{k}}_{\lambda}\frac{1}{2}\varphi_{mnl}\Gamma^{m}_{ab}g^{na}\partial\gamma^{b}e^{l}]+[\frac{1}{2}{\varphi_{ijk}\Gamma^{i}_{ab}g^{ja}\partial\gamma^{b}e^{k}}_{\lambda}\frac{1}{6}\varphi_{mnl}e^{m}e^{n}e^{l}]\nonumber\\
=&\frac{1}{2}\varphi_{ijk}\varphi_{mnl}g^{il}g^{na}\Gamma^{m}_{ab}\partial\gamma^{b}e^{j}e^{k}\nonumber\\
=&\frac{1}{2}\left(g_{jm}g_{kn}-g_{jn}g_{km}-\psi_{jkmn}\right)g^{na}\Gamma^{m}_{ab}\partial\gamma^{b}e^{j}e^{k}\nonumber\\
=&\frac{1}{2}g_{jm}\Gamma^{m}_{kb}\partial\gamma^{b}e^{j}e^{k}+(-\frac{1}{2})g_{km}\Gamma^{m}_{jb}\partial\gamma^{b}e^{j}e^{k}+(-\frac{1}{2})\psi_{jkmn}g^{na}\Gamma^{m}_{ab}\partial\gamma^{b}e^{j}e^{k}\nonumber\\
=&(-\frac{1}{2})\psi_{jkmn}g^{na}\Gamma^{m}_{ab}\partial\gamma^{b}e^{j}e^{k}+(-1)g_{km}\Gamma^{m}_{jb}\partial\gamma^{b}e^{j}e^{k}+\frac{1}{2}g_{jm}\partial(g^{jk})\Gamma^{m}_{kb}\partial\gamma^{b}\nonumber\\
=&(-\frac{1}{2})\psi_{jkmn}g^{na}\Gamma^{m}_{ab}\partial\gamma^{b}e^{j}e^{k}+(-1)g_{km}\Gamma^{m}_{jb}\partial\gamma^{b}e^{j}e^{k}\nonumber\\
&+(-\frac{1}{2})g_{jm}g^{ck}\Gamma^{j}_{cd}\Gamma^{m}_{kd}\partial\gamma^{b}\partial\gamma^{d}+(-\frac{1}{2})\Gamma^{k}_{cd}\Gamma^{c}_{kb}\partial\gamma^{b}\partial\gamma^{d}.\nonumber
\end{align*}
Here were used the identity (\ref{dosPhiunag}), (\ref{trick1}) and that $g$ is covariantly constant.

\item[c)] Computing\;\;\; $[\frac{1}{2}{\varphi_{ijk}\Gamma^{i}_{rs}g^{jr}\partial\gamma^{s}e^{k}}_{\lambda}\frac{1}{2}\varphi_{mnl}\Gamma^{m}_{ab}g^{na}\partial\gamma^{b}e^{l}]$.
\begin{align*}
&[\frac{1}{2}{\varphi_{ijk}\Gamma^{i}_{rs}g^{jr}\partial\gamma^{s}e^{k}}_{\lambda}\frac{1}{2}\varphi_{mnl}\Gamma^{m}_{ab}g^{na}\partial\gamma^{b}e^{l}]\\
&=\frac{1}{4}\varphi_{mnl}\varphi_{ijk}g^{lk}\Gamma^{m}_{ab}g^{na}\partial\gamma^{b}\Gamma^{i}_{rs}g^{jr}\partial\gamma^{s}\\
&=\frac{1}{4}\left(g_{mi}g_{nj}-g_{mj}g_{ni}-\psi_{mnij}\right)\Gamma^{m}_{ab}g^{na}\partial\gamma^{b}\Gamma^{i}_{rs}g^{jr}\partial\gamma^{s}\\
&=-\frac{1}{4}\psi_{mnij}\Gamma^{m}_{ab}g^{na}\partial\gamma^{b}\Gamma^{i}_{rs}g^{jr}\partial\gamma^{s}+\frac{1}{4}g_{mi}g^{na}\Gamma^{m}_{ab}\Gamma^{i}_{ns}\partial\gamma^{b}\partial\gamma^{s}\\
&\quad-\frac{1}{4}\Gamma^{m}_{ib}\Gamma^{i}_{ms}\partial\gamma^{b}\partial\gamma^{s}.
\end{align*}
\end{itemize}
Finally we combining the result in $a),b)$ and $c)$ to obtain 
\begin{equation}
[{\Phi_{+}}_{\lambda}\Phi_{+}]=(-\tfrac{7}{2}) \lambda^{2}+6X_{+},
\end{equation}
where $X_{+}$ is:
\begin{eqnarray}
X_{+}&=&-\tfrac{1}{24}\psi_{ijkl}e^{i}_{+}e^{j}_{+}e^{k}_{+}e^{l}_{+}-\tfrac{1}{4}\psi_{ijkl}\Gamma^{i}_{mn}g^{jm}\partial\gamma^{n}e^{k}_{+}e^{l}_{+}\nonumber\\
&&-\tfrac{1}{8}\psi_{ijkl}\Gamma^{i}_{m_{1}n_{1}}g^{jm_{1}}\partial\gamma^{n_{1}}\Gamma^{k}_{m_{2}n_{2}}g^{lm_{2}}\partial\gamma^{n_{2}}-\tfrac{1}{2}g_{ij}\partial(e^{i}_{+})e^{j}_{+}\nonumber\\
&&-\tfrac{1}{2}g_{ij}\Gamma^{j}_{kl}\partial\gamma^{k}e^{l}_{+}e^{i}_{+}-\tfrac{1}{4}\Gamma^{i}_{jk}\Gamma^{k}_{il}\partial\gamma^{j}\partial\gamma^{l},\nonumber
\end{eqnarray}
\begin{eqnarray}
X_{-}&=&-\tfrac{1}{24}\psi_{ijkl}e^{i}_{-}e^{j}_{-}e^{k}_{-}e^{l}_{-}+\tfrac{1}{4}\psi_{ijkl}\Gamma^{i}_{mn}g^{jm}\partial\gamma^{n}e^{k}_{-}e^{l}_{-}\nonumber\\
&&-\tfrac{1}{8}\psi_{ijkl}\Gamma^{i}_{m_{1}n_{1}}g^{jm_{1}}\partial\gamma^{n_{1}}\Gamma^{k}_{m_{2}n_{2}}g^{lm_{2}}\partial\gamma^{n_{2}}+\tfrac{1}{2}g_{ij}\partial(e^{i}_{-})e^{j}_{-}\nonumber\\
&&+\tfrac{1}{2}g_{ij}\Gamma^{j}_{kl}\partial\gamma^{k}e^{l}_{-}e^{i}_{-}-\tfrac{1}{4}\Gamma^{i}_{jk}\Gamma^{k}_{il}\partial\gamma^{j}\partial\gamma^{l}.\nonumber
\end{eqnarray}
To get this precise form of $X_{+}$ we only need to manipulate a little more the terms without $e$'s that come from $a),b)$ and $c)$:
\small
\begin{align*}
&-\tfrac{3}{4}\psi_{amjn}\Gamma^{a}_{bi}g^{mi}\Gamma^{j}_{cd}g^{nc}\partial\gamma^{b}\partial\gamma^{d}-3\Gamma^{c}_{bi}\Gamma^{i}_{cd}\partial\gamma^{b}\partial\gamma^{d}-\tfrac{3}{2}g_{aj}g^{ic}\Gamma^{a}_{bi}\Gamma^{j}_{cd}\partial\gamma^{b}\partial\gamma^{d}\\
&+\tfrac{3}{4}g_{lk}\partial^{2}g^{lk}\nonumber\\
=&-\tfrac{3}{4}\psi_{amjn}\Gamma^{a}_{bi}g^{mi}\Gamma^{j}_{cd}g^{nc}\partial\gamma^{b}\partial\gamma^{d}-\frac{3}{2}\Gamma^{c}_{bi}\Gamma^{i}_{cd}\partial\gamma^{b}\partial\gamma^{d}+\frac{3}{4}\partial\left(g_{ij}\partial(g^{ij})\right)\nonumber\\
=&-\tfrac{3}{4}\psi_{amjn}\Gamma^{a}_{bi}g^{mi}\Gamma^{j}_{cd}g^{nc}\partial\gamma^{b}\partial\gamma^{d}-\frac{3}{2}\Gamma^{c}_{bi}\Gamma^{i}_{cd}\partial\gamma^{b}\partial\gamma^{d}.\nonumber
\end{align*}
\normalsize 
The first equality is proved using:
\begin{align*}
\partial(g_{ij}\partial(g^{ij}))=&\partial(g_{ij})\partial(g^{ij})+g_{ij}\partial^{2}(g^{ij})\\
=&(-2)\Gamma^{c}_{bi}\Gamma^{i}_{cd}\partial\gamma^{b}\partial\gamma^{d}+(-2)g_{aj}g^{ic}\Gamma^{a}_{bi}\Gamma^{j}_{cd}\partial\gamma^{b}\partial\gamma^{d}+g_{ij}\partial^{2}(g^{ij}),
\end{align*}
\noindent
which itself is proved using that $g$ is covariantly constant. The second equality follows from (\ref{trick1}) and (\ref{localcoordinateassumption}).

\subsection{\texorpdfstring{${X_{+}}_{(2)}\Phi_{+}$}{} and \texorpdfstring{${X_{+}}_{(1)}\Phi_{+}$}{}}\label{XwithPhi}

To simplify the notation we drop the ``+'' subscript.\\
Below we only compute the $\lambda$-bracket between the sumands of $X$ and $\Phi$ contributing to ${X}_{(2)}\Phi$ and ${X}_{(1)}\Phi$.
\begin{itemize}
\item[a)] Computing\;\;\; $[-\frac{1}{24}{\psi_{abcd}e^{a}e^{b}e^{c}e^{d}}_{\lambda}\frac{1}{6}\varphi_{ijk}e^{i}e^{j}e^{k}]$.
\begin{align*}
&[-\frac{1}{24}{\psi_{abcd}e^{a}e^{b}e^{c}e^{d}}_{\lambda}\frac{1}{6}\varphi_{ijk}e^{i}e^{j}e^{k}]\nonumber\\
=&\left(-\frac{1}{6}\psi_{abcd}\varphi_{ijk}g^{ai}g^{bj}g^{ck}e^{d}\right)\lambda^{2}+\left(-\frac{1}{4}\psi_{abcd}\varphi_{ijk}g^{bi}g^{dj}e^{a}e^{c}e^{k}\right.\nonumber\\
&\left.-\frac{3}{8}\psi_{abcd}\varphi_{ijk}g^{bi}g^{cj}\partial(g^{dk})e^{a}+\frac{1}{6}\varphi_{ijk}\partial(\psi_{abcd}g^{bi}g^{cj}g^{dk}e^{a})\right)\lambda\\
&+\text{terms without} \; \lambda\nonumber\\
=&\left(-\varphi_{ack}e^{a}e^{c}e^{k}+\frac{1}{6}\varphi_{ijk}\partial(\psi_{abcd})g^{bi}g^{cj}g^{dk}e^{a}\right)\lambda+ \text{terms without} \; \lambda\nonumber\\
=&\left(-\varphi_{ijk}e^{i}e^{j}e^{k}+(-2)\varphi_{ijk}\Gamma^{i}_{lm}g^{lj}\partial\gamma^{m}e^{k}\right)\lambda+ \text{terms without} \; \lambda.\nonumber
\end{align*}
Here were used the identities (\ref{PhiPsitresg}) and (\ref{PhiPsidosg}), and that $d*\varphi=0$, i.e., (\ref{Psicovariantlyconstant}).

\item[b)] Computing\;\;\; $[-\frac{1}{4}{\psi_{abcd}\Gamma^{a}_{mn}g^{bm}\partial\gamma^{n}e^{c}e^{d}}_{\lambda}\frac{1}{6}\varphi_{ijk}e^{i}e^{j}e^{k}]$.
\begin{eqnarray}
&&[-\frac{1}{4}{\psi_{abcd}\Gamma^{a}_{mn}g^{bm}\partial\gamma^{n}e^{c}e^{d}}_{\lambda}\frac{1}{6}\varphi_{ijk}e^{i}e^{j}e^{k}]\nonumber\\
&=&\left(\frac{1}{4}\varphi_{ijk}\psi_{abcd}\Gamma^{a}_{mn}g^{bm}\partial\gamma^{n}g^{ci}g^{dj}e^{k}\right)\lambda + \text{terms without} \; \lambda\nonumber\\
&=&\left(-\varphi_{abk}\Gamma^{a}_{mn}g^{bm}\partial\gamma^{n}e^{k}\right)+ \text{terms without} \; \lambda.\nonumber
\end{eqnarray}
Here was used the identity (\ref{PhiPsidosg}).

\item[c)] Computing\;\;\; $[-\frac{1}{2}{g_{lb}\partial(e^{b})e^{l}}_{\lambda}\frac{1}{6}\varphi_{ijk}e^{i}e^{j}e^{k}]$.
\small
\begin{align*}
&[-\frac{1}{2}{g_{lb}\partial(e^{b})e^{l}}_{\lambda}\frac{1}{6}\varphi_{ijk}e^{i}e^{j}e^{k}]\nonumber\\
=&\left((-\frac{1}{2})\varphi_{ijk}g^{ij}e^{k}\right)\lambda^{2}+\left((-\frac{1}{4})\varphi_{ijk}e^{i}e^{j}e^{k}+(-\frac{1}{2})\varphi_{ijk}g_{lb}\partial(g^{bi}g^{kl}e^{j})\right.\nonumber\\
&\left.+\frac{1}{2}\varphi_{ijk}\partial(g^{ik}e^{j})\right)\lambda + \text{terms without} \; \lambda\nonumber\\
=&\left((-\frac{1}{4})\varphi_{ijk}e^{i}e^{j}e^{k}\right)\lambda+ \text{terms without} \; \lambda.\nonumber
\end{align*}
\normalsize
\item[d)] Computing\;\;\; $[-\frac{1}{2}{g_{lb}\partial(e^{b})e^{l}}_{\lambda}\frac{1}{2}\varphi_{ijk}\Gamma^{i}_{mn}g^{jm}\partial\gamma^{n}e^{k}]$.
\begin{eqnarray}
&&[-\frac{1}{2}{g_{lb}\partial(e^{b})e^{l}}_{\lambda}\frac{1}{2}\varphi_{ijk}\Gamma^{i}_{mn}g^{jm}\partial\gamma^{n}e^{k}]\nonumber\\
&=&\left(-\frac{1}{4}\varphi_{ijk}\Gamma^{i}_{mn}g^{jm}\partial\gamma^{n}e^{k}\right)\lambda + \text{terms without} \; \lambda.\nonumber
\end{eqnarray}

\item[e)] Computing\;\;\; $[-\frac{1}{2}{g_{am}\Gamma^{a}_{nb}\partial\gamma^{b}e^{n}e^{m}}_{\lambda}\frac{1}{6}\varphi_{ijk}e^{i}e^{j}e^{k}]$.
\begin{eqnarray}
&&[-\frac{1}{2}{g_{am}\Gamma^{a}_{nb}\partial\gamma^{b}e^{n}e^{m}}_{\lambda}\frac{1}{6}\varphi_{ijk}e^{i}e^{j}e^{k}]\nonumber\\
&=&\left(-\frac{1}{2}\varphi_{ijk}\Gamma^{i}_{nb}g^{nj}\partial\gamma^{b}e^{k}\right)\lambda + \text{terms without} \; \lambda.\nonumber
\end{eqnarray}

\end{itemize}

\noindent
Combining $a),b),c),d)$ and $e)$ we get that ${X_{+}}_{(2)}\Phi_{+}=0$ and ${X_{+}}_{(1)}\Phi_{+}=-\frac{15}{2}\Phi_{+}$.

\subsection{\texorpdfstring{$[{\Phi_{\pm}}_{\lambda}K_{\pm}]$}{}}\label{Phi-lambda-K}

To perform this computation we express $\Phi_{\pm}$ explicitly in terms of the $bc-\beta\gamma$ system.
\begin{align*}
\Phi_{+}=&\tfrac{1}{12\sqrt{2}}\varphi_{ijk} c^{i}c^{j}c^{k}+ \tfrac{1}{4\sqrt{2}}\varphi_{ijk}g^{il}c^{j}c^{k}b_{l} + \tfrac{1}{4\sqrt{2}}\varphi_{ijk} g^{il}g^{jm}c^{k} b_{l}b_{m}\nonumber\\
&+ \tfrac{1}{12\sqrt{2}}\varphi_{ijk} g^{il}g^{jm}g^{kn}b_{l}b_{m}b_{n}+\tfrac{1}{2\sqrt{2}}\varphi_{ijk}\Gamma^{i}_{mn}g^{jm}\partial \gamma^{n}g^{kl}b_{l}\\
&+\tfrac{1}{2\sqrt{2}}\varphi_{ijk}\Gamma^{i}_{mn}g^{jm}\partial \gamma^{n}c^{k},\nonumber
\end{align*}
\begin{align*}
\Phi_{-}=&\tfrac{i}{12\sqrt{2}}\varphi_{ijk} c^{i}c^{j}c^{k}- \tfrac{i}{4\sqrt{2}}\varphi_{ijk}g^{il}c^{j}c^{k}b_{l}+\tfrac{i}{4\sqrt{2}}\varphi_{ijk} g^{il}g^{jm}c^{k} b_{l}b_{m}\nonumber\\
&-\tfrac{i}{12\sqrt{2}}\varphi_{ijk} g^{il}g^{jm}g^{kn}b_{l}b_{m}b_{n}+\tfrac{i}{2\sqrt{2}}\varphi_{ijk}\Gamma^{i}_{mn}g^{jm}\partial \gamma^{n}g^{kl}b_{l}\\
&-\tfrac{i}{2\sqrt{2}}\varphi_{ijk}\Gamma^{i}_{mn}g^{jm}\partial \gamma^{n}c^{k}.\nonumber
\end{align*}

As by definition $K_{\pm}:=G_{(0)}(\Phi_{\pm})$ we need to apply $G_{(0)}$ to each of the summands of $\Phi_{\pm}$. To simplify the notation we denote $G_{(0)}$ simply by $D$.

\begin{align*}
D(\varphi_{ijk} c^{i}c^{j}c^{k})=(\varphi_{ijk,l}c^{l}) c^{i}c^{j}c^{k} + 3 \varphi_{ijk}\partial \gamma^{i}c^{j}c^{k},\nonumber
\end{align*}
\begin{align*}
D(\varphi_{ijk}g^{il}c^{j}c^{k}b_{l})=&D(\tensor{\varphi}{^l_j_k}c^{j}c^{k}b_{l})\\
=&(\tensor{\varphi}{^l_j_k_{,m}}c^{m})c^{j}c^{k}b_{l} + 2 \tensor{\varphi}{^l_j_k}\partial \gamma^{j}c^{k}b_{l}+\tensor{\varphi}{^l_j_k}c^{j}c^{k}\beta_{l},\nonumber
\end{align*}
\begin{align*}
D(\varphi_{ijk} g^{il}g^{jm}c^{k} b_{l}b_{m})=&D(\tensor{\varphi}{^l^m_k}c^{k} b_{l}b_{m})\\
=&(\tensor{\varphi}{^l^m_k_{,n}}c^{n})c^{k}b_{l}b_{m}+\tensor{\varphi}{^l^m_k}\partial \gamma^{k}b_{l}b_{m}+2\tensor{\varphi}{^l^m_k}c^{k}b_{l}\beta_{m},\nonumber
\end{align*}
\begin{align*}
D(\varphi_{ijk} g^{il}g^{jm}g^{kn}b_{l}b_{m}b_{n})=&D(\tensor{\varphi}{^l^m^n}b_{l}b_{m}b_{n})\\
=&(\tensor{\varphi}{^l^m^n_{,a}}c^{a})b_{l}b_{m}b_{n}+3\tensor{\varphi}{^l^m^n}b_{l}b_{m}\beta_{n},\nonumber
\end{align*}
\begin{align*}
D(\varphi_{ijk}\Gamma^{i}_{mn}g^{jm}\partial \gamma^{n}g^{kl}b_{l})=&D(\varphi_{ijk}\Gamma^{i}_{mn}g^{jm}g^{kl}\partial \gamma^{n}b_{l})\\
=&(F^{l}_{n,a}c^{a})\partial \gamma^{n}b_{l}+ F^{l}_{n}\partial c^{n}b_{l}+ F^{l}_{n}\partial\gamma^{n}\beta_{l},\nonumber
\end{align*}
\noindent
where  $F^{l}_{n}:=\varphi_{ijk}\Gamma^{i}_{mn}g^{jm}g^{kl}$,
\begin{align*}
D(\varphi_{ijk}\Gamma^{i}_{mn}g^{jm}\partial \gamma^{n}c^{k})=(F_{kn,a}c^{a})\partial \gamma^{n}c^{k}+ F_{kn}\partial c^{n}c^{k}+ F_{kn}\partial \gamma^{n}\partial\gamma^{k},\nonumber
\end{align*}
\noindent
where  $F_{kn}:=\varphi_{ijk}\Gamma^{i}_{mn}g^{jm}$.\\
Now we collect the non-zero $\lambda$-bracket between the summands of $\Phi_{\pm}$ and the summands of $K_{\pm}$ that contains terms with $\lambda$. To compute this $\lambda$-brackets we used the $\mathtt{Mathematica}$ package \cite{Thielemans91}:
\begin{align*}
&[{\varphi_{ijk} c^{i}c^{j}c^{k}}_{\lambda} D(\varphi_{ijk} g^{il}g^{jm}c^{k} b_{l}b_{m})]\\
=&\left(6\varphi_{ijk}\tensor{\varphi}{^j^k_n_{,s}}c^{i}c^{n}c^{s}+6\varphi_{ijk}\tensor{\varphi}{^k^j_n}\partial\gamma^{n}c^{i}\right.\\
&\left.+6\tensor{\varphi}{^k^m_n}\varphi_{ijk,m}c^{i}c^{j}c^{n}\right)\lambda+\text{terms without $\lambda$},
\end{align*}
\begin{align*}
&[{\varphi_{ijk} c^{i}c^{j}c^{k}}_{\lambda} D(\varphi_{ijk} g^{il}g^{jm}g^{kn}b_{l}b_{m}b_{n})]\\
=&\left(6\varphi_{ijk}\tensor{\varphi}{^{ijk}_{,s}}c^{s}+18\varphi^{jkn}\varphi_{ijk,n}c^{i}\right)\frac{\lambda^{2}}{2}\\
&+\left((-18)\varphi_{ijk}\tensor{\varphi}{^{ljk}_{,s}}c^{i}c^{s}b_{l}+(-18)\varphi_{ijk}\partial\tensor{\varphi}{^{sjk}_{,s}}c^{i}\right.\\
&\left.+6\partial\varphi_{ijk}\tensor{\varphi}{^{ijk}_{,s}}c^{s}+(-18)\varphi_{ijk}\varphi^{jkn}c^{i}\beta_{n}\right.\\
&\left.+18\varphi^{lkn}\varphi_{ijk,n}c^{i}c^{j}b_{l}+18\varphi^{jkn}\varphi_{ijk,n}\partial c^{i}\right)\lambda\\
&+\text{terms without $\lambda$},
\end{align*}
\begin{align*}
&[{\varphi_{ijk}g^{il}c^{j}c^{k}b_{l}}_{\lambda} D(\varphi_{ijk}g^{il}c^{j}c^{k}b_{l})]\\
=&\left(4\tensor{\varphi}{^{l}_{mi}}\tensor{\varphi}{^{i}_{jk,l}}c^{j}c^{k}c^{m}+4\tensor{\varphi}{^{l}_{mi}}\tensor{\varphi}{^{i}_{jl,s}}c^{j}c^{m}c^{s}+(-4)\tensor{\varphi}{^{i}_{jk}}\tensor{\varphi}{^{k}_{in}}\partial\gamma^{j}c^{n}\right)\lambda\\
&+\text{terms without $\lambda$},
\end{align*}
\begin{align*}
&[{\varphi_{ijk}g^{il}c^{j}c^{k}b_{l}}_{\lambda} D(\varphi_{ijk} g^{il}g^{jm}c^{k} b_{l}b_{m})]\\
=&\left(2\tensor{\varphi}{^{l}_{ji}}\tensor{\varphi}{^{ij}_{k,l}}c^{k}+(-2)\tensor{\varphi}{^{l}_{ji}}\tensor{\varphi}{^{ij}_{l,s}}c^{s}+(-4)\tensor{\varphi}{^{ij}_{k}}\tensor{\varphi}{^{k}_{mi,j}}c^{m}\right)\frac{\lambda^{2}}{2}\\
&+\left((-4)\tensor{\varphi}{^{l}_{mj}}\tensor{\varphi}{^{ij}_{k,l}}c^{k}c^{m}b_{i}+(-4)\tensor{\varphi}{^{k}_{mj}}\tensor{\varphi}{^{ij}_{k,s}}c^{m}c^{s}b_{i}\right.\\
&\left.+(-4)\tensor{\varphi}{^{k}_{mj}}\partial(\tensor{\varphi}{^{ij}_{k,i}})c^{m}+(-2)\tensor{\varphi}{^{l}_{ji}}\tensor{\varphi}{^{ij}_{k,s}}c^{k}c^{s}b_{l}\right.\\
&\left.+2\partial(\tensor{\varphi}{^{l}_{ji}})\tensor{\varphi}{^{ij}_{k,l}}c^{k}+(-2)\partial(\tensor{\varphi}{^{k}_{ji}})\tensor{\varphi}{^{ij}_{k,s}}c^{s}+2\tensor{\varphi}{^{ij}_{k}}\tensor{\varphi}{^{l}_{ji}}\partial\gamma^{k}b_{l}\right.\\
&\left.+4\tensor{\varphi}{^{ij}_{k}}\tensor{\varphi}{^{k}_{mi}}c^{m}\beta_{j}+(-4)\tensor{\varphi}{^{ij}_{k}}\tensor{\varphi}{^{l}_{mi,j}}c^{k}c^{m}b_{l}\right.\\
&\left.+(-2)\tensor{\varphi}{^{ij}_{k}}\tensor{\varphi}{^{k}_{mn,j}}c^{m}c^{n}b_{i}+(-4)\tensor{\varphi}{^{ij}_{k}}\tensor{\varphi}{^{k}_{mi,j}}\partial c^{m}\right)\lambda\\
&+\text{terms without $\lambda$},
\end{align*}
\begin{align*}
&[{\varphi_{ijk}g^{il}c^{j}c^{k}b_{l}}_{\lambda}D(\varphi_{ijk} g^{il}g^{jm}g^{kn}b_{l}b_{m}b_{n})]\\
=&\left(6\tensor{\varphi}{^{l}_{mk}}\tensor{\varphi}{^{ijk}_{,l}}c^{m}b_{i}b_{j}+6\tensor{\varphi}{^{l}_{kj}}\tensor{\varphi}{^{ijk}_{,s}}c^{s}b_{i}b_{l}+6\tensor{\varphi}{^{l}_{kj}}\partial(\tensor{\varphi}{^{ijk}_{,i}})b_{l}\right.\\
&\left.+6\partial(\tensor{\varphi}{^{l}_{kj}})\tensor{\varphi}{^{ijk}_{,l}}b_{i}+(-6)\tensor{\varphi}{^{ijk}}\tensor{\varphi}{^{l}_{ij}}b_{l}\beta_{k}\right.\\
&\left.+12\varphi^{ijk}\tensor{\varphi}{^{l}_{mj,k}}c^{m}b_{i}b_{l}+6\tensor{\varphi}{^{ijk}}\tensor{\varphi}{^{l}_{ij,k}}\partial b_{l}\right)\lambda\\
&+\text{terms without $\lambda$},
\end{align*}
\begin{align*}
&[{\varphi_{ijk}g^{il}c^{j}c^{k}b_{l}}_{\lambda}D(\varphi_{ijk}\Gamma^{i}_{mn}g^{jm}\partial \gamma^{n}g^{kl}b_{l})]\\
=&\left((-2)F^{i}_{j}\tensor{\varphi}{^{j}_{in}}c^{n}\right)\frac{\lambda^{2}}{2} + \left( (-2)\tensor{\varphi}{^{k}_{in}}F^{i}_{j,k}\partial \gamma^{j}c^{n}\right.\\
&\left.+ F^{i}_{j}\tensor{\varphi}{^{j}_{mn}}c^{m}c^{n}b_{i}+ (-2)F^{i}_{j}\tensor{\varphi}{^{j}_{in}}\partial c^{n} + (-2)F^{i}_{j}\partial \tensor{\varphi}{^{j}_{in}}c^{n}\right)\lambda\\
&+\text{terms without $\lambda$},
\end{align*}
\begin{align*}
&[{\varphi_{ijk}g^{il}c^{j}c^{k}b_{l}}_{\lambda}D(\varphi_{ijk}\Gamma^{i}_{mn}g^{jm}\partial \gamma^{n}c^{k})]\\
=&\left(F_{ji}\tensor{\varphi}{^{i}_{mn}}c^{j}c^{m}c^{n}\right)\lambda + \text{terms without $\lambda$},
\end{align*}
\begin{align*}
&[{\varphi_{ijk} g^{il}g^{jm}c^{k} b_{l}b_{m}}_{\lambda}D(\varphi_{ijk} c^{i}c^{j}c^{k})]\\
=&\left(-6\tensor{\varphi}{^{ij}_{k}}\varphi_{lmj,i}c^{k}c^{l}c^{m}-6\tensor{\varphi}{^{ij}_{k}}\varphi_{lji,s}c^{k}c^{l}c^{s}+6\tensor{\varphi}{^{ij}_{k}}\varphi_{lji}\partial\gamma^{l}c^{k}\right)\lambda\nonumber\\
&+\text{terms without $\lambda$},\nonumber
\end{align*}
\begin{align*}
&[{\varphi_{ijk} g^{il}g^{jm}c^{k} b_{l}b_{m}} _{\lambda}D(\varphi_{ijk}g^{il}c^{j}c^{k}b_{l})]\\
=&\left((-4)\tensor{\varphi}{^{ij}_{k}}\tensor{\varphi}{^{k}_{mj,i}}c^{m}+2\tensor{\varphi}{^{ij}_{k}}\tensor{\varphi}{^{k}_{ij,s}}c^{s}+2\tensor{\varphi}{^{l}_{ij}}\tensor{\varphi}{^{ij}_{k,l}}c^{k}\right) \frac{\lambda^{2}}{2}\\
&+\left(4\tensor{\varphi}{^{ij}_{k}}\tensor{\varphi}{^{l}_{mj,i}}c^{k}c^{m}b_{l}+ 2\tensor{\varphi}{^{ij}_{k}}\tensor{\varphi}{^{l}_{ji,s}}c^{k}c^{s}b_{l}\right.\\
&\left.+ (-2)\tensor{\varphi}{^{ij}_{k}}\tensor{\varphi}{^{k}_{mn,j}}c^{m}c^{n}b_{i}+4\tensor{\varphi}{^{ij}_{k}}\tensor{\varphi}{^{k}_{mj,s}}c^{m}c^{s}b_{i}\right.\\
&\left.+2\tensor{\varphi}{^{ij}_{k}}\partial(\tensor{\varphi}{^{l}_{ji,l}})c^{k}+ 2\partial(\tensor{\varphi}{^{ij}_{k}})\tensor{\varphi}{^{k}_{ij,s}}c^{s}-2\partial(\tensor{\varphi}{^{ij}_{k}})\tensor{\varphi}{^{k}_{mj,i}}c^{m}\right.\\
&\left.+ 2\partial(\tensor{\varphi}{^{ij}_{k}})\tensor{\varphi}{^{k}_{jn,i}}c^{n}+(-4)\tensor{\varphi}{^{ij}_{k}}\tensor{\varphi}{^{k}_{mj}}\partial\gamma^{m}b_{i}\right.\\
&\left.+2\tensor{\varphi}{^{ij}_{k}}\tensor{\varphi}{^{l}_{ji}}c^{k}\beta_{l}+(-4)\tensor{\varphi}{^{l}_{mj}}\tensor{\varphi}{^{ij}_{k,l}}c^{k}c^{m}b_{i}+2\tensor{\varphi}{^{l}_{ij}}\tensor{\varphi}{^{ij}_{k,l}}\partial
 c^{k}\right)\lambda\\
&+\text{terms without $\lambda$},
\end{align*}
\begin{align*}
&[{\varphi_{ijk} g^{il}g^{jm}c^{k} b_{l}b_{m}}_{\lambda}D(\varphi_{ijk} g^{il}g^{jm}c^{k} b_{l}b_{m})]\\
=&\left(4\tensor{\varphi}{^{km}_{n}}\tensor{\varphi}{^{ij}_{k,m}}c^{n}b_{i}b_{j}+8\tensor{\varphi}{^{ij}_{k}}\tensor{\varphi}{^{lk}_{n,j}}c^{n}b_{i}b_{l}\right.\\
&\left.-4\tensor{\varphi}{^{ij}_{k}}\tensor{\varphi}{^{lk}_{j,s}}c^{s}b_{i}b_{l}+4\tensor{\varphi}{^{ij}_{k}}\tensor{\varphi}{^{km}_{i}}b_{j}\beta_{m}\right.\\
&\left.+4\tensor{\varphi}{^{ij}_{k}}\partial(\tensor{\varphi}{^{lk}_{j,l}})b_{i}+4\partial(\tensor{\varphi}{^{ij}_{k}})\tensor{\varphi}{^{lk}_{j,i}}b_{l}\right.\\
&\left.+4\tensor{\varphi}{^{km}_{j}}\tensor{\varphi}{^{ij}_{k,m}}\partial b_{i}\right)\lambda+ \text{terms without $\lambda$},
\end{align*}
\begin{align*}
&[{\varphi_{ijk} g^{il}g^{jm}c^{k} b_{l}b_{m}}_{\lambda}D(\varphi_{ijk} g^{il}g^{jm}g^{kn}b_{l}b_{m}b_{n})]\\
=&\left(12\tensor{\varphi}{^{ij}_{k}}\tensor{\varphi}{^{lmk}_{,i}}b_{j}b_{l}b_{m}\right)\lambda + \text{terms without $\lambda$},
\end{align*}
\begin{align*}
&[{\varphi_{ijk} g^{il}g^{jm}c^{k} b_{l}b_{m}}_{\lambda}D(\varphi_{ijk}\Gamma^{i}_{mn}g^{jm}\partial \gamma^{n}g^{kl}b_{l})]\\
=&\left(2\tensor{\varphi}{^{ij}_{k}}F^{k}_{i}b_{j}\right)\frac{\lambda^{2}}{2}+\left(2\tensor{\varphi}{^{ij}_{k}}F^{k}_{m,i}\partial \gamma^{m}b_{j}\right.\\
&\left.+2\tensor{\varphi}{^{ij}_{k}}F^{l}_{j}c^{k}b_{i}b_{l}+2\tensor{\varphi}{^{ij}_{k}}F^{k}_{i}\partial b_{j}+2\partial(\tensor{\varphi}{^{ij}_{k}})F^{k}_{i}b_{j}\right)\lambda\\
&+\text{terms without $\lambda$},
\end{align*}
\begin{align*}
&[{\varphi_{ijk} g^{il}g^{jm}c^{k} b_{l}b_{m}}_{\lambda}D(\varphi_{ijk}\Gamma^{i}_{mn}g^{jm}\partial \gamma^{n}c^{k})]\\
=&\left( 2\tensor{\varphi}{^{ij}_{k}}F_{ij}c^{k}\right)\frac{\lambda^{2}}{2}+\left(2\tensor{\varphi}{^{ij}_{k}}F_{il,j}\partial\gamma^{l}c^{k}+2\tensor{\varphi}{^{ij}_{k}}F_{mi}c^{k}c^{m}b_{j}\right.\\
&\left.+2\tensor{\varphi}{^{ij}_{k}}F_{ij}\partial c^{k}+ 2\partial(\tensor{\varphi}{^{ij}_{k}})F_{ij}c^{k}\right)\lambda+ \text{terms without $\lambda$},
\end{align*}
\begin{align*}
&[{\varphi_{ijk} g^{il}g^{jm}g^{kn}b_{l}b_{m}b_{n}}_{\lambda}D(\varphi_{ijk} c^{i}c^{j}c^{k})]\\
=&\left(18\varphi^{ijk}\varphi_{lkj,i}c^{l}+6\varphi^{ijk}\varphi_{ijk,s}c^{s}\right)\frac{\lambda^{2}}{2}\\
&+\left(18\varphi^{ijk}\varphi_{lmj,k}c^{l}c^{m}b_{i}+18\varphi^{ijk}\varphi_{ljk,s}c^{l}c^{s}b_{i}\right.\\
&\left.+18\partial(\varphi^{ijk})\varphi_{lkj,i}c^{l}+6\partial(\varphi^{ijk})\varphi_{ijk,s}c^{s}\right.\\
&\left.+18\varphi^{ijk}\varphi_{lkj}\partial\gamma^{l}b_{i}\right)\lambda+\text{terms without $\lambda$},
\end{align*}
\begin{align*}
&[{\varphi_{ijk} g^{il}g^{jm}g^{kn}b_{l}b_{m}b_{n}}_{\lambda}D(\varphi_{ijk}g^{il}c^{j}c^{k}b_{l})]\\
=&\left(12\varphi^{ijk}\tensor{\varphi}{^{l}_{mj,k}}c^{m}b_{i}b_{l}+6\tensor{\varphi}{^{ijk}}\tensor{\varphi}{^{l}_{jk,s}}c^{s}b_{i}b_{l}\right.\\
&\left.+6\varphi^{ijk}\partial(\tensor{\varphi}{^{l}_{kj,l}})b_{i}+6\partial(\tensor{\varphi}{^{ijk}})\tensor{\varphi}{^{l}_{kj,i}}b_{l}+6\varphi^{ijk}\tensor{\varphi}{^{l}_{kj}}b_{i}\beta_{l}\right.\\
&\left.+6\tensor{\varphi}{^{l}_{mk}}\tensor{\varphi}{^{ijk}_{,l}}c^{m}b_{i}b_{j}+6\tensor{\varphi}{^{l}_{jk}}\tensor{\varphi}{^{ijk}_{,l}}\partial b_{i}\right)\lambda\\
&+\text{terms without $\lambda$},
\end{align*}
\begin{align*}
&[{\varphi_{ijk} g^{il}g^{jm}g^{kn}b_{l}b_{m}b_{n}}_{\lambda}D(\varphi_{ijk} g^{il}g^{jm}c^{k} b_{l}b_{m})]\\
=&\left(12\tensor{\varphi}{^{ijk}}\tensor{\varphi}{^{lm}_{i,j}}b_{k}b_{l}b_{m}\right)\lambda+ \text{terms without $\lambda$},
\end{align*}
\begin{align*}
&[{\varphi_{ijk} g^{il}g^{jm}g^{kn}b_{l}b_{m}b_{n}}_{\lambda}D(\varphi_{ijk}\Gamma^{i}_{mn}g^{jm}\partial \gamma^{n}g^{kl}b_{l})]\\
=&\left(3\varphi^{ijk}F^{l}_{i}b_{j}b_{k}b_{l}\right)\lambda + \text{terms without $\lambda$}, 
\end{align*}
\begin{align*}
&[{\varphi_{ijk} g^{il}g^{jm}g^{kn}b_{l}b_{m}b_{n}}_{\lambda}D(\varphi_{ijk}\Gamma^{i}_{mn}g^{jm}\partial \gamma^{n}c^{k})]\\
=&\left(6\varphi^{ijk}F_{ij}b_{k}\right)\frac{\lambda^{2}}{2} +\left( 6\varphi^{ijk}F_{il,j}\partial\gamma^{l}b_{k}\right.\\
&\left.+3\varphi^{ijk}F_{mi}c^{m}b_{j}b_{k}+6\varphi^{ijk}F_{ij}\partial b_{k}+6\partial(\varphi^{ijk})F_{ij}b_{k}\right)\lambda\\
&+\text{terms without $\lambda$},
\end{align*}
\begin{align*}
&[{F^{l}_{m}\partial\gamma^{m}b_{l}}_{\lambda}D(\varphi_{ijk}g^{il}c^{j}c^{k}b_{l})]\\
=&\left(2F^{j}_{i}\tensor{\varphi}{^{i}_{jk}}c^{k}\right)\frac{\lambda^{2}}{2}+\left(F^{l}_{i}\tensor{\varphi}{^{i}_{jk}}c^{j}c^{k}b_{l}\right.\\
&\left.+2\tensor{\varphi}{^{i}_{jk}}F^{k}_{m,i}\partial\gamma^{m}c^{j}+2\partial(F^{j}_{i})\tensor{\varphi}{^{i}_{jk}}c^{k}\right)\lambda+ \text{terms without $\lambda$},
\end{align*}
\begin{align*}
&[{F^{l}_{m}\partial\gamma^{m}b_{l}}_{\lambda}D(\varphi_{ijk} g^{il}g^{jm}c^{k} b_{l}b_{m})]\\
=&\left(2F^{k}_{j}\tensor{\varphi}{^{ij}_{k}}b_{i}\right)\frac{\lambda^{2}}{2}+\left(2F^{l}_{j}\tensor{\varphi}{^{ij}_{k}}c^{k}b_{i}b_{l}+(-2)\tensor{\varphi}{^{ij}_{k}}F^{k}_{m,j}\partial\gamma^{m}b_{i}\right.\\
&\left.+2\partial(F^{k}_{j})\tensor{\varphi}{^{ij}_{k}}b_{i}\right)\lambda+ \text{terms without $\lambda$},
\end{align*}
\begin{align*}
[{F^{l}_{m}\partial\gamma^{m}b_{l}}_{\lambda}D(\varphi_{ijk} g^{il}g^{jm}g^{kn}b_{l}b_{m}b_{n})]&=\left(3F^{l}_{k}\varphi^{ijk}b_{i}b_{j}b_{l}\right)\lambda+ \text{terms without $\lambda$},
\end{align*}
\begin{align*}
[{F^{l}_{m}\partial\gamma^{m}b_{l}}_{\lambda}D(\varphi_{ijk}\Gamma^{i}_{mn}g^{jm}\partial \gamma^{n}g^{kl}b_{l})]&=\left(2F^{j}_{m}F^{i}_{j}\partial\gamma^{m}b_{i}\right)\lambda+ \text{terms without $\lambda$},
\end{align*}
\begin{align*}
[{F^{l}_{m}\partial\gamma^{m}b_{l}}_{\lambda}D(\varphi_{ijk}\Gamma^{i}_{mn}g^{jm}\partial \gamma^{n}c^{k})]&=\left(F^{i}_{m}F_{ji}\partial\gamma^{m}c^{j}\right)\lambda+ \text{terms without $\lambda$},
\end{align*}
\begin{align*}
[{F_{ji}\partial\gamma^{i}c^{j}}_{\lambda}D(\varphi_{ijk}g^{il}c^{j}c^{k}b_{l})]&=\left(F_{ji}\tensor{\varphi}{^{i}_{mn}}c^{j}c^{m}c^{n}\right)\lambda+ \text{terms without $\lambda$},
\end{align*}
\begin{align*}
&[{F_{ji}\partial\gamma^{i}c^{j}}_{\lambda}D(\varphi_{ijk} g^{il}g^{jm}c^{k} b_{l}b_{m})]\\
=&\left(2F_{ji}\tensor{\varphi}{^{ij}_{n}}c^{n}\right)\frac{\lambda^{2}}{2} + \left(2F_{ji}\tensor{\varphi}{^{li}_{n}}c^{j}c^{n}b_{l}\right.\\
&\left.+2\tensor{\varphi}{^{jm}_{n}}F_{ji,m}\partial\gamma^{i}c^{n}+2\partial(F_{ji})\tensor{\varphi}{^{ij}_{n}}c^{n}\right)\lambda+ \text{terms without $\lambda$},
\end{align*}
\begin{align*}
&[{F_{ji}\partial\gamma^{i}c^{j}}_{\lambda}D(\varphi_{ijk} g^{il}g^{jm}g^{kn}b_{l}b_{m}b_{n})]\\
=&\left(6F_{ji}\varphi^{ijm}b_{m}\right)\frac{\lambda^{2}}{2}+\left(3F_{ji}\varphi^{lmi}c^{j}b_{l}b_{m}\right.\\
&\left.+6\varphi^{ljn}F_{ji,n}\partial\gamma^{i}b_{l}+6\partial(F_{ji})\varphi^{ijl}b_{l}\right)\lambda+ \text{terms without $\lambda$},
\end{align*}
\begin{align*}
[{F_{ji}\partial\gamma^{i}c^{j}}_{\lambda}D(\varphi_{ijk}\Gamma^{i}_{mn}g^{jm}\partial \gamma^{n}g^{kl}b_{l})]&=\left(F_{ji}F^{i}_{m}\partial\gamma^{m}c^{j}\right)\lambda+ \text{terms without $\lambda$}.
\end{align*}

\subsubsection{\texorpdfstring{$({\Phi_{+}}_{(1)}K_{+})+({\Phi_{-}}_{(1)}K_{-})$}{}}\label{PhioneK(1)}
Now we want to compute $({\Phi_{+}}_{(1)}K_{+})+({\Phi_{-}}_{(1)}K_{-})$, it should be notice that the $\lambda$-brackets used to compute ${\Phi_{+}}_{(1)}K_{+}$ and ${\Phi_{-}}_{(1)}K_{-}$ are the same modulo a sign. Then to compute $({\Phi_{+}}_{(1)}K_{+})+({\Phi_{-}}_{(1)}K_{-})$ we only need to take into account the $\lambda$-brackets that have the same sign an consider each one twice.\\
We compute  $({\Phi_{+}}_{(1)}K_{+})+({\Phi_{-}}_{(1)}K_{-})$ analyzing the coefficient of each type of term that appears. All the coefficients were obtained after a long but straightforward computation using identities (\ref{dosPhidosg}) and (\ref{dosPhiunag}), and the fact that $d\varphi=0$ and $\nabla g=0$, except the coefficient of $\partial\gamma^{i}c^{j}$ that is more involved and is detailed below.\\

\fbox{coefficient of $b_{i}b_{j}b_{k}$:} \hspace{0.1in} $0$,

\fbox{coefficient of $c_{i}c_{j}b_{k} $:} \hspace{0.1in} $0$,

\fbox{coefficient of $\partial c^{i} $:} \hspace{0.1in} $0$,

\fbox{coefficient of $\partial\gamma^{i}b_{i} $:} \hspace{0.1in}  $(-3)\partial\gamma^{i}b_{i}$,

\fbox{coefficient of $c^{i}\beta_{i}$:} \hspace{0.1in}  $(-3)c^{i}\beta_{i}$,

\fbox{coefficient of $\partial\gamma^{i}c^{j}$:} \hspace{0.1in} $0$.

\textbf{Computations to obtain the coefficient of} $\partial\gamma^{i}c^{j}$:\\ 
Denote by $A_{1}$ the terms of type $\partial\gamma^{i}c^{j}$ that appear in the computations of the $c^{i}\beta_{i}$ coefficient due to quasi-associativity (\ref{quasi-associativity}):
\begin{align*}
A_{1}=&-\dfrac{1}{8}\partial\tensor{\varphi}{_{ijk}}\tensor{\varphi}{^{jkn}_{,n}}c^{i}+(-\dfrac{1}{8})\partial\tensor{\varphi}{^{jkn}}\tensor{\varphi}{_{ijk,n}}c^{i}+\dfrac{1}{4}\partial\tensor{\varphi}{^{ij}_{k}}\tensor{\varphi}{^{k}_{mi,j}}c^{m}\nonumber\\
&+\dfrac{1}{4}\partial\tensor{\varphi}{^{k}_{mi}}\tensor{\varphi}{^{ij}_{k,j}}c^{m}+\dfrac{1}{8}\partial\tensor{\varphi}{^{ij}_{k}}\tensor{\varphi}{^{l}_{ji,l}}c^{k}+\dfrac{1}{8}\partial\tensor{\varphi}{^{l}_{ji}}\tensor{\varphi}{^{ij}_{k,l}}c^{k},\nonumber
\end{align*}
Collecting the other terms that contain $\partial\gamma^{i}c^{j}$: denote by $A_{2}$ the sum of the terms that does not contain derivatives of the Christoffel symbols, denote by $A_{3}$ the sum of the terms containing derivatives of the Christoffel symbols.\\
We have:
\begin{equation}
A_{3}=\dfrac{1}{2}\tensor{\varphi}{^m^l_i}\tensor{\varphi}{^n_l_s}(\Gamma^{i}_{mn})_{,r}\partial\gamma^{r}c^{s}+\dfrac{1}{2}\tensor{\varphi}{^s_r_l}\tensor{\varphi}{^m^l_i}(\Gamma^{i}_{mn})_{,s}\partial\gamma^{n}c^{r}.\nonumber
\end{equation}
Using the identity (\ref{dosPhiunag}) we get:
\begin{align*}
A_{3}=&\dfrac{1}{2}g_{is}g^{mn}(\Gamma^{i}_{mn})_{,r}\partial\gamma^{r}c^{s}+(-\dfrac{1}{2})g_{ri}g^{sm}(\Gamma^{i}_{mn})_{,s}\partial\gamma^{n}c^{r}+\dfrac{1}{2}(\Gamma^{s}_{mn})_{,s}\partial\gamma^{n}c^{m}\nonumber\\
&+\dfrac{1}{2}\tensor{\psi}{^s_r^m_i}(\Gamma^{i}_{mn})_{,s}\partial\gamma^{n}c^{r}.\nonumber
\end{align*}
Let $R$ denote the Riemann curvature, using the identity (\ref{RiemanncurvatureChristoffel}) we can work the first two summands of $A_{3}$:
\begin{align*}
&\dfrac{1}{2}g_{is}g^{mn}(\Gamma^{i}_{mn})_{,r}\partial\gamma^{r}c^{s}\\
=&\dfrac{1}{2}g_{is}g^{mn}\tensor{R}{^i_n_r_m}\partial\gamma^{r}c^{s}+\dfrac{1}{2}g_{is}g^{mn}(\tensor{\Gamma}{_r^i_n})_{,m}\partial\gamma^{r}c^{s}\nonumber\\
&+(-\dfrac{1}{2})g_{is}g^{mn}\tensor{\Gamma}{_r^i_a}\tensor{\Gamma}{_m^a_n}\partial\gamma^{r}c^{s}+\dfrac{1}{2}g_{is}g^{mn}\tensor{\Gamma}{_m^i_a}\tensor{\Gamma}{_r^a_n}\partial\gamma^{r}c^{s}\nonumber\\
=&\dfrac{1}{2}R_{sr}\partial\gamma^{r}c^{s}+\dfrac{1}{2}g_{is}g^{mn}(\tensor{\Gamma}{_r^i_n})_{,m}\partial\gamma^{r}c^{s}+(-\dfrac{1}{2})\tensor{\Gamma}{_i_a_s}\tensor{\Gamma}{_j^i^j}\partial\gamma^{s}c^{a}\nonumber\\
&+\dfrac{1}{2}\tensor{\Gamma}{_i_a_j}\tensor{\Gamma}{^i^j_s}\partial\gamma^{s}c^{a},\nonumber
\end{align*}
\begin{align*}
&\dfrac{1}{2}g_{ri}g^{sm}(\Gamma^{i}_{mn})_{,s}\partial\gamma^{n}c^{r}\\
=&(-\dfrac{1}{2})g_{ri}g^{sm}\tensor{R}{^i_n_s_m}\partial\gamma^{n}c^{r}+(-\dfrac{1}{2})g_{ri}g^{sm}(\tensor{\Gamma}{_s^i_n})_{,m}\partial\gamma^{n}c^{r}\nonumber\\
&+\dfrac{1}{2}g_{ri}g^{sm}\tensor{\Gamma}{_s^i_a}\tensor{\Gamma}{_m^a_n}\partial\gamma^{n}c^{r}+(-\dfrac{1}{2})g_{ri}g^{sm}\tensor{\Gamma}{_m^i_a}\tensor{\Gamma}{_s^a_n}\partial\gamma^{n}c^{r}\nonumber\\
=&(-\dfrac{1}{2})g_{ri}g^{sm}(\tensor{\Gamma}{_s^i_n})_{,m}\partial\gamma^{n}c^{r}\nonumber\\
=&-\dfrac{1}{2}g_{is}g^{mn}(\tensor{\Gamma}{_r^i_n})_{,m}\partial\gamma^{r}c^{s},\nonumber
\end{align*}
then
\begin{align*}
A_{3}=&\dfrac{1}{2}R_{sr}\partial\gamma^{r}c^{s}+(-\dfrac{1}{2})\tensor{\Gamma}{_i_a_s}\tensor{\Gamma}{_j^i^j}\partial\gamma^{s}c^{a}+\dfrac{1}{2}\tensor{\Gamma}{_i_a_j}\tensor{\Gamma}{^i^j_s}\partial\gamma^{s}c^{a}+\dfrac{1}{2}(\Gamma^{s}_{mn})_{,s}\partial\gamma^{n}c^{m}\nonumber\\
&+\dfrac{1}{2}\tensor{\psi}{^s_r^m_i}(\Gamma^{i}_{mn})_{,s}\partial\gamma^{n}c^{r}.\nonumber
\end{align*}
We also have
\begin{align*}
A_{1} + A_{2}=&(-\dfrac{1}{2})\tensor{\Gamma}{_a^i_j}\tensor{\Gamma}{_i^j_s}\partial\gamma^{s}c^{a}+\dfrac{1}{2}\tensor{\Gamma}{_i_a_s}\tensor{\Gamma}{_j^i^j}\partial\gamma^{s}c^{a}+(-\dfrac{1}{2})\tensor{\Gamma}{_i_a_j}\tensor{\Gamma}{^i^j_s}\partial\gamma^{s}c^{a}\nonumber\\
&+\dfrac{1}{2}\tensor{\Gamma}{_i^j_k}\tensor{\Gamma}{_l^i_s}\tensor{\psi}{^k_a^l_j}\partial\gamma^{s}c^{a},\nonumber
\end{align*}
then
\begin{align*}
&A_{1}+A_{2}+A_{3}\\
=&(-\dfrac{1}{2})\tensor{\Gamma}{_a^i_j}\tensor{\Gamma}{_i^j_s}\partial\gamma^{s}c^{a}+\dfrac{1}{2}\tensor{\Gamma}{_i^j_k}\tensor{\Gamma}{_l^i_s}\tensor{\psi}{^k_a^l_j}\partial\gamma^{s}c^{a}+\dfrac{1}{2}R_{sr}\partial\gamma^{r}c^{s}\nonumber\\
&+\dfrac{1}{2}(\Gamma^{s}_{mn})_{,s}\partial\gamma^{n}c^{m}+\dfrac{1}{2}\tensor{\psi}{^s_r^m_i}(\Gamma^{i}_{mn})_{,s}\partial\gamma^{n}c^{r}\nonumber\\
=&(\dfrac{1}{2}(\Gamma^{s}_{mn})_{,s}\partial\gamma^{n}c^{m}+(-\dfrac{1}{2})\tensor{\Gamma}{_a^i_j}\tensor{\Gamma}{_i^j_s}\partial\gamma^{s}c^{a})+(\dfrac{1}{2}\tensor{\psi}{^s_r^m_i}(\Gamma^{i}_{mn})_{,s}\partial\gamma^{n}c^{r}\nonumber\\
&+\dfrac{1}{2}\tensor{\psi}{^s_r^m_i}\tensor{\Gamma}{_a^i_s}\tensor{\Gamma}{_m^a_n}\partial\gamma^{n}c^{r})+\dfrac{1}{2}R_{sr}\partial\gamma^{r}c^{s}\nonumber\\
=&\dfrac{1}{2}R_{mn}\partial\gamma^{n}c^{m}+\dfrac{1}{2}\tensor{\psi}{^s_r^m_i}\tensor{R}{^i_m_s_n}\partial\gamma^{n}c^{r}+\dfrac{1}{2}R_{sr}\partial\gamma^{r}c^{s}\nonumber\\
=&R_{mn}\partial\gamma^{n}c^{m}+ \dfrac{1}{2}\tensor{R}{_a_b_c_d}\tensor{\psi}{^c^d_l_m}g^{am}\partial\gamma^{b}c^{l},\nonumber\\
=&0.\nonumber
\end{align*}
\noindent
To conclude the last equality is zero we use the Lemma \ref{lemmaG2Ricciflat} and Lemma \ref{lemmaG2Curvatureidentity}.\\
Finally we have proved that $$({\Phi_{+}}_{(1)}K_{+})+({\Phi_{-}}_{(1)}K_{-})=(-3)\partial\gamma^{i}b_{i}+(-3)c^{i}\beta_{i}$$.
\subsubsection{\texorpdfstring{${\Phi_{\pm}}_{(1)}K_{\pm}$}{}}\label{PhioneK(2)}
Now we want to compute ${\Phi_{+}}_{(1)}K_{+}\;\left({\Phi_{-}}_{(1)}K_{-}\right)$, as was noted in \ref{PhioneK(1)} the $\lambda$-brackets used to compute ${\Phi_{+}}_{(1)}K_{+}$ and ${\Phi_{-}}_{(1)}K_{-}$ are exactly the same modulo a sign. Then to compute $({\Phi_{+}}_{(1)}K_{+})\;\left({\Phi_{-}}_{(1)}K_{-}\right)$ we only need to consider the $\lambda$-brackets that change sign and remember to add one half of the sum $({\Phi_{+}}_{(1)}K_{+})+({\Phi_{-}}_{(1)}K_{-})$. We compute  ${\Phi_{+}}_{(1)}K_{+}\;\left({\Phi_{-}}_{(1)}K_{-}\right)$ analyzing the coefficient of each type of term that appears. All the coefficients were obtained after a long but straightforward computation using identities (\ref{dosPhidosg}) and (\ref{dosPhiunag}), and the fact that $d\varphi=0$ and $\nabla g=0$, except the coefficient of $\partial \gamma^{i}b_{j}$ that is more involved and is detailed below.

\fbox{coefficient of $b_{i}\beta_{j}$:}   \hspace{0.1in} $(-\tfrac{3}{2})g^{ij}b_{i}\beta_{j},$

\fbox{coefficient of $\partial b_{i}$:}   \hspace{0.1in} $(-\tfrac{3}{2})g^{ij}\Gamma^{k}_{ij}\partial b_{k}$, 

\fbox{coefficient of $c^{i}c^{j}c^{k}$:}  \hspace{0.1in} $0$,

\fbox{coefficient of $c^{i}b_{j}b_{k}$:}  \hspace{0.1in} $(-\tfrac{3}{2})g^{ij}\Gamma^{l}_{ik}c^{k}b_{j}b_{l},$

\fbox{coefficient of $\partial \gamma^{i}c^{j}$:}  \hspace{0.1in} $(-\tfrac{3}{2})g_{ij}\partial\gamma^{i}c^{j},$

\fbox{coefficient of $\partial \gamma^{i}b_{j}$:} \hspace{0.1in} $(-3)g^{ij}\Gamma^{k}_{il}\Gamma^{l}_{jm}\partial\gamma^{m}b_{k}$. \hspace{0.1in} 

\textbf{Computations to obtain the coefficient of} $\partial \gamma^{i}b_{j}$.\\

Denote by $A_{1}$ the terms of type $\partial \gamma^{i}b_{j}$ that appear in the computations of the $b_{i}\beta_{j}$ coefficient due to quasi-associativity (\ref{quasi-associativity}):
\begin{align*}
A_{1}=&(-\dfrac{1}{16})\partial(\varphi^{ijk})\tensor{\varphi}{^{l}_{ij,k}}b_{l}+ (-\dfrac{1}{16})\partial(\tensor{\varphi}{^{l}_{ij}})\tensor{\varphi}{^{ijk}_{,k}}b_{l}+ \dfrac{1}{16}\partial(\varphi^{ijk})\tensor{\varphi}{^{l}_{kj,l}}b_{i}\nonumber\\
&+\dfrac{1}{16}\partial(\tensor{\varphi}{^{l}_{kj}})\tensor{\varphi}{^{ijk}_{,l}}b_{i}+\dfrac{1}{8}\partial(\tensor{\varphi}{^{ij}_{k}})\tensor{\varphi}{^{km}_{i,m}}b_{j}+\dfrac{1}{8}\partial(\tensor{\varphi}{^{km}_{i}})\tensor{\varphi}{^{ij}_{k,m}}b_{j}.\nonumber
\end{align*}
Collecting the other terms that contain $\partial \gamma^{i}b_{j}$: denote by $A_{2}$ the sum of the terms that does not contain derivatives of the Christoffel symbols, denote by $A_{3}$ the sum of the terms containing derivatives of the Christoffel symbols.\\
We have:
\begin{align*}
A_{1}+A_{2}=&(-\dfrac{1}{2})\tensor{\Gamma}{_i^j_s}\tensor{\Gamma}{^a^i_j}\partial\gamma^{s}b_{a}+\dfrac{1}{2}\tensor{\Gamma}{_i^a_s}\tensor{\Gamma}{_j^i^j}\partial\gamma^{s}b_{a}+(-\dfrac{7}{2})\tensor{\Gamma}{_i^a_j}\tensor{\Gamma}{^i^j_s}\partial\gamma^{s}b_{a}\nonumber\\
&+(-\dfrac{1}{2})\tensor{\psi}{^a_j^k^l}\tensor{\Gamma}{_i^j_k}\tensor{\Gamma}{_l^i_s}\partial\gamma^{s}b_{a},\nonumber
\end{align*}
\begin{align*}
A_{3}=&\dfrac{1}{2}g^{lk}(\tensor{\Gamma}{_k^i_l})_{,s}\partial\gamma^{s}b_{i}+(-\dfrac{1}{2})g^{rm}(\tensor{\Gamma}{_m^i_s})_{,r}\partial\gamma^{s}b_{i}+\dfrac{1}{2}g^{am}(\tensor{\Gamma}{_m^i_s})_{,i}\partial\gamma^{s}b_{a}\nonumber\\
&+\dfrac{1}{2}\tensor{\psi}{^r^a^m^n}g_{in}(\tensor{\Gamma}{_m^i_s})_{,r}\partial\gamma^{s}b_{a}.\nonumber
\end{align*}
Let $R$ denote the Riemann curvature, using the identity (\ref{RiemanncurvatureChristoffel}) we can write:
\begin{align*}
&\dfrac{1}{2}g^{lk}(\tensor{\Gamma}{_k^i_l})_{,s}\partial\gamma^{s}b_{i}\\
=&\dfrac{1}{2}g^{kl}\tensor{R}{^i_l_s_k}\partial\gamma^{s}b_{i}+\dfrac{1}{2}g^{kl}(\tensor{\Gamma}{_s^i_l})_{,k}\partial\gamma^{s}b_{i}+(-\dfrac{1}{2})g^{kl}\tensor{\Gamma}{_s^i_a}\tensor{\Gamma}{_k^a_l}\partial\gamma^{s}b_{i}\nonumber\\
&+\dfrac{1}{2}g^{kl}\tensor{\Gamma}{_k^i_a}\tensor{\Gamma}{_s^a_l}\partial\gamma^{s}b_{i},\nonumber
\end{align*}
and using identity (\ref{RiccitensorChristoffel}) we have
\begin{align*}
\dfrac{1}{2}g^{am}(\tensor{\Gamma}{_m^i_s})_{,i}\partial\gamma^{s}b_{a}=&\dfrac{1}{2}g^{am}\tensor{R}{_m_s}\partial\gamma^{s}b_{a}+\dfrac{1}{2}g^{am}\tensor{\Gamma}{_m^j_l}\tensor{\Gamma}{_s^l_j}\partial\gamma^{s}b_{a}.\nonumber
\end{align*}
Then
\begin{align*}
A_{3}=&\dfrac{1}{2}\tensor{R}{^i_s}\partial\gamma^{s}b_{i}+(-\dfrac{1}{2})g^{kl}\tensor{\Gamma}{_s^i_a}\tensor{\Gamma}{_k^a_l}\partial\gamma^{s}b_{i}++\dfrac{1}{2}g^{kl}\tensor{\Gamma}{_k^i_a}\tensor{\Gamma}{_s^a_l}\partial\gamma^{s}b_{i}\nonumber\\
&+\dfrac{1}{2}\tensor{R}{^a_s}\partial\gamma^{s}b_{a}+\dfrac{1}{2}g^{am}\tensor{\Gamma}{_m^j_l}\tensor{\Gamma}{_s^l_j}\partial\gamma^{s}b_{a}+\dfrac{1}{2}\tensor{\psi}{^r^a^m^n}g_{in}(\tensor{\Gamma}{_m^i_s})_{,r}\partial\gamma^{s}b_{a}.\nonumber
\end{align*}
Using again the identity (\ref{RiemanncurvatureChristoffel}) we can write:
\begin{align*}
&(-\dfrac{1}{2})\tensor{\psi}{^a_j^k^l}\tensor{\Gamma}{_i^j_k}\tensor{\Gamma}{_l^i_s}\partial\gamma^{s}b_{a}+\dfrac{1}{2}\tensor{\psi}{^r^a^m^n}g_{in}(\tensor{\Gamma}{_m^i_s})_{,r}\partial\gamma^{s}b_{a}\nonumber\\
=&(-\dfrac{1}{2})\tensor{\psi}{^a_j^k^l}\tensor{\Gamma}{_k^j_i}\tensor{\Gamma}{_l^i_s}\partial\gamma^{s}b_{a}+(-\dfrac{1}{2})\tensor{\psi}{^a_j^k^l}(\tensor{\Gamma}{_l^j_s})_{,k}\partial\gamma^{s}b_{a}\nonumber\\
=&(-\dfrac{1}{2})\tensor{\psi}{^a_j^k^l}\tensor{R}{^j_l_k_s}\partial\gamma^{s}b_{a}\nonumber\\
=&\dfrac{1}{2}\tensor{R}{_k_s_j_l}\tensor{\psi}{^j^l_i^k}g^{ai}\partial\gamma^{s}b_{a},\nonumber
\end{align*}
then
\begin{align*}
A_{1}+A_{2}+A_{3}=&(-3)g^{kl}\tensor{\Gamma}{_k^i_a}\tensor{\Gamma}{_s^a_l}\partial\gamma^{s}b_{i}+\tensor{R}{^i_j}\partial\gamma^{j}b_{i}+\dfrac{1}{2}\tensor{R}{_k_s_j_l}\tensor{\psi}{^j^l_i^k}g^{ai}\partial\gamma^{s}b_{a}\nonumber\\
=&(-3)g^{ij}\Gamma^{k}_{il}\Gamma^{l}_{jm}\partial\gamma^{m}b_{k}.\nonumber
\end{align*}
\noindent
To establish the last equality we use the Lemma \ref{lemmaG2Ricciflat} and Lemma \ref{lemmaG2Curvatureidentity}.\\
Finally we have proved that:
\begin{align*}
{\Phi_{+}}_{(1)}K_{+}=&-\tfrac{3}{2}c^{i}\beta_{i}-\tfrac{3}{2}\partial\gamma^{i}b_{i}-\tfrac{3}{2}g^{ij}b_{i}\beta_{j}-\tfrac{3}{2}g^{ij}\Gamma^{l}_{ik}c^{k}b_{j}b_{l}-\tfrac{3}{2}g^{ij}\Gamma^{k}_{ij}\partial b_{k}\nonumber\\
&-\tfrac{3}{2}g_{ij}\partial\gamma^{i}c^{j}-3g^{ij}\Gamma^{k}_{il}\Gamma^{l}_{jm}\partial\gamma^{m}b_{k},\nonumber
\end{align*}
\begin{align*}
{\Phi_{-}}_{(1)}K_{-}=&-\tfrac{3}{2}c^{i}\beta_{i}-\tfrac{3}{2}\partial\gamma^{i}b_{i}+(\tfrac{3}{2})g^{ij}b_{i}\beta_{j}+(\tfrac{3}{2})g^{ij}\Gamma^{l}_{ik}c^{k}b_{j}b_{l}+(\tfrac{3}{2})g^{ij}\Gamma^{k}_{ij}\partial b_{k}\nonumber\\
&+(\tfrac{3}{2})g_{ij}\partial\gamma^{i}c^{j}+3g^{ij}\Gamma^{k}_{il}\Gamma^{l}_{jm}\partial\gamma^{m}b_{k}.\nonumber
\end{align*}

\subsection{\texorpdfstring{$[{G_{+}}_{\lambda}G_{-}]$}{}}\label{G+withG-}

Now we compute $[{G_{+}}_{\lambda}G_{-}]$,
\begin{eqnarray}
G_{+}&=&\tfrac{1}{2}c^{i}\beta_{i}+\tfrac{1}{2}\partial\gamma^{i}b_{i}+\tfrac{1}{2}g^{ij}b_{i}\beta_{j}+\tfrac{1}{2}g^{ij}\Gamma^{l}_{ik}c^{k}b_{j}b_{l}+\tfrac{1}{2}g^{ij}\Gamma^{k}_{ij}\partial b_{k}\nonumber\\
&&+\tfrac{1}{2}g_{ij}\partial\gamma^{i}c^{j}+g^{ij}\Gamma^{k}_{il}\Gamma^{l}_{jm}\partial\gamma^{m}b_{k},\nonumber
\end{eqnarray}
\begin{eqnarray}
G_{-}&=&\tfrac{1}{2}c^{i}\beta_{i}+\tfrac{1}{2}\partial\gamma^{i}b_{i}+(-\tfrac{1}{2})g^{ij}b_{i}\beta_{j}+(-\tfrac{1}{2})g^{ij}\Gamma^{l}_{ik}c^{k}b_{j}b_{l}+(-\tfrac{1}{2})g^{ij}\Gamma^{k}_{ij}\partial b_{k}\nonumber\\
&&+(-\tfrac{1}{2})g_{ij}\partial\gamma^{i}c^{j}+(-1)g^{ij}\Gamma^{k}_{il}\Gamma^{l}_{jm}\partial\gamma^{m}b_{k}.\nonumber
\end{eqnarray}
We list the  non-zero $\lambda$-brackets between the summands of $G_{+}$ and the summands of $G_{-}$, to compute these we used the $\mathtt{Mathematica}$ package \cite{Thielemans91}.
\begin{align*}
[{c^{i}\beta_{i}}_{\lambda}\partial\gamma^{j}b_{j}]=&\tfrac{7}{2}\lambda^{2}+(c^{i}b_{i})\lambda+ \partial c^{i}b_{i}+\partial\gamma^{j}\beta_{j},\nonumber
\end{align*}
\begin{align*}
[{c^{i}\beta_{i}}_{\lambda}g^{lm}b_{l}\beta_{m}]=&\left((-1)(g^{lm}_{,l})_{,m}\right)\dfrac{\lambda^2}{2}+\left(g^{lm}_{,l}\beta_{m}+(-1)(g^{lm}_{,i})_{,m}c^{i}b_{l}\right)\lambda\nonumber\\
&+g^{lm}\beta_{l}\beta_{m}+g^{lm}_{,i}c^{i}b_{l}\beta_{m}+(-1)(g^{lm}_{,i})_{,m}\partial c^{i}b_{l}+\partial(g^{lm}_{,l})\beta_{m},\nonumber
\end{align*}
\begin{align*}
&[{c^{i}\beta_{i}}_{\lambda}g^{lm}\Gamma^{n}_{la}c^{a}b_{m}b_{n}]\\
=&\left((g^{lm}\Gamma^{n}_{la})_{,n}c^{a}b_{m}+(-1)(g^{lm}\Gamma^{n}_{la})_{,m}c^{a}b_{n}\right)\lambda\nonumber\\
&+(g^{lm}\Gamma^{n}_{la})c^{a}b_{m}\beta_{n}+(-1)(g^{lm}\Gamma^{n}_{la})c^{a}b_{n}\beta_{m}\nonumber\\
&+(-1)(g^{lm}\Gamma^{n}_{la})_{,i}c^{a}c^{i}b_{m}b_{n}+\partial[(g^{lm}\Gamma^{n}_{la})_{,n}]c^{a}b_{m}\nonumber\\
&+(-1)\partial[(g^{lm}\Gamma^{n}_{la})_{,m}]c^{a}b_{n},\nonumber
\end{align*}
\begin{align*}
[{c^{i}\beta_{i}}_{\lambda}g^{lm}\Gamma^{n}_{lm}\partial b_{n}]=&\left( (g^{lm}\Gamma^{n}_{lm})_{,n} \right)\dfrac{\lambda^2}{2}+\left(  (g^{lm}\Gamma^{n}_{lm})\beta_{n}+ \partial[(g^{lm}\Gamma^{n}_{lm})_{,n}]    \right)\lambda\nonumber\\
&(g^{lm}\Gamma^{n}_{lm})\partial\beta_{n}+(g^{lm}\Gamma^{n}_{lm})_{,i}c^{i}\partial b_{n}+\dfrac{1}{2}\partial^{2}[(g^{lm}\Gamma^{n}_{lm})_{,n}],\nonumber
\end{align*}
\begin{align*}
[{c^{i}\beta_{i}}_{\lambda}g_{lm}\partial\gamma^{l}c^{m}]=&g_{lm}\partial c^{l}c^{m}+(g_{lm})_{,i}\partial\gamma^{l}c^{i}c^{m},\nonumber
\end{align*}
\begin{align*}
&[{c^{i}\beta_{i}}_{\lambda}g^{lm}\Gamma^{n}_{la}\Gamma^{a}_{mr}\partial\gamma^{r}b_{n}]\\
=&\left(g^{lm}\Gamma^{n}_{la}\Gamma^{a}_{mn}\right)\dfrac{\lambda^2}{2}+\left( g^{lm}\Gamma^{n}_{la}\Gamma^{a}_{mi}c^{i}b_{n}+(g^{lm}\Gamma^{n}_{la}\Gamma^{a}_{mr})_{,n}\partial\gamma^{r}\right)\lambda\nonumber\\
&+g^{lm}\Gamma^{n}_{la}\Gamma^{a}_{mi}\partial c^{i}b_{n}+(g^{lm}\Gamma^{n}_{la}\Gamma^{a}_{mr})\partial\gamma^{r}\beta_{n}\nonumber\\
&+(g^{lm}\Gamma^{n}_{la}\Gamma^{a}_{mr})_{,i}\partial\gamma^{r}c^{i}b_{n}+\partial[(g^{lm}\Gamma^{n}_{la}\Gamma^{a}_{mr})_{,n}]\partial\gamma^{r},\nonumber
\end{align*}
\begin{align*}
[{\partial\gamma^{i}b_{i}}_{\lambda} c^{j}\beta_{j}]=&\tfrac{7}{2}\lambda^2+\left((-1)c^{i}b_{i}\right)\lambda + (-1)c^{i}\partial b_{i}+ \partial\gamma^{i}\beta_{i},\nonumber
\end{align*}
\begin{align*}
[{\partial\gamma^{i}b_{i}}_{\lambda} g^{lm}b_{l}\beta_{m}]=&g^{lm}\partial b_{m}b_{l},\nonumber
\end{align*}
\begin{align*}
[{\partial\gamma^{i}b_{i}}_{\lambda}g^{lm}\Gamma^{n}_{la}c^{a}b_{m}b_{n}]=&g^{lm}\Gamma^{n}_{li}\partial\gamma^{i}b_{m}b_{n},\nonumber
\end{align*}
\begin{align*}
[{\partial\gamma^{i}b_{i}}_{\lambda} g_{lm}\partial\gamma^{l}c^{m}]=&g_{lm}\partial\gamma^{l}\partial\gamma^{m},\nonumber
\end{align*}
\begin{align*}
[{g^{lm}b_{l}\beta_{m}}_{\lambda}c^{i}\beta_{i}]=&\left((-1)(g^{lm}_{,l})_{,m}\right)\dfrac{\lambda^2}{2}+\left((-1)(g^{lm}_{,l})\beta_{m}+(g^{lm}_{,i})_{,m}c^{i}b_{l}\right.\nonumber\\
&\left.+(-1)\partial[(g^{lm}_{,l})_{,m}]\right)\lambda+g^{lm}\beta_{l}\beta_{m}+g^{lm}_{,i}c^{i}b_{l}\beta_{m}\nonumber\\
&+(-1)g^{lm}_{,l}\partial\beta_{m}+(g^{lm}_{,i})_{,m}c^{i}\partial b_{l}+\partial[(g^{lm}_{,i})_{,m}]c^{i}b_{l}\nonumber\\
&+(-\dfrac{1}{2})\partial^2[(g^{lm}_{,l})_{,m}],\nonumber
\end{align*}
\begin{align*}
[{g^{lm}b_{l}\beta_{m}}_{\lambda}\partial\gamma^{i}b_{i}]=&(-1)g^{lm}b_{m}\partial b_{l},\nonumber
\end{align*}
\begin{align*}
&[{g^{lm}b_{l}\beta_{m}}_{\lambda}g^{ij}b_{i}\beta_{j}]\\
=&\left((g^{lm}_{,j})(g^{ij}_{,m})b_{i}b_{l}\right)\lambda+(-1)g^{lm}g^{ij}_{,m}b_{i}b_{l}\beta_{j}+g^{lm}(g^{ij}_{,m})_{,j}b_{i}\partial b_{l}\nonumber\\
&+g^{ij}(g^{lm}_{,j})b_{i}b_{l}\beta_{m}+g^{ij}(g^{lm}_{,j})_{,m}b_{i}\partial b_{l}+g^{ij}\partial[(g^{lm}_{,j})_{,m}]b_{i}b_{l}\nonumber\\
&+g^{lm}_{,j}g^{ij}_{,m}b_{i}\partial b_{l}+g^{lm}_{,j}\partial[(g^{ij}_{,m})]b_{i}b_{l}+\partial(g^{lm})(g^{ij}_{,m})_{,j}b_{i}b_{l},\nonumber
\end{align*}
\begin{align*}
&[{g^{lm}b_{l}\beta_{m}}_{\lambda}g^{ij}\Gamma^{a}_{ik}c^{k}b_{j}b_{a}]\\
=&\left((-1)g^{lm}(g^{ij}\Gamma^{a}_{il})_{,m}b_{a}b_{j}\right)\lambda+ (-1)g^{lm}(g^{ij}\Gamma^{a}_{il})b_{a}b_{j}\beta_{m}\nonumber\\
&+g^{lm}(g^{ij}\Gamma^{a}_{ik})_{,m}c^{k}b_{a}b_{j}b_{l}+(-1)g^{lm}\partial[(g^{ij}\Gamma^{a}_{il})_{,m}]b_{a}b_{j}\nonumber\\
&+(-1)\partial(g^{lm})(g^{ij}\Gamma^{a}_{il})_{,m}b_{a}b_{j},\nonumber
\end{align*}
\begin{align*}
[{g^{lm}b_{l}\beta_{m}}_{\lambda}g^{ij}\Gamma^{k}_{ij}\partial b_{k}]=&(-1)g^{lm}(g^{ij}\Gamma^{k}_{ij})_{,m}\partial b_{k}b_{l},\nonumber
\end{align*}
\begin{align*}
&[{g^{lm}b_{l}\beta_{m}}_{\lambda}g_{ij}\partial\gamma^{i}c^{j}]\\
=&\dfrac{1}{2}\lambda^{2}+\left((-1)c^{j}b_{j}+g^{lm}(g_{il})_{,m}\partial\gamma^{i}+g_{lm}\partial(g^{lm})\right)\lambda\nonumber\\
&+(-1)c^{j}\partial b_{j}+\partial\gamma^{i}\beta_{i}+(-1)g^{lm}(g_{ij})_{,m}\partial\gamma^{i}c^{j}b_{l}\nonumber\\
&+g^{lm}\partial[(g_{il})_{,m}]\partial\gamma^{i}+(-1)\partial(g^{lm})g_{mj}c^{j}b_{l}+2\partial(g^{lm})(g_{il})_{,m}\partial\gamma^{i}\nonumber\\
&+\dfrac{1}{2}g_{lm}\partial^{2}(g^{lm})+\partial(g_{il})(g^{lm})_{,m}\partial\gamma^{i},\nonumber
\end{align*}
\begin{align*}
&[{g^{lm}b_{l}\beta_{m}}_{\lambda}g^{ij}\Gamma^{k}_{in}\Gamma^{n}_{ja}\partial\gamma^{a}b_{k}]\\
=&\left((-1)g^{lm}g^{ij}\Gamma^{k}_{in}\Gamma^{n}_{jm}b_{k}b_{l}\right)\lambda+(-1)g^{lm}g^{ij}\Gamma^{k}_{in}\Gamma^{n}_{jm}b_{k}\partial b_{l}\nonumber\\
&+(-1)g^{lm}(g^{ij}\Gamma^{k}_{in}\Gamma^{n}_{ja})_{,m}\partial\gamma^{a}b_{k}b_{l}+(-1)\partial(g^{lm})g^{ij}\Gamma^{k}_{in}\Gamma^{n}_{jm}b_{k}b_{l},\nonumber
\end{align*}
\begin{align*}
&[{g^{ij}\Gamma^{l}_{ik}c^{k}b_{j}b_{l}}_{\lambda}c^{m}\beta_{m}]\\
=&\left((-1)(g^{ij}\Gamma^{m}_{ik})_{,m}c^{k}b_{j}+(g^{im}\Gamma^{l}_{ik})_{,m}c^{k}b_{l}\right)\lambda\nonumber\\
&+(g^{ij}\Gamma^{m}_{ik})c^{k}b_{j}\beta_{m}+(-1)(g^{im}\Gamma^{l}_{ik})c^{k}b_{l}\beta_{m}+(-1)(g^{ij}\Gamma^{l}_{ik})_{,m}c^{k}c^{m}b_{j}b_{l}\nonumber\\
&(-1)(g^{ij}\Gamma^{m}_{ik})_{,m}c^{k}\partial b_{j}+(g^{im}\Gamma^{l}_{ik})_{,m}c^{k}\partial b_{l}+(-1)(g^{ij}\Gamma^{m}_{ik})_{,m}\partial c^{k}b_{j}\nonumber\\
&+(g^{im}\Gamma^{l}_{ik})_{,m}\partial c^{k}b_{l},\nonumber
\end{align*}
\begin{align*}
[{g^{ij}\Gamma^{l}_{ik}c^{k}b_{j}b_{l}}_{\lambda}\partial\gamma^{m}b_{m}]=&g^{ij}\Gamma^{l}_{im}\partial\gamma^{m}b_{j}b_{l},\nonumber
\end{align*}
\begin{align*}
&[{g^{ij}\Gamma^{l}_{ik}c^{k}b_{j}b_{l}}_{\lambda}g^{mn}b_{m}\beta_{n}]\\
=&\left((-1)g^{mn}(g^{ij}\Gamma^{l}_{im})_{,n}b_{j}b_{l}\right)\lambda+(g^{ij}\Gamma^{l}_{im})g^{mn}b_{j}b_{l}\beta_{n}\nonumber\\
&+(-1)g^{mn}(g^{ij}\Gamma^{l}_{im})_{,n}b_{j}\partial b_{l}+(-1)g^{mn}(g^{ij}\Gamma^{l}_{ik})_{,n}c^{k}b_{j}b_{l}b_{m}\nonumber\\
&+(-1)g^{mn}(g^{ij}\Gamma^{l}_{im})_{,n}\partial b_{j}b_{l},\nonumber
\end{align*}
\begin{align*}
&[{g^{ij}\Gamma^{l}_{ik}c^{k}b_{j}b_{l}}_{\lambda}g^{mn}\Gamma^{r}_{ma}c^{a}b_{n}b_{r}]\\
=&\left(g^{ij}\Gamma^{l}_{ir}g^{mn}\Gamma^{r}_{ml}b_{j}b_{n}+(-1)g^{ij}\Gamma^{l}_{in}g^{mn}\Gamma^{r}_{ml}b_{j}b_{r}\right.\nonumber\\
&\left.+(-1)g^{ij}\Gamma^{l}_{ir}g^{mn}\Gamma^{r}_{mj}b_{l}b_{n}+g^{ij}\Gamma^{l}_{in}g^{mn}\Gamma^{r}_{mj}b_{l}b_{r}\right)\lambda\nonumber\\
&+g^{ij}\Gamma^{l}_{ik}g^{mn}\Gamma^{k}_{ma}c^{a}b_{j}b_{l}b_{n}+(-1)g^{ij}\Gamma^{l}_{ik}g^{mk}\Gamma^{r}_{ma}c^{a}b_{j}b_{l}b_{r}\nonumber\\
&+g^{ij}\Gamma^{l}_{ik}g^{mn}\Gamma^{r}_{ml}c^{k}b_{j}b_{n}b_{r}+(-1)g^{ij}\Gamma^{l}_{ik}g^{mn}\Gamma^{r}_{mj}c^{k}b_{l}b_{n}b_{r}\nonumber\\
&+g^{ij}\Gamma^{l}_{ik}g^{mn}\Gamma^{k}_{ml}\partial b_{j}b_{n}+(-1)g^{ij}\Gamma^{l}_{ik}g^{mk}\Gamma^{r}_{ml}\partial b_{j}b_{r}\nonumber\\
&+(-1)g^{ij}\Gamma^{l}_{ik}g^{mn}\Gamma^{k}_{mj}\partial b_{l}b_{n}+g^{ij}\Gamma^{l}_{ik}g^{mk}\Gamma^{r}_{mj}\partial b_{l}b_{r}\nonumber\\
&+\partial(g^{ij}\Gamma^{l}_{ik})g^{mn}\Gamma^{k}_{ml}b_{j}b_{n}+(-1)\partial(g^{ij}\Gamma^{l}_{ik})g^{mk}\Gamma^{r}_{ml}b_{j}b_{r}\nonumber\\
&+(-1)\partial(g^{ij}\Gamma^{l}_{ik})g^{mn}\Gamma^{k}_{mj}b_{l}b_{n}+\partial(g^{ij}\Gamma^{l}_{ik})g^{mk}\Gamma^{r}_{mj}b_{l}b_{r},\nonumber
\end{align*}
\begin{align*}
[{g^{ij}\Gamma^{l}_{ik}c^{k}b_{j}b_{l}}_{\lambda}g^{mn}\Gamma^{a}_{mn}\partial b_{a}]=&\left(g^{ij}\Gamma^{l}_{ik}g^{mn}\Gamma^{k}_{mn}b_{j}b_{l}\right)\lambda+g^{ij}\Gamma^{l}_{ik}g^{mn}\Gamma^{k}_{mn}b_{j}\partial b_{l}\nonumber\\
&+g^{ij}\Gamma^{l}_{ik}g^{mn}\Gamma^{k}_{mn}\partial b_{j}b_{l}+ \partial(g^{ij}\Gamma^{l}_{ik})g^{mn}\Gamma^{k}_{mn}b_{j}b_{l},\nonumber
\end{align*}
\begin{align*}
[{g^{ij}\Gamma^{l}_{ik}c^{k}b_{j}b_{l}}_{\lambda}g_{mn}\partial\gamma^{m}c^{n}]=&g^{ij}\Gamma^{l}_{ik}g_{lm}\partial\gamma^{m}c^{k}b_{j}+(-1)\Gamma^{l}_{ik}\partial\gamma^{i}c^{k}b_{l},\nonumber
\end{align*}
\begin{align*}
[{g^{ij}\Gamma^{l}_{ik}c^{k}b_{j}b_{l}}_{\lambda}g^{mn}\Gamma^{a}_{ms}\Gamma^{s}_{nr}\partial\gamma^{r}b_{a}]=&g^{ij}\Gamma^{l}_{ik}g^{mn}\Gamma^{k}_{ms}\Gamma^{s}_{nr}\partial\gamma^{r}b_{j}b_{l},\nonumber
\end{align*}
\begin{align*}
[{g^{ij}\Gamma^{k}_{ij}\partial b_{k}}_{\lambda}c^{m}\beta_{m}]=&\left((g^{ij}\Gamma^{k}_{ij})_{,k}\right)\dfrac{\lambda^2}{2}+\left((-1)(g^{ij}\Gamma^{k}_{ij})\beta_{k}\right)\lambda\nonumber\\
&+(g^{ij}\Gamma^{k}_{ij})_{,m}c^{m}\partial b_{k}+(-1)\partial(g^{ij}\Gamma^{k}_{ij})\beta_{k},\nonumber
\end{align*}
\begin{align*}
[{g^{ij}\Gamma^{k}_{ij}\partial b_{k}}_{\lambda}g^{lm}b_{l}\beta_{m}]=&(-1)g^{lm}(g^{ij}\Gamma^{k}_{ij})_{,m}\partial b_{k}b_{l},\nonumber
\end{align*}
\begin{align*}
&[{g^{ij}\Gamma^{k}_{ij}\partial b_{k}}_{\lambda}g^{lm}\Gamma^{n}_{la}c^{a}b_{m}b_{n}]\\
=&\left((-1)g^{ij}\Gamma^{k}_{ij}g^{lm}\Gamma^{n}_{lk}b_{m}b_{n}\right)\lambda+(-1)\partial(g^{ij}\Gamma^{k}_{ij})g^{lm}\Gamma^{n}_{lk}b_{m}b_{n},\nonumber
\end{align*}
\begin{align*}
[{g^{ij}\Gamma^{k}_{ij}\partial b_{k}}_{\lambda}g_{lm}\partial\gamma^{l}c^{m}]=&\left((-1)g^{ij}\Gamma^{k}_{ij}g_{lk}\partial\gamma^{l}\right)\lambda+(-1)\partial(g^{ij}\Gamma^{k}_{ij})g_{lk}\partial\gamma^{l},\nonumber
\end{align*}
\begin{align*}
[{g_{mn}\partial\gamma^{m}c^{n}}_{\lambda}c^{i}\beta_{i}]=&(-1)g_{mn}c^{m}\partial c^{n}+(g_{mn})_{,i}\partial\gamma^{m}c^{i}c^{n},\nonumber
\end{align*}
\begin{align*}
[{g_{mn}\partial\gamma^{m}c^{n}}_{\lambda}\partial\gamma^{i}b_{i}]=&g_{mn}\partial\gamma^{m}\partial\gamma^{n},\nonumber
\end{align*}
\begin{align*}
&[{g_{mn}\partial\gamma^{m}c^{n}}_{\lambda}g^{ij}b_{i}\beta_{j}]\\
=&\dfrac{\lambda^2}{2}+\left(c^{i}b_{i}+(-1)g^{ij}(g_{mi})_{,j}\partial\gamma^{m}+g^{ij}\partial(g_{ij})\right)\lambda\nonumber\\
&+\partial c^{i}b_{i}+\partial\gamma^{j}\beta_{j}+(-1)g^{ij}(g_{mn})_{,j}\partial\gamma^{m}c^{n}b_{i}+(-1)g^{ij}(g_{mi})_{,j}\partial^2\gamma^{m}\nonumber\\
&+\partial(g_{mn})g^{im}c^{n}b_{i}+\dfrac{1}{2}\partial^2(g_{ij})g^{ij}+\partial(g^{ij})(g_{mi})_{,j}\partial\gamma^{m}\nonumber\\
&+\partial(g_{mi})(g^{ij})_{,j}\partial\gamma^{m},\nonumber
\end{align*}
\begin{align*}
[{g_{mn}\partial\gamma^{m}c^{n}}_{\lambda}g^{ij}\Gamma^{l}_{ik}c^{k}b_{j}b_{l}]=&g_{lm}g^{ij}\Gamma^{l}_{ik}\partial\gamma^{m}c^{k}b_{j}+(-1)\Gamma^{l}_{ik}\partial\gamma^{i}c^{k}b_{l},\nonumber
\end{align*}
\begin{align*}
&[{g_{mn}\partial\gamma^{m}c^{n}}_{\lambda}g^{ij}\Gamma^{k}_{ij}\partial b_{k}]\\
=&\left(g_{mk}g^{ij}\Gamma^{k}_{ij}\partial\gamma^{m}\right)\lambda+g_{mk}g^{ij}\Gamma^{k}_{ij}\partial^2\gamma^{m}+\partial(g_{mk})g^{ij}\Gamma^{k}_{ij}\partial\gamma^{m},\nonumber
\end{align*}
\begin{align*}
[{g_{mn}\partial\gamma^{m}c^{n}}_{\lambda}g^{ij}\Gamma^{k}_{il}\Gamma^{l}_{ja}\partial\gamma^{a}b_{k}]=&g_{mn}g^{ij}\Gamma^{n}_{il}\Gamma^{l}_{ja}\partial\gamma^{a}\partial\gamma^{m},\nonumber
\end{align*}
\begin{align*}
&[{g^{ij}\Gamma^{k}_{il}\Gamma^{l}_{ja}\partial\gamma^{a}b_{k}}_{\lambda}c^{m}\beta_{m}]\\
=&\left(g^{ij}\Gamma^{k}_{il}\Gamma^{l}_{jk}\right)\dfrac{\lambda^2}{2}+\left((-1)g^{ij}\Gamma^{k}_{il}\Gamma^{l}_{jm}c^{m}b_{k}\right.\nonumber\\
&\left.+(-1)(g^{ij}\Gamma^{k}_{il}\Gamma^{l}_{ja})_{,k}\partial\gamma^{a}+\partial(g^{ij}\Gamma^{k}_{il}\Gamma^{l}_{jk})\right)\lambda+(-1)g^{ij}\Gamma^{k}_{il}\Gamma^{l}_{jm}c^{m}\partial b_{k}\nonumber\\
&+(g^{ij}\Gamma^{k}_{il}\Gamma^{l}_{ja})\partial\gamma^{a}\beta_{k}+(g^{ij}\Gamma^{k}_{il}\Gamma^{l}_{ja})_{,m}\partial\gamma^{a}c^{m}b_{k}+(-1)(g^{ij}\Gamma^{k}_{il}\Gamma^{l}_{ja})_{,k}\partial^2\gamma^{a}\nonumber\\
&+(-1)\partial(g^{ij}\Gamma^{k}_{il}\Gamma^{l}_{jm})c^{m}b_{k}+\tfrac{1}{2}\partial^2(g^{ij}\Gamma^{k}_{il}\Gamma^{l}_{jk}),\nonumber
\end{align*}
\begin{align*}
&[{g^{ij}\Gamma^{k}_{il}\Gamma^{l}_{ja}\partial\gamma^{a}b_{k}}_{\lambda}g^{mn}b_{m}\beta_{n}]\\
=&\left(g^{ij}\Gamma^{k}_{il}\Gamma^{l}_{ja}g^{ma}b_{k}b_{m}\right)\lambda+g^{ij}\Gamma^{k}_{il}\Gamma^{l}_{ja}g^{ma}\partial b_{k}b_{m}\nonumber\\
&(-1)(g^{ij}\Gamma^{k}_{il}\Gamma^{l}_{ja})_{,n}g^{mn}\partial\gamma^{a}b_{k}b_{m}+ \partial(g^{ij}\Gamma^{k}_{il}\Gamma^{l}_{ja})g^{ma}b_{k}b_{m},\nonumber
\end{align*}
\begin{align*}
[{g^{ij}\Gamma^{k}_{il}\Gamma^{l}_{ja}\partial\gamma^{a}b_{k}}_{\lambda}g^{mn}\Gamma^{r}_{ms}c^{s}b_{n}b_{r}]=&g^{ij}\Gamma^{k}_{il}\Gamma^{l}_{ja}g^{mn}\Gamma^{r}_{mk}\partial\gamma^{a}b_{n}b_{r},\nonumber
\end{align*}
\begin{align*}
[{g^{ij}\Gamma^{k}_{il}\Gamma^{l}_{ja}\partial\gamma^{a}b_{k}}_{\lambda}g^{mn}\partial\gamma^{m}c^{n}]=&g^{ij}\Gamma^{k}_{il}\Gamma^{l}_{ja}g_{mk}\partial\gamma^{a}\partial\gamma^{m}.\nonumber
\end{align*}

\subsubsection{Coefficient of \texorpdfstring{$\lambda^2$}{}}

We prove that it is zero after an easy computation using only that the metric is covariantly constant.

\subsubsection{Coefficient of \texorpdfstring{$\lambda$}{}}\label{coefflambdainG+lambdaG-}
We compute the coefficient of each type of term that appears:
\begin{itemize}
\item \fbox{coefficient $c^{i}b_{j}:$} \hspace{0.1in} $0$,
\begin{align*}
&(-1)(g^{lm}\Gamma^{n}_{la})_{,n}c^{a}b_{m}+(-1)g^{lm}\Gamma^{n}_{la}\Gamma^{a}_{mi}c^{i}b_{n}\nonumber\\
=&g^{sm}\Gamma^{l}_{sn}\Gamma^{n}_{la}c^{a}b_{m}+(-1)g^{lm}(\Gamma^{n}_{la})_{,n}c^{a}b_{m}\nonumber\\
=&g^{sm}(\Gamma^{n}_{sa})_{,n}c^{a}b_{m}+(-1)g^{sm}R_{sa}c^{a}b_{m}+(-1)g^{lm}(\Gamma^{n}_{la})_{,n}c^{a}b_{m}\nonumber\\
=&(-1)g^{sm}R_{sa}c^{a}b_{m}\nonumber\\
=&0,\nonumber
\end{align*}
\noindent
by Lemma (\ref{lemmaG2Ricciflat}).

\item \fbox{coefficient $\partial\gamma^{i}:$}\hspace{0.1in} $0$,
\begin{equation}
(-1)(g^{ij}\Gamma^{k}_{il}\Gamma^{l}_{ja})_{,k}\partial\gamma^{a}+\dfrac{1}{2}\partial(g^{ij}\Gamma^{k}_{il}\Gamma^{l}_{jk}).\nonumber
\end{equation}
Expanding each summand separately:
\begin{align*}
&(-1)(g^{ij}\Gamma^{k}_{il}\Gamma^{l}_{ja})_{,k}\partial\gamma^{a}\\
=&(-1)g^{ij}_{,k}\Gamma^{k}_{il}\Gamma^{l}_{ja}\partial\gamma^{a}+(-1)g^{ij}(\Gamma^{k}_{il})_{,k}\Gamma^{l}_{ja}\partial\gamma^{a}+(-1)g^{ij}\Gamma^{k}_{il}(\Gamma^{l}_{ja})_{,k}\partial\gamma^{a}\\
=&(g^{sj}\Gamma^{i}_{sk}+g^{is}\Gamma^{j}_{sk})\Gamma^{k}_{il}\Gamma^{l}_{ja}\partial\gamma^{a}+(-1)g^{ij}(\Gamma^{m}_{ik}\Gamma^{k}_{lm})\Gamma^{l}_{ja}\partial\gamma^{a}+(-1)g^{ij}R_{il}\Gamma^{l}_{ja}\partial\gamma^{a}\\
&+(-1)g^{ij}\Gamma^{k}_{il}\left[R^{l}_{jka}+(\Gamma^{l}_{jk})_{,a}-\Gamma^{l}_{ks}\Gamma^{s}_{ja}+\Gamma^{l}_{as}\Gamma^{s}_{jk}\right]\partial\gamma^{a}\\
=&(-1)g^{ij}R_{il}\Gamma^{l}_{ja}\partial\gamma^{a}+(-1)g^{ij}\Gamma^{k}_{il}(\Gamma^{l}_{jk})_{,a}\partial\gamma^{a}+g^{ij}\Gamma^{k}_{il}\Gamma^{l}_{ks}\Gamma^{s}_{ja}\partial\gamma^{a}.
\end{align*}
Above, to cancel out the summand involving the curvature tensor we used that the Riemann's curvature tensor is antisymmetric in the first two indices (\ref{symmetriesoftheRiemannTensor}).
\begin{align*}
&\dfrac{1}{2}\partial(g^{ij}\Gamma^{k}_{il}\Gamma^{l}_{jk})\\
=&\dfrac{1}{2}g^{ij}_{,a}\Gamma^{k}_{il}\Gamma^{l}_{jk}\partial\gamma^{a}+\dfrac{1}{2}g^{ij}(\Gamma^{k}_{il})_{,a}\Gamma^{l}_{jk}\partial\gamma^{a}+\dfrac{1}{2}g^{ij}\Gamma^{k}_{il}(\Gamma^{l}_{jk})_{,a}\partial\gamma^{a}\nonumber\\
=&\dfrac{1}{2}\left(-g^{sj}\Gamma^{i}_{sa}-g^{is}\Gamma^{j}_{sa}\right)\Gamma^{k}_{il}\Gamma^{l}_{jk}\partial\gamma^{a}+g^{ij}\Gamma^{k}_{il}(\Gamma^{l}_{jk})_{,a}\partial\gamma^{a}\nonumber\\
=&(-1)g^{sj}\Gamma^{i}_{sa}\Gamma^{k}_{il}\Gamma^{l}_{jk}\partial\gamma^{a}+g^{ij}\Gamma^{k}_{il}(\Gamma^{l}_{jk})_{,a}\partial\gamma^{a}.\nonumber
\end{align*}
Combining the two expansions and using Lemma \ref{lemmaG2Ricciflat} we get the desired result.
\end{itemize}

\subsubsection{Coefficient of \texorpdfstring{$\lambda^0$}{}}

We compute the coefficient of each type of term that appears:
\begin{itemize}
\item \fbox{coefficient $\partial c^{i}b_{j}:$} \hspace{0.1in} $0$,
\begin{align*}
&\dfrac{1}{4}(g^{lm}_{,i})_{,m}\partial c^{i}b_{l}+(-\dfrac{1}{2})g^{lm}\Gamma^{n}_{la}\Gamma^{a}_{mi}\partial c^{i}b_{n}+(-\dfrac{1}{4})(g^{ij}\Gamma^{m}_{ik})_{,m}\partial c^{k}b_{j}\\
&+\dfrac{1}{4}(g^{im}\Gamma^{l}_{ik})_{,m}\partial c^{k}b_{l}\nonumber\\
=&(-\dfrac{1}{2})g^{lm}\Gamma^{n}_{la}\Gamma^{a}_{mi}\partial c^{i}b_{n}+(-\dfrac{1}{2})(g^{ij}\Gamma^{m}_{ik})_{,m}\partial c^{k}b_{j}\nonumber\\
=&\dfrac{1}{2}g^{aj}\Gamma^{i}_{ma}\Gamma^{m}_{ik}\partial c^{k}b_{j}+(-\dfrac{1}{2})g^{ij}(\Gamma^{m}_{ik})_{,m}\partial c^{k}b_{j}\nonumber\\
=&(-\dfrac{1}{2})g^{sj}R_{sk}\partial c^{k}b_{j}\nonumber\\
=&0,\nonumber
\end{align*}
\noindent
by Lemma \ref{lemmaG2Ricciflat}.

\item \fbox{coefficient $\partial\gamma^{i}\beta_{j}:$} \hspace{0.1in} $0$.

\item \fbox{coefficient $c^{i}b_{j}\beta_{k}:$} \hspace{0.1in} $0$.

\item \fbox{coefficient $c^{i}\partial b_{j}:$}\hspace{0.1in} $0$,
\begin{align*}
&(-\dfrac{1}{2})g^{lm}\Gamma^{n}_{la}\Gamma^{a}_{mi} c^{i}\partial b_{n}+(-\dfrac{1}{2})(g^{ij}\Gamma^{m}_{ik})_{,m} c^{k}\partial b_{j}\nonumber\\
=&\dfrac{1}{2}g^{aj}\Gamma^{i}_{ma}\Gamma^{m}_{ik} c^{k}\partial b_{j}+(-\dfrac{1}{2})(g^{ij}\Gamma^{m}_{ik})_{,m} c^{k}\partial b_{j}\nonumber\\
=&(-\dfrac{1}{2})g^{sj}R_{sk} c^{k}\partial b_{j}\nonumber\\
=&0,\nonumber
\end{align*}
\noindent
by Lemma \ref{lemmaG2Ricciflat}.

\item \fbox{coefficient $c^{i}c^{j}b_{m}b_{n}:$}\hspace{0.1in} $0$.

\item \fbox{coefficient $\partial\gamma^{i}c^{j}b_{k}:$} \hspace{0.1in} $0$,

using that the metric is covariantly constant we get:
\begin{align*}
&(-\dfrac{1}{2})\partial\left[(g^{lk}\Gamma^{n}_{lm})_{,n}\right]c^{m}b_{k}+(-\dfrac{1}{2})\partial(g^{ij}\Gamma^{k}_{il}\Gamma^{l}_{jm})c^{m}b_{k}\nonumber\\
=&(-\dfrac{1}{2})\partial \left[(g^{lk}\Gamma^{n}_{lm})_{,n}+g^{ij}\Gamma^{k}_{il}\Gamma^{l}_{jm}\right]c^{m}b_{k}\nonumber\\
=&(-\dfrac{1}{2})\partial\left[g^{sm}R_{sa}\right]c^{a}b_{m}\nonumber\\
=&0,\nonumber
\end{align*}
\noindent
by Lemma \ref{lemmaG2Ricciflat}.

\item \fbox{coefficient $\partial\beta_{m}:$}\hspace{0.1in} $0$.

\item \fbox{coefficient $\partial c^{i}c^{j}:$}\hspace{0.1in} $0$.

\item \fbox{coefficient $\partial\gamma^{i} c^{j}c^{k}:$}\hspace{0.1in} $0$.

\item \fbox{coefficient $ b^{i}b^{j}\beta^{k}:$} \hspace{0.1in} $0$,

using that the metric is covariantly constant we get:
\begin{eqnarray}
(-\dfrac{1}{2})\left(g^{lm}g^{ia}\Gamma^{j}_{am}\right)b_{i}b_{l}\beta_{j}=0,\nonumber 
\end{eqnarray}
because the factor between parentheses is symmetric in $i$ and $l$.

Due to the quasi-associativity (\ref{quasi-associativity}) we obtain some terms of type $\partial\gamma^{i}b_{j}b_{k}$:
\begin{align}\label{restfrombbbeta}
Q_{1}=&(-\dfrac{1}{2})\partial(g^{lm})\left(g^{aj}\Gamma^{i}_{am}\right)_{,j}b_{i}b_{l}+(-\dfrac{1}{2})g^{lm}_{,j}\partial\left(g^{aj}\Gamma^{i}_{am}\right)b_{i}b_{l}\nonumber\\
&+(-\dfrac{1}{2})\partial(g^{lm})\left(g^{ia}\Gamma^{j}_{am}\right)_{,j}b_{i}b_{l}+(-\dfrac{1}{2})g^{lm}_{,j}\partial\left(g^{ia}\Gamma^{j}_{am}\right)b_{i}b_{l}\nonumber\\
&+\dfrac{1}{2}\partial(g^{lm})\left(g^{ij}\Gamma^{a}_{il}\right)_{,m}b_{a}b_{j}+\dfrac{1}{2}g^{lm}_{,m}\partial\left(g^{ij}\Gamma^{a}_{il}\right)b_{a}b_{j}.
\end{align}
\item \fbox{coefficient $ \gamma^{i}\gamma^{j}:$} \hspace{0.1in} $0$,

using that the metric is covariantly constant we get:
\begin{align*}
&(-\dfrac{1}{2})\partial[(g^{lm}\Gamma^{n}_{la}\Gamma^{a}_{mr})_{,n}]\partial\gamma^{r}+\dfrac{1}{4}\partial^{2}(g^{ij}\Gamma^{k}_{il}\Gamma^{l}_{jk})+(-\dfrac{1}{8})g_{ij}\partial^{2}(g^{ij})\nonumber\\
&+(-\dfrac{1}{8})\partial^{2}(g_{ij})g^{ij}+\dfrac{1}{2}g^{rm}g_{al}\Gamma^{l}_{rs}\Gamma^{a}_{im}\partial\gamma^{s}\partial\gamma^{i}+\dfrac{1}{2}\Gamma^{m}_{rs}\Gamma^{r}_{mi}\partial\gamma^{s}\partial\gamma^{i}\nonumber\\
=&0.\nonumber
\end{align*}
Here we used the following three identities
$$0=(-\dfrac{1}{8})\partial^{2}(g_{ij}g^{ij})=(-\dfrac{1}{8})\partial^{2}g_{ij}g^{ij}+(-\dfrac{1}{4})\partial g_{ij}\partial g^{ij}+(-\dfrac{1}{8})g_{ij}\partial^{2}g^{ij},$$
$$\dfrac{1}{4}\partial g_{ij}\partial g^{ij}=(-\dfrac{1}{2})\Gamma^{l}_{is}\Gamma^{i}_{lr}\partial\gamma^{s}\partial\gamma^{r}+(-\dfrac{1}{2})g_{aj}g^{il}\Gamma^{a}_{is}\Gamma^{j}_{lr}\partial\gamma^{s}\partial\gamma^{r},$$
$$0=(-\dfrac{1}{2})\partial[(g^{lm}\Gamma^{n}_{la}\Gamma^{a}_{mr})_{,n}]\partial\gamma^{r}+(-\dfrac{1}{2})(g^{lm}\Gamma^{n}_{la}\Gamma^{a}_{mr})_{,n}\partial^{2}\gamma^{r}+\dfrac{1}{4}\partial^{2}(g^{ij}\Gamma^{k}_{il}\Gamma^{l}_{jk}).$$
The last identity follows taking the derivative of the equality:

$0=(-1)(g^{ij}\Gamma^{k}_{il}\Gamma^{l}_{ja})_{,k}\partial\gamma^{a}+\dfrac{1}{2}\partial(g^{ij}\Gamma^{k}_{il}\Gamma^{l}_{jk})$ \hspace{0.05in}(this is exactly the coefficient of $\partial\gamma^{i}$ in the terms with $\lambda$ in the subsection \ref{coefflambdainG+lambdaG-} above).

Due to the quasi-associativity (\ref{quasi-associativity}) we obtain some terms of type $\partial^{2}\gamma^{i}$:
\begin{equation}\label{restfromdergammadergamma}
Q_{2}=\dfrac{1}{2}(g^{ij}\Gamma^{k}_{il}\Gamma^{l}_{ja})_{,k}\partial^{2}\gamma^{a}.
\end{equation}
\item \fbox{coefficient $\partial^{2}\gamma^{i} :$} \hspace{0.1in} $0$,

just using that the metric is covariantly constant and taking into account the term (\ref{restfromdergammadergamma}) that comes from the coefficient of $\gamma^{i}\gamma^{j}$.

\item \fbox{coefficient $ c^{i}b^{j}b^{k}b^{l}:$} \hspace{0.1in} $0$,
\begin{align*}
&(-\dfrac{1}{2})g^{lm}g^{ij}(\Gamma^{a}_{ik})_{,m}c^{k}b_{a}b_{j}b_{l}+(-\dfrac{1}{2})g^{ij}g^{mn}\Gamma^{l}_{ik}\Gamma^{k}_{ma}c^{a}b_{j}b_{l}b_{n}\\
&+\dfrac{1}{2}g^{lm}g^{sj}\Gamma^{i}_{sm}\Gamma^{a}_{ik}c^{k}b_{a}b_{j}b_{l}\nonumber\\
=&\dfrac{1}{2}g^{lm}g^{ij}\tensor{R}{^a_i_k_m}c^{k}b_{a}b_{j}b_{l}+(-\dfrac{1}{2})g^{lm}g^{ij}(\Gamma^{a}_{im})_{,k}c^{k}b_{a}b_{j}b_{l}\\
&+(-\dfrac{1}{2})g^{lm}g^{ij}\Gamma^{a}_{kx}\Gamma^{x}_{im}c^{k}b_{a}b_{j}b_{l}+\dfrac{1}{2}g^{lm}g^{ij}\Gamma^{a}_{mx}\Gamma^{x}_{ik}c^{k}b_{a}b_{j}b_{l}\\
&+(-\dfrac{1}{2})g^{ij}g^{mn}\Gamma^{l}_{ik}\Gamma^{k}_{ma}c^{a}b_{j}b_{l}b_{n}+\dfrac{1}{2}g^{lm}g^{sj}\Gamma^{i}_{sm}\Gamma^{a}_{ik}c^{k}b_{a}b_{j}b_{l}\nonumber\\
=&-\dfrac{1}{2}\tensor{R}{^a^j^l_k}c^{k}b_{a}b_{j}b_{l}\nonumber\\
=&-\dfrac{1}{2}R^{[ajl]}_{\;\;\;\;\;\;\;k}c^{k}b_{a}b_{j}b_{l}\nonumber\\
=&0.\nonumber
\end{align*}
Here [\;\;\;] denotes the anti-symmetrization of the indices, the last equality follows from the Bianchi's first identity (\ref{firstBianchiidentity}) which implies that $R^{[ajl]}_{\;\;\;\;\;\;\;k}=0$. 

\item \fbox{coefficient $ \partial b_{i}b_{j}:$} \hspace{0.1in} $0$,

denote by $A_{1}$ terms containing derivatives of the Christoffel symbols are:
\small
\begin{equation}
A_{1}=\dfrac{1}{2}g^{lm}g^{ia}(\Gamma^{j}_{am})_{,j}b_{i}\partial b_{l}+(-\dfrac{1}{2})g^{lm}g^{ij}(\Gamma^{k}_{ij})_{,m}b_{l}\partial b_{k}+\dfrac{1}{2}g^{ij}g^{am}(\Gamma^{l}_{aj})_{,m}b_{i}\partial b_{l},\nonumber
\end{equation}
\normalsize
denote by $A_{2}$ the sum of terms that doesn't contain derivatives of the Christoffel symbols:
\begin{equation}
A_{2}=(-\dfrac{1}{2})\tensor{\Gamma}{_i^j^i}\tensor{\Gamma}{_j^l^a}b_{a}\partial b_{l}+(-\dfrac{1}{2})\tensor{\Gamma}{_i^j^l}\tensor{\Gamma}{^a^i_j}b_{a}\partial b_{l}+(\dfrac{1}{2})\tensor{\Gamma}{_i^j^a}\tensor{\Gamma}{^i^l_j}b_{a}\partial b_{l}.\nonumber
\end{equation}
Replacing the derivatives of the Christoffel symbol in $A_{1}$ using (\ref{RiemanncurvatureChristoffel}) and (\ref{RiccitensorChristoffel}) we get:
\begin{align*}
A_{1}=&\dfrac{1}{2}g^{lm}g^{ia}\tensor{R}{_a_m}b_{i}\partial b_{l}+\dfrac{1}{2}g^{lm}g^{ia}\Gamma^{x}_{ay}\Gamma^{y}_{mx}b_{i}\partial b_{l}\\
&+\dfrac{1}{2}g^{lm}g^{ij}\tensor{R}{^k_i_j_m}b_{l}\partial b_{k}+(-\dfrac{1}{2})g^{lm}g^{ij}(\Gamma^{k}_{im})_{,j}b_{l}\partial b_{k}\\
&+(-\dfrac{1}{2})g^{lm}g^{ij}\Gamma^{k}_{js}\Gamma^{s}_{im}b_{l}\partial b_{k}+\dfrac{1}{2}g^{lm}g^{ij}\Gamma^{k}_{ms}\Gamma^{s}_{ij}b_{l}\partial b_{k}\\
&+\dfrac{1}{2}g^{ij}g^{am}\tensor{R}{^l_j_a_m}b_{i}\partial b_{l}+\dfrac{1}{2}g^{ij}g^{am}(\Gamma^{l}_{jm})_{,a}b_{i}\partial b_{l}\\
&+\dfrac{1}{2}g^{ij}g^{am}\Gamma^{l}_{as}\Gamma^{s}_{jm}b_{i}\partial b_{l}+(-\dfrac{1}{2})g^{ij}g^{am}\Gamma^{l}_{ms}\Gamma^{s}_{ja}b_{i}\partial b_{l}\\
=&\dfrac{1}{2}g^{lm}g^{ia}\tensor{R}{_a_m}b_{i}\partial b_{l}+\dfrac{1}{2}g^{lm}g^{ia}\Gamma^{x}_{ay}\Gamma^{y}_{mx}b_{i}\partial b_{l}\\
&+\dfrac{1}{2}g^{lm}g^{ij}\tensor{R}{^k_i_j_m}b_{l}\partial b_{k}+\dfrac{1}{2}g^{lm}g^{ij}\Gamma^{k}_{ms}\Gamma^{s}_{ij}b_{l}\partial b_{k}\\
&+\dfrac{1}{2}g^{ij}g^{am}\tensor{R}{^l_j_a_m}b_{i}\partial b_{l}+(-\dfrac{1}{2})g^{ij}g^{am}\Gamma^{l}_{ms}\Gamma^{s}_{ja}b_{i}\partial b_{l},
\end{align*}
then
\begin{align*}
A_{1}+A_{2}=&\dfrac{1}{2}g^{lm}g^{ia}\tensor{R}{_a_m}b_{i}\partial b_{l}+\dfrac{1}{2}g^{lm}g^{ij}\tensor{R}{^k_i_j_m}b_{l}\partial b_{k}+\dfrac{1}{2}g^{ij}g^{am}\tensor{R}{^l_j_a_m}b_{i}\partial b_{l}\nonumber\\
=&\dfrac{1}{2}g^{ij}g^{am}\tensor{R}{^l_j_a_m}b_{i}\partial b_{l}\nonumber\\
=&0,\nonumber
\end{align*}
using the symmetries (\ref{symmetriesoftheRiemannTensor}).

\item \fbox{coefficient $ \partial\gamma^{i} b_{j}b_{k}:$} \hspace{0.1in} $0$,

due to the quasi-associativity in the computations of the $ b^{i}b^{j}\beta^{k}$ coefficient we need to take into account (\ref{restfrombbbeta}).

Collecting the terms that contain $\partial\gamma^{i} b_{j}b_{k}$ : denote by $A_{1}$ the sum of the terms that does not contain derivative of the Christoffel symbols, denote by $A_{2}$ the sum of the terms containing derivative of the Christoffel symbols.
We have:
\begin{align*}
A_{1}=&\dfrac{1}{2}g^{lm}_{,j}g^{aj}_{,s}\Gamma^{i}_{am}\partial\gamma^{s}b_{i}b_{l}+g^{lm}g^{ij}_{,m}\Gamma^{k}_{in}\Gamma^{n}_{ja}\partial\gamma^{a}b_{k}b_{l}\\
&+(-1)g^{ij}g^{mn}\Gamma^{l}_{ik}\Gamma^{k}_{ms}\Gamma^{s}_{nr}\partial\gamma^{r}b_{j}b_{l}\\
=&\dfrac{1}{2}g^{lx}g^{ak}\Gamma^{m}_{xj}\Gamma^{j}_{ks}\Gamma^{i}_{am}\partial\gamma^{s}b_{i}b_{l}+(-\dfrac{1}{2})g^{lm}g^{xj}\Gamma^{i}_{xm}\Gamma^{k}_{in}\Gamma^{n}_{ja}\partial\gamma^{a}b_{k}b_{l}\nonumber\\
&+(-1)g^{lm}g^{ix}\Gamma^{j}_{xm}\Gamma^{k}_{in}\Gamma^{n}_{ja}\partial\gamma^{a}b_{k}b_{l}+(-1)g^{ij}g^{mn}\Gamma^{l}_{ik}\Gamma^{k}_{ms}\Gamma^{s}_{nr}\partial\gamma^{r}b_{j}b_{l},\nonumber
\end{align*}
\begin{align*}
A_{2}=&g^{rm}g^{ia}\Gamma^{l}_{rj}(\Gamma^{j}_{am})_{,s}\partial\gamma^{s}b_{i}b_{l}+(-\dfrac{1}{2})g^{rm}g^{ij}\Gamma^{l}_{rs}(\Gamma^{a}_{il})_{,m}\partial\gamma^{s}b_{a}b_{j}\nonumber\\
&+(-\dfrac{1}{2})g^{lr}g^{ij}\Gamma^{m}_{rs}(\Gamma^{a}_{il})_{,m}\partial\gamma^{s}b_{a}b_{j}+g^{lm}g^{ij}\Gamma^{k}_{in}(\Gamma^{n}_{ja})_{,m}\partial\gamma^{a}b_{k}b_{l}\nonumber\\
&+g^{lm}g^{ij}\Gamma^{n}_{ja}(\Gamma^{k}_{in})_{,m}\partial\gamma^{a}b_{k}b_{l}.\nonumber
\end{align*}
Substituting the derivative of the Christoffel symbols in the second, third and fourth summand of $A_{2}$ by the Riemann curvature tensor (\ref{RiemanncurvatureChristoffel}) we get:
\begin{align*}
A_{2}=&g^{rm}g^{ia}\Gamma^{l}_{rj}(\Gamma^{j}_{am})_{,s}\partial\gamma^{s}b_{i}b_{l}+\dfrac{1}{2}g^{rm}g^{ij}\Gamma^{l}_{rs}\tensor{R}{^a_l_i_m}\partial\gamma^{s}b_{a}b_{j}\\
&+(-\dfrac{1}{2})g^{rm}g^{ij}\Gamma^{l}_{rs}(\Gamma^{a}_{lm})_{,i}\partial\gamma^{s}b_{a}b_{j}+(-\dfrac{1}{2})g^{rm}g^{ij}\Gamma^{l}_{rs}\Gamma^{a}_{ix}\Gamma^{x}_{lm}\partial\gamma^{s}b_{a}b_{j}\\
&+\dfrac{1}{2}g^{rm}g^{ij}\Gamma^{l}_{rs}\Gamma^{a}_{mx}\Gamma^{x}_{li}\partial\gamma^{s}b_{a}b_{j}+\dfrac{1}{2}g^{lr}g^{ij}\Gamma^{m}_{rs}\tensor{R}{^a_l_i_m}\partial\gamma^{s}b_{a}b_{j}\\
&+(-\dfrac{1}{2})g^{lr}g^{ij}\Gamma^{m}_{rs}(\Gamma^{a}_{lm})_{,i}\partial\gamma^{s}b_{a}b_{j}+(-\dfrac{1}{2})g^{lr}g^{ij}\Gamma^{m}_{rs}\Gamma^{a}_{ix}\Gamma^{x}_{lm}\partial\gamma^{s}b_{a}b_{j}\\
&+\dfrac{1}{2}g^{lr}g^{ij}\Gamma^{m}_{rs}\Gamma^{a}_{mx}\Gamma^{x}_{li}\partial\gamma^{s}b_{a}b_{j}+(-1)g^{lm}g^{ij}\Gamma^{k}_{in}\tensor{R}{^n_j_a_m}\partial\gamma^{a}b_{k}b_{l}\\
&+g^{lm}g^{ij}\Gamma^{k}_{in}(\Gamma^{n}_{jm})_{,a}\partial\gamma^{a}b_{k}b_{l}+g^{lm}g^{ij}\Gamma^{k}_{in}\Gamma^{n}_{ax}\Gamma^{x}_{jm}\partial\gamma^{a}b_{k}b_{l}\\
&+(-1)g^{lm}g^{ij}\Gamma^{k}_{in}\Gamma^{n}_{mx}\Gamma^{x}_{ja}\partial\gamma^{a}b_{k}b_{l}+g^{lm}g^{ij}\Gamma^{n}_{ja}(\Gamma^{k}_{in})_{,m}\partial\gamma^{a}b_{k}b_{l}\\
=&(-1)g^{rm}g^{ij}\Gamma^{l}_{rs}\Gamma^{a}_{ix}\Gamma^{x}_{lm}\partial\gamma^{s}b_{a}b_{j}+(-\dfrac{1}{2})g^{lm}g^{ij}\Gamma^{k}_{in}\Gamma^{n}_{mx}\Gamma^{x}_{ja}\partial\gamma^{a}b_{k}b_{l}\\
&+\dfrac{1}{2}g^{lr}g^{ij}\Gamma^{m}_{rs}\Gamma^{a}_{mx}\Gamma^{x}_{li}\partial\gamma^{s}b_{a}b_{j}+g^{lm}g^{ij}\Gamma^{k}_{in}\Gamma^{n}_{ax}\Gamma^{x}_{jm}\partial\gamma^{a}b_{k}b_{l}.
\end{align*}
Finally we check that $A_{1}+A_{2}=0.$
\end{itemize}
\noindent
We conclude that $[{G_{+}}_{\lambda}G_{-}]=0$.

\subsection{\texorpdfstring{$[{G_{+}}_{\lambda} \Phi_{-}]$}{} and \texorpdfstring{$[{G_{-}}_{\lambda} \Phi_{+}]$}{}}\label{G+withPhi-}

Now we compute $[{G_{+}}_{\lambda} \Phi_{-}]$, $[{G_{-}}_{\lambda} \Phi_{+}]$ is computed similarly.
\begin{align*}
G_{+}=&\dfrac{1}{2}c^{i}\beta_{i}+\dfrac{1}{2}\partial\gamma^{i}b_{i}+\dfrac{1}{2}g^{ij}b_{i}\beta_{j}+\dfrac{1}{2}g^{ij}\Gamma^{l}_{ik}c^{k}b_{j}b_{l}+\dfrac{1}{2}g^{ij}\Gamma^{k}_{ij}\partial b_{k}\nonumber\\
&+\dfrac{1}{2}g_{ij}\partial\gamma^{i}c^{j}+g^{ij}\Gamma^{k}_{il}\Gamma^{l}_{jm}\partial\gamma^{m}b_{k},\nonumber
\end{align*}
\begin{align*}
\Phi_{-}=&\frac{i}{12\sqrt{2}}\varphi_{ijk} c^{i}c^{j}c^{k}+ \frac{-i}{4\sqrt{2}}\varphi_{ijk}g^{il}c^{j}c^{k}b_{l}+  \frac{i}{4\sqrt{2}}\varphi_{ijk} g^{il}g^{jm}c^{k} b_{l}b_{m}\nonumber\\
&+\frac{-i}{12\sqrt{2}}\varphi_{ijk} g^{il}g^{jm}g^{kn}b_{l}b_{m}b_{n}+\frac{i}{2\sqrt{2}}\varphi_{ijk}\Gamma^{i}_{mn}g^{jm}\partial \gamma^{n}g^{kl}b_{l}\nonumber\\
&-\frac{i}{2\sqrt{2}}\varphi_{ijk}\Gamma^{i}_{mn}g^{jm}\partial \gamma^{n}c^{k}.\nonumber
\end{align*}
We list the  non-zero $\lambda$-brackets between the summands of $G_{+}$ and the summands of $\Phi_{-}$, to compute these we used the $\mathtt{Mathematica}$ package \cite{Thielemans91}.
\begin{align*}
[{c^{i}\beta_{i}}_{\lambda}\varphi_{lmn}c^{l}c^{m}c^{n}]=&\varphi_{lmn,i}c^{i}c^{l}c^{m}c^{n}\nonumber\\
=&\left(\varphi_{amn}\Gamma^{a}_{il}+\varphi_{lan}\Gamma^{a}_{im}+\varphi_{lma}\Gamma^{a}_{in}\right)c^{i}c^{l}c^{m}c^{n}\nonumber\\
=&0,\nonumber
\end{align*}
\begin{align*}
[{c^{i}\beta_{i}}_{\lambda}\tensor{\varphi}{^l_m_n}c^{m}c^{n}b_{l}]=&\left(\tensor{\varphi}{^l_m_n_{,l}}c^{m}c^{n}\right)\lambda+\tensor{\varphi}{^l_m_n}c^{m}c^{n}\beta_{l}\nonumber\\
&+\tensor{\varphi}{^l_m_n_{,i}}c^{i}c^{m}c^{n}b_{l}+\partial\left(\tensor{\varphi}{^l_m_n_{,l}}\right)c^{m}c^{n},\nonumber
\end{align*}
\begin{align*}
[{c^{i}\beta_{i}}_{\lambda}\tensor{\varphi}{^l^m_n}c^{n}b_{l}b_{m}]=&2\left(\tensor{\varphi}{^l^m_n_{,m}}c^{n}b_{l}\right)\lambda+2\tensor{\varphi}{^l^m_n}c^{n}b_{l}\beta_{m}\nonumber\\
&+\tensor{\varphi}{^l^m_n_{,i}}c^{i}c^{n}b_{l}b_{m}+2\partial\left(\tensor{\varphi}{^l^m_n_{,m}}\right)c^{n}b_{l},\nonumber
\end{align*}
\begin{align*}
[{c^{i}\beta_{i}}_{\lambda}\tensor{\varphi}{^l^m^n}b_{l}b_{m}b_{n}]=&3\left(\tensor{\varphi}{^l^m^n_{,n}}b_{l}b_{m}\right)\lambda+3\tensor{\varphi}{^l^m^n}b_{l}b_{m}\beta_{n}\nonumber\\
&+\tensor{\varphi}{^l^m^n_{,i}}c^{i}b_{l}b_{m}b_{n}+3\partial\left(\tensor{\varphi}{^l^m^n_{,n}}\right)b_{l}b_{m},\nonumber
\end{align*}
\begin{align*}
[{c^{i}\beta_{i}}_{\lambda}\tensor{\varphi}{_l^m^n}\Gamma^{l}_{ms}\partial\gamma^{s}b_{n}]=&\left(\tensor{\varphi}{_l^m^n}\Gamma^{l}_{ms}c^{s}b_{n}+\left(\tensor{\varphi}{_l^m^n}\Gamma^{l}_{ms}\right)_{,n}\partial\gamma^{s}\right)\lambda\nonumber\\
&+\tensor{\varphi}{_l^m^n}\Gamma^{l}_{ms}\partial c^{s}b_{n}+\left(\tensor{\varphi}{_l^m^n}\Gamma^{l}_{ms}\right)\partial\gamma^{s}\beta_{n}\nonumber\\
&+\left(\tensor{\varphi}{_l^m^n}\Gamma^{l}_{ms}\right)_{,i}\partial\gamma^{s}c^{i}b_{n}+\partial[\left(\tensor{\varphi}{_l^m^n}\Gamma^{l}_{ms}\right)_{,n}]\partial\gamma^{s},\nonumber
\end{align*}
\begin{align*}
[{c^{i}\beta_{i}}_{\lambda}\tensor{\varphi}{_l^m_n}\Gamma^{l}_{ms}\partial\gamma^{s}c^{n}]=&\left(\tensor{\varphi}{_l^m_n}\Gamma^{l}_{ms}c^{s}c^{n}\right)\lambda+\tensor{\varphi}{_l^m_n}\Gamma^{l}_{ms}\partial c^{s}c^{n}\nonumber\\
&+\left(\tensor{\varphi}{_l^m_n}\Gamma^{l}_{ms}\right)_{,i}\partial\gamma^{s}c^{i}c^{n},\nonumber
\end{align*}
\begin{align*}
[{\partial\gamma^{i}b_{i}}_{\lambda}\tensor{\varphi}{_l_m_n}c^{l}c^{m}c^{n}]=&3\tensor{\varphi}{_l_m_n}\partial\gamma^{n}c^{l}c^{m},\nonumber
\end{align*}
\begin{align*}
[{\partial\gamma^{i}b_{i}}_{\lambda}\tensor{\varphi}{^l_m_n}c^{m}c^{n}b_{l}]=&2\tensor{\varphi}{^l_m_n}\partial\gamma^{m}c^{n}b_{l},\nonumber
\end{align*}
\begin{align*}
[{\partial\gamma^{i}b_{i}}_{\lambda}\tensor{\varphi}{^l^m_n}c^{n}b_{l}b_{m}]=&\tensor{\varphi}{^l^m_n}\partial\gamma^{n}b_{l}b_{m},\nonumber
\end{align*}
\begin{align*}
[{\partial\gamma^{i}b_{i}}_{\lambda}\tensor{\varphi}{_l^m_n}\Gamma^{l}_{ms}\partial\gamma^{s}c^{n}]=&\tensor{\varphi}{_l^m_n}\Gamma^{l}_{ms}\partial\gamma^{n}\partial\gamma^{s},\nonumber
\end{align*}
\begin{align*}
[{g^{ij}b_{i}\beta_{j}}_{\lambda}\tensor{\varphi}{_l_m_n}c^{l}c^{m}c^{n}]=&3\left(g^{ni}\tensor{\varphi}{_l_m_n_{,i}}c^{l}c^{m}\right)\lambda+3\left(g^{ni}\tensor{\varphi}{_l_m_n}\right)c^{l}c^{m}\beta_{i}\nonumber\\
&+(-1)g^{ij}\tensor{\varphi}{_l_m_n_{,j}}c^{l}c^{m}c^{n}b_{i}+3g^{ni}\partial[\tensor{\varphi}{_l_m_n_{,i}}]c^{l}c^{m}\nonumber\\
&+6\partial(g^{ni})\tensor{\varphi}{_l_m_n_{,i}}c^{l}c^{m}+3\partial\left(\tensor{\varphi}{_l_m_n}\right)g^{ni}_{,i}c^{l}c^{m},\nonumber
\end{align*}
\begin{align*}
[{g^{ij}b_{i}\beta_{j}}_{\lambda}\tensor{\varphi}{^l_m_n}c^{m}c^{n}b_{l}]=&(-2)\left(g^{ni}\tensor{\varphi}{^l_m_n_{,i}}c^{m}b_{l}\right)\lambda+2\left(g^{mi}\tensor{\varphi}{^l_m_n}\right)c^{n}b_{l}\beta_{i}\nonumber\\
&+g^{ij}\tensor{\varphi}{^l_m_n_{,j}}c^{m}c^{n}b_{i}b_{l}+2g^{mi}\partial[\tensor{\varphi}{^l_m_n_{,i}}]c^{n}b_{l}\nonumber\\
&+4\partial(g^{mi})\tensor{\varphi}{^l_m_n_{,i}}c^{n}b_{l}+2\partial\left(\tensor{\varphi}{^l_m_n}\right)g^{mi}_{,i}c^{n}b_{l},\nonumber
\end{align*}
\begin{align*}
[{g^{ij}b_{i}\beta_{j}}_{\lambda}\tensor{\varphi}{^l^m_n}c^{n}b_{l}b_{m}]=&\left(g^{ni}\tensor{\varphi}{^l^m_n_{,i}}b_{l}b_{m}\right)\lambda+\left(g^{ni}\tensor{\varphi}{^l^m_n}\right)b_{l}b_{m}\beta_{i}\nonumber\\
&+(-1)g^{ij}\tensor{\varphi}{^l^m_n_{,j}}c^{n}b_{i}b_{l}b_{m}+g^{ni}\partial\left(\tensor{\varphi}{^l^m_n_{,i}}\right)b_{l}b_{m}\nonumber\\
&+2\partial\left(g^{ni}\right)\tensor{\varphi}{^l^m_n_{,i}}b_{l}b_{m}+\partial\left(\tensor{\varphi}{^l^m_n}\right)g^{ni}_{,i}b_{l}b_{m},\nonumber
\end{align*}
\begin{align*}
[{g^{ij}b_{i}\beta_{j}}_{\lambda}\tensor{\varphi}{^l^m^n}b_{l}b_{m}b_{n}]=&g^{ij}\tensor{\varphi}{^l^m^n_{,j}}b_{i}b_{l}b_{m}b_{n},\nonumber
\end{align*}
\begin{align*}
[{g^{ij}b_{i}\beta_{j}}_{\lambda}\tensor{\varphi}{_l^m^n}\Gamma^{l}_{ms}\partial\gamma^{s}b_{n}]=&\left(g^{ij}\tensor{\varphi}{_l^m^n}\Gamma^{l}_{mj}b_{i}b_{n}\right)\lambda+g^{ij}\tensor{\varphi}{_l^m^n}\Gamma^{l}_{mj}\partial b_{i}b_{n}\nonumber\\
&+g^{ij}\left(\tensor{\varphi}{_l^m^n}\Gamma^{l}_{ms}\right)_{,j}\partial\gamma^{s}b_{i}b_{n}+\partial(g^{ij})\tensor{\varphi}{_l^m^n}\Gamma^{l}_{mj}b_{i}b_{n},\nonumber
\end{align*}
\begin{align*}
&[{g^{ij}b_{i}\beta_{j}}_{\lambda}\tensor{\varphi}{_l^m_n}\Gamma^{l}_{ms}\partial\gamma^{s}c^{n}]\\
=&\left((-1)g^{ij}\tensor{\varphi}{_l^m_n}\Gamma^{l}_{mj}c^{n}b_{i}+g^{nj}(\tensor{\varphi}{_l^m_n}\Gamma^{l}_{ms})_{,j}\partial\gamma^{s}\right.\nonumber\\
&\left.+\partial(g^{ns})\tensor{\varphi}{_l^m_n}\Gamma^{l}_{ms}\right)\lambda+(-1)g^{is}\tensor{\varphi}{_l^m_n}\Gamma^{l}_{ms}c^{n}\partial b_{i}+\left(g^{nj}\tensor{\varphi}{_l^m_n}\Gamma^{l}_{ms}\right)\partial\gamma^{s}\beta_{j}\nonumber\\
&+(-1)g^{ij}\left(\tensor{\varphi}{_l^m_n}\Gamma^{l}_{ms}\right)_{,j}\partial\gamma^{s}c^{n}b_{i}+g^{nj}\partial[\left(\tensor{\varphi}{_l^m_n}\Gamma^{l}_{ms}\right)_{,j}]\partial\gamma^{s}\nonumber\\
&+(-1)\partial(g^{is})\tensor{\varphi}{_l^m_n}\Gamma^{l}_{ms}c^{n}b_{i}+2\partial(g^{nj})\left(\tensor{\varphi}{_l^m_n}\Gamma^{l}_{ms}\right)_{,j}\partial\gamma^{s}\nonumber\\
&+\dfrac{1}{2}\partial^{2}(g^{ns})\tensor{\varphi}{_l^m_n}\Gamma^{l}_{ms}+g^{nj}_{,j}\partial\left(\tensor{\varphi}{_l^m_n}\Gamma^{l}_{ms}\right)\partial\gamma^{s},\nonumber
\end{align*}
\begin{align*}
[{g^{ij}\Gamma^{r}_{ik}c^{k}b_{j}b_{r}}_{\lambda}\tensor{\varphi}{_l_m_n}c^{l}c^{m}c^{n}]=&\left(6g^{in}\Gamma^{m}_{ik}\tensor{\varphi}{_l_m_n}c^{k}c^{l}\right)\lambda+3g^{ij}\Gamma^{n}_{ik}\tensor{\varphi}{_l_m_n}c^{k}c^{l}c^{m}b_{j}\nonumber\\
&+3g^{im}\Gamma^{r}_{ik}\tensor{\varphi}{_l_m_n}c^{k}c^{l}c^{n}b_{r}+6g^{in}\Gamma^{m}_{ik}\tensor{\varphi}{_l_m_n}\partial c^{k}c^{l}\nonumber\\
&+6\partial\left(g^{in}\Gamma^{m}_{ik}\right)\tensor{\varphi}{_l_m_n}c^{k}c^{l},\nonumber
\end{align*}
\begin{align*}
&[{g^{ij}\Gamma^{r}_{ik}c^{k}b_{j}b_{r}}_{\lambda}\tensor{\varphi}{^l_m_n}c^{m}c^{n}b_{l}]\\
=&\left(2g^{in}\Gamma^{m}_{ik}\tensor{\varphi}{^l_m_n}c^{k}b_{l}+2g^{ij}\Gamma^{n}_{il}\tensor{\varphi}{^l_m_n}c^{m}b_{j}\right.\nonumber\\
&+\left.2g^{im}\Gamma^{r}_{il}\tensor{\varphi}{^l_m_n}c^{n}b_{r}+2\partial\left(g^{in}\Gamma^{m}_{il}\right)\tensor{\varphi}{^l_m_n}\right)\lambda+2g^{ij}\Gamma^{n}_{ik}\tensor{\varphi}{^l_m_n}c^{k}c^{m}b_{j}b_{l}\nonumber\\
&+2g^{in}\Gamma^{r}_{ik}\tensor{\varphi}{^l_m_n}c^{k}c^{m}b_{l}b_{r}+g^{ij}\Gamma^{r}_{il}\tensor{\varphi}{^l_m_n}c^{m}c^{n}b_{j}b_{r}+2g^{ij}\Gamma^{n}_{il}\tensor{\varphi}{^l_m_n}c^{m}\partial b_{j}\nonumber\\
&+2g^{im}\Gamma^{r}_{il}\tensor{\varphi}{^l_m_n}c^{n}\partial b_{r}+2g^{in}\Gamma^{m}_{ik}\tensor{\varphi}{^l_m_n}\partial c^{k}b_{l}+2\partial\left(g^{in}\Gamma^{m}_{ik}\right)\tensor{\varphi}{^l_m_n}c^{k}b_{l}\nonumber\\
&+2\partial\left(g^{ij}\Gamma^{n}_{il}\right)\tensor{\varphi}{^l_m_n}c^{m}b_{j}+2\partial\left(g^{im}\Gamma^{r}_{il}\right)\tensor{\varphi}{^l_m_n}c^{n}b_{r}+\partial^{2}\left(g^{in}\Gamma^{m}_{il}\right)\tensor{\varphi}{^l_m_n},\nonumber
\end{align*}
\begin{align*}
&[{g^{ij}\Gamma^{r}_{ik}c^{k}b_{j}b_{r}}_{\lambda}\tensor{\varphi}{^l^m_n}c^{n}b_{l}b_{m}]\\
=&\left(2g^{ij}\Gamma^{n}_{im}\tensor{\varphi}{^l^m_n}b_{j}b_{l}+2g^{in}\Gamma^{r}_{im}\tensor{\varphi}{^l^m_n}b_{l}b_{r}\right)\lambda+2g^{in}\Gamma^{r}_{im}\tensor{\varphi}{^l^m_n}b_{l}\partial b_{r}\nonumber\\
&+g^{ij}\Gamma^{n}_{ik}\tensor{\varphi}{^l^m_n}c^{k}b_{j}b_{l}b_{m}+(-1)g^{in}\Gamma^{r}_{ik}\tensor{\varphi}{^l^m_n}c^{k}b_{l}b_{m}b_{r}\nonumber\\
&+2g^{ij}\Gamma^{r}_{il}\tensor{\varphi}{^l^m_n}c^{n}b_{j}b_{m}b_{r}+2g^{ij}\Gamma^{n}_{im}\tensor{\varphi}{^l^m_n}\partial b_{j}b_{l}+2\partial\left(g^{ij}\Gamma^{n}_{im}\right)\tensor{\varphi}{^l^m_n}b_{j}b_{l}\nonumber\\
&+2\partial\left(g^{in}\Gamma^{r}_{im}\right)\tensor{\varphi}{^l^m_n}b_{l}b_{r},\nonumber
\end{align*}
\begin{align*}
[{g^{ij}\Gamma^{r}_{ik}c^{k}b_{j}b_{r}}_{\lambda}\tensor{\varphi}{^l^m^n}b_{l}b_{m}b_{n}]=&3g^{ij}\Gamma^{r}_{in}\tensor{\varphi}{^l^m^n}b_{j}b_{l}b_{m}b_{r},\nonumber
\end{align*}
\begin{align*}
[{g^{ij}\Gamma^{r}_{ik}c^{k}b_{j}b_{r}}_{\lambda}\tensor{\varphi}{_l^m^n}\Gamma^{l}_{ms}\partial\gamma^{s}b_{n}]=&g^{ij}\Gamma^{r}_{in}\Gamma^{l}_{ms}\tensor{\varphi}{_l^m^n}\partial\gamma^{s}b_{j}b_{r},\nonumber
\end{align*}
\begin{align*}
&[{g^{ij}\Gamma^{r}_{ik}c^{k}b_{j}b_{r}}_{\lambda}\tensor{\varphi}{_l^m_n}\Gamma^{l}_{ms}\partial\gamma^{s}c^{n}]\\
=&g^{ij}\Gamma^{n}_{ik}\Gamma^{l}_{ms}\tensor{\varphi}{_l^m_n}\partial\gamma^{s}c^{k}b_{j}+(-1)g^{in}\Gamma^{r}_{ik}\Gamma^{l}_{ms}\tensor{\varphi}{_l^m_n}\partial\gamma^{s}c^{k}b_{r},\nonumber
\end{align*}
\begin{align*}
[{g^{ij}\Gamma^{k}_{ij}\partial b_{k}}_{\lambda}\tensor{\varphi}{_l_m_n}c^{l}c^{m}c^{n}]=&\left(3g^{ij}\Gamma^{m}_{ij}\tensor{\varphi}{_l_m_n}c^{l}c^{n}\right)\lambda+3\partial\left(g^{ij}\Gamma^{m}_{ij}\right)\tensor{\varphi}{_l_m_n}c^{l}c^{n},\nonumber
\end{align*}
\begin{align*}
[{g^{ij}\Gamma^{k}_{ij}\partial b_{k}}_{\lambda}\tensor{\varphi}{^l_m_n}c^{m}c^{n}b_{l}]=&\left(2g^{ij}\Gamma^{n}_{ij}\tensor{\varphi}{^l_m_n}c^{m}b_{l}\right)\lambda+2\partial\left(g^{ij}\Gamma^{n}_{ij}\right)\tensor{\varphi}{^l_m_n}c^{m}b_{l},\nonumber
\end{align*}
\begin{align*}
&[{g^{ij}\Gamma^{k}_{ij}\partial b_{k}}_{\lambda}\tensor{\varphi}{^l^m_n}c^{n}b_{l}b_{m}]\\
=&\left((-1)g^{ij}\Gamma^{n}_{ij}\tensor{\varphi}{^l^m_n}b_{l}b_{m}\right)\lambda+(-1)\partial(g^{ij}\Gamma^{n}_{ij})\tensor{\varphi}{^l^m_n}b_{l}b_{m},\nonumber
\end{align*}
\begin{align*}
&[{g^{ij}\Gamma^{k}_{ij}\partial b_{k}}_{\lambda}\tensor{\varphi}{_l^m_n}\Gamma^{l}_{ms}\partial\gamma^{s}c^{n}]\\
=&\left((-1)g^{ij}\Gamma^{n}_{ij}\tensor{\varphi}{_l^m_n}\Gamma^{l}_{ms}\partial\gamma^{s}\right)\lambda+(-1)\partial\left(g^{ij}\Gamma^{n}_{ij}\right)\tensor{\varphi}{_l^m_n}\Gamma^{l}_{ms}\partial\gamma^{s},\nonumber
\end{align*}
\begin{align*}
[{g_{ij}\partial\gamma^{i}c^{j}}_{\lambda}\tensor{\varphi}{^l_m_n}c^{m}c^{n}b_{l}]=&g_{il}\tensor{\varphi}{^l_m_n}\partial\gamma^{i}c^{m}c^{n},\nonumber
\end{align*}
\begin{align*}
[{g_{ij}\partial\gamma^{i}c^{j}}_{\lambda}\tensor{\varphi}{^l^m_n}c^{n}b_{l}b_{m}]=&2g_{im}\tensor{\varphi}{^l^m_n}\partial\gamma^{i}c^{n}b_{l},\nonumber
\end{align*}
\begin{align*}
[{g_{ij}\partial\gamma^{i}c^{j}}_{\lambda}\tensor{\varphi}{^l^m^n}b_{l}b_{m}b_{n}]=&3g_{in}\tensor{\varphi}{^l^m^n}\partial\gamma^{i}b_{l}b_{m},\nonumber
\end{align*}
\begin{align*}
[{g_{ij}\partial\gamma^{i}c^{j}}_{\lambda}\tensor{\varphi}{_l^m^n}\Gamma^{l}_{ms}\partial\gamma^{s}b_{n}]=&g_{in}\tensor{\varphi}{_l^m^n}\Gamma^{l}_{ms}\partial\gamma^{i}\partial\gamma^{s},\nonumber
\end{align*}
\begin{align*}
[{g^{ij}\Gamma^{k}_{ia}\Gamma^{a}_{jr}\partial\gamma^{r}b_{k}}_{\lambda}\tensor{\varphi}{_l_m_n}c^{l}c^{m}c^{n}]=&3g^{ij}\Gamma^{n}_{ia}\Gamma^{a}_{jr}\tensor{\varphi}{_l_m_n}\partial\gamma^{r}c^{l}c^{m},\nonumber
\end{align*}
\begin{align*}
[{g^{ij}\Gamma^{k}_{ia}\Gamma^{a}_{jr}\partial\gamma^{r}b_{k}}_{\lambda}\tensor{\varphi}{^l_m_n}c^{m}c^{n}b_{l}]=&(-2)g^{ij}\Gamma^{n}_{ia}\Gamma^{a}_{jr}\tensor{\varphi}{^l_m_n}\partial\gamma^{r}c^{m}b_{l},\nonumber
\end{align*}
\begin{align*}
[{g^{ij}\Gamma^{k}_{ia}\Gamma^{a}_{jr}\partial\gamma^{r}b_{k}}_{\lambda}\tensor{\varphi}{^l^m_n}c^{n}b_{l}b_{m}]=&g^{ij}\Gamma^{n}_{ia}\Gamma^{a}_{jr}\tensor{\varphi}{^l^m_n}\partial\gamma^{r}b_{l}b_{m},\nonumber
\end{align*}
\begin{align*}
[{g^{ij}\Gamma^{k}_{ia}\Gamma^{a}_{jr}\partial\gamma^{r}b_{k}}_{\lambda}\tensor{\varphi}{_l^m_n}\Gamma^{l}_{ms}\partial\gamma^{s}c^{n}]=&g^{ij}\Gamma^{n}_{ia}\Gamma^{a}_{jr}\tensor{\varphi}{_l^m_n}\Gamma^{l}_{ms}\partial\gamma^{r}\partial\gamma^{s}.\nonumber
\end{align*}

\subsubsection{Coefficient of \texorpdfstring{$\lambda$}{}}

We prove that it is zero after an easy computation using only that the metric is covariantly constant and that $d\varphi=0$, (\ref{Phicovariantlyconstant}).

\subsubsection{Coefficient of \texorpdfstring{$\lambda^0$}{}}
We compute the coefficient of each type of term that appears:
\begin{itemize}
\item \fbox{Coefficient of $c^{i}c^{j}\beta_{k}$:} \hspace{0.1in} $0$.

\item \fbox{Coefficient of $c^{i}c^{j}c^{k}b_{l}$:} \hspace{0.1in} $0$, we use that $d\varphi=0.$

\item \fbox{Coefficient of $\partial\gamma^{s}c^{i}c^{j}$:}\\

After some simplifications using that $d\varphi=0$ and $\nabla g=0$ we arrived at the following expression:
\begin{align*}
&(-\dfrac{i}{4\sqrt{2}})\left(g^{ij}\Gamma^{k}_{lm}\Gamma^{l}_{js}\tensor{\varphi}{_k_i_n}\partial\gamma^{s}c^{m}c^{n}+(-1)g^{ij}\Gamma^{k}_{sl}\Gamma^{l}_{jm}\tensor{\varphi}{_k_i_n}\partial\gamma^{s}c^{m}c^{n}\right.\nonumber\\
&\left.+g^{ij}\left(\Gamma^{k}_{js}\right)_{,m}\tensor{\varphi}{_k_i_n}\partial\gamma^{s}c^{m}c^{n}+(-1)g^{ij}\left(\Gamma^{k}_{jm}\right)_{,s}\tensor{\varphi}{_k_i_n}\partial\gamma^{s}c^{m}c^{n}\right)\nonumber\\
=&(-\dfrac{i}{4\sqrt{2}})g^{ij}\tensor{\varphi}{_k_i_n}\tensor{R}{^k_j_m_s}\partial\gamma^{s}c^{m}c^{n}\nonumber\\
=&0,\nonumber
\end{align*}
the last equality is zero by Lemma \ref{lemmaG2Curvatureidentitywiththreeform}.

\item \fbox{Coefficient of $c^{i}b_{j}\beta_{k}$:} \hspace{0.1in} $0$.

\item \fbox{Coefficient of $c^{i}c^{j}b_{l}b_{m}$:}\hspace{0.1in} $0$, we use that $d\varphi=0$.

\item \fbox{Coefficient of $\partial\gamma^{s}c^{i}b_{j}$:}

After some simplifications using that $d\varphi=0$ and $\nabla g=0$ we arrived at the following expression:
\begin{align*}
&\dfrac{i}{4\sqrt{2}}\left(\tensor{\varphi}{_i^j^k}\left(\Gamma^{i}_{js}\right)_{,l}\partial\gamma^{s}c^{l}b_{k}+(-1)\tensor{\varphi}{_i^j^k}\left(\Gamma^{i}_{jl}\right)_{,s}\partial\gamma^{s}c^{l}b_{k}\right.\\
&\left.+\tensor{\varphi}{_i^j^k}\Gamma^{i}_{la}\Gamma^{a}_{js}\partial\gamma^{s}c^{l}b_{k}+(-1)\tensor{\varphi}{_i^j^k}\Gamma^{i}_{sa}\Gamma^{a}_{jl}\partial\gamma^{s}c^{l}b_{k}\right)\\
&+\dfrac{i}{4\sqrt{2}}\left(g^{ij}\tensor{\varphi}{_l^m_n}\left(\Gamma^{l}_{ms}\right)_{,j}\partial\gamma^{s}c^{n}b_{i}+(-1)g^{ij}\tensor{\varphi}{_l^m_n}\left(\Gamma^{l}_{mj}\right)_{,s}\partial\gamma^{s}c^{n}b_{i}\right.\\
&\left.+g^{ij}\tensor{\varphi}{_l^m_n}\Gamma^{l}_{ja}\Gamma^{a}_{ms}\partial\gamma^{s}c^{n}b_{i}+(-1)g^{ij}\tensor{\varphi}{_l^m_n}\Gamma^{l}_{sa}\Gamma^{a}_{mj}\partial\gamma^{s}c^{n}b_{i}\right)\\
=&\dfrac{i}{4\sqrt{2}}\tensor{\varphi}{_i^j^k}\tensor{R}{^i_j_l_s}\partial\gamma^{s}c^{l}b_{k}+\dfrac{i}{4\sqrt{2}}g^{ij}\tensor{\varphi}{_l^m_n}\tensor{R}{^l_m_j_s}\partial\gamma^{s}c^{n}b_{i}\nonumber\\
=&0,
\end{align*}
the last equality is zero by Lemma \ref{lemmaG2Curvatureidentitywiththreeform}.

\item  \fbox{Coefficient of $b_{i}b_{j}\beta_{k}$:}\hspace{0.1in} $0$.

\item \fbox{Coefficient of $c^{l}b_{i}b_{j}b_{k}$:}\hspace{0.1in} $0$, we use that $d\varphi=0$.

\item \fbox{Coefficient of $\partial\gamma^{s}b_{i}b_{j}$:}

After some simplifications using that $d\varphi=0$ and $\nabla g=0$ we arrived at the following expression:
\begin{align*}
&\dfrac{i}{2\sqrt{2}}\left(g^{ij}\tensor{\varphi}{_l^m^n}\left(\Gamma^{l}_{ms}\right)_{,j}\partial\gamma^{s}b_{i}b_{n}+(-1)g^{ij}\tensor{\varphi}{_l^m^n}\left(\Gamma^{l}_{mj}\right)_{,s}\partial\gamma^{s}b_{i}b_{n}\right.\nonumber\\
&\left.+g^{ij}\tensor{\varphi}{_l^m^n}\Gamma^{l}_{ja}\Gamma^{a}_{ms}\partial\gamma^{s}b_{i}b_{n}+(-1)g^{ij}\tensor{\varphi}{_l^m^n}\Gamma^{l}_{sa}\Gamma^{a}_{mj}\partial\gamma^{s}b_{i}b_{n}\right)\nonumber\\
=&\dfrac{i}{2\sqrt{2}}g^{ij}\tensor{\varphi}{_l^m^n}\tensor{R}{^l_m_j_s}\partial\gamma^{s}b_{i}b_{n}\nonumber\\
=&0,\nonumber
\end{align*}
the last equality is zero by Lemma \ref{lemmaG2Curvatureidentitywiththreeform}.

\item \fbox{Coefficient of $c^{i}\partial b_{j}$:}\hspace{0.1in} $0$.

\item \fbox{Coefficient of $\partial c^{i}b_{j}$:}\hspace{0.1in} $0$.

\item \fbox{Coefficient of $\partial c^{i}c_{j}$:}\hspace{0.1in} $0$.

\item \fbox{Coefficient of $b_{i}b_{j}b_{k}b_{l}$:}\hspace{0.1in} $0$, we use that $d\varphi=0$.

\item \fbox{Coefficient of $\partial b_{i}b_{j}$:}\hspace{0.1in} $0$.

\item \fbox{Coefficient of $\partial\gamma^{s}\beta_{i}$:}\hspace{0.1in} $0$.

\item \fbox{Coefficient of $\partial\gamma^{s}\partial\gamma^{r}$:}

After some simplifications using that $d\varphi=0$ and $\nabla g=0$ we arrived at the following expression:
\begin{align*}
&\dfrac{(-i)}{2\sqrt{2}}\left(g^{ij}\Gamma^{k}_{ir}\left(\Gamma^{l}_{km}\right)_{,s}\tensor{\varphi}{^m_j_l}\partial\gamma^{s}\partial\gamma^{r}+(-1)g^{ij}\Gamma^{k}_{ir}\left(\Gamma^{l}_{ms}\right)_{,k}\tensor{\varphi}{^m_j_l}\partial\gamma^{s}\partial\gamma^{r}\right.\nonumber\\
&\left.+g^{ij}\Gamma^{k}_{ir}\left(\Gamma^{l}_{sa}\Gamma^{a}_{mk}\right)\tensor{\varphi}{^m_j_l}\partial\gamma^{s}\partial\gamma^{r}+(-1)g^{ij}\Gamma^{k}_{ir}\left(\Gamma^{l}_{ka}\Gamma^{a}_{ms}\right)\tensor{\varphi}{^m_j_l}\partial\gamma^{s}\partial\gamma^{r}\right)\nonumber\\
&+\dfrac{(-i)}{2\sqrt{2}}\left(g^{ij}\Gamma^{l}_{js}\left(\Gamma^{m}_{in}\right)_{,r}\tensor{\varphi}{^n_l_m}\partial\gamma^{r}\partial\gamma^{s}+(-1)g^{ij}\Gamma^{l}_{js}\left(\Gamma^{m}_{nr}\right)_{,i}\tensor{\varphi}{^n_l_m}\partial\gamma^{r}\partial\gamma^{s}\right.\nonumber\\
&\left.+g^{ij}\Gamma^{l}_{js}\left(\Gamma^{m}_{ra}\Gamma^{a}_{ni}\right)\tensor{\varphi}{^n_l_m}\partial\gamma^{r}\partial\gamma^{s}+(-1)g^{ij}\Gamma^{l}_{js}\left(\Gamma^{m}_{ia}\Gamma^{a}_{nr}\right)\tensor{\varphi}{^n_l_m}\partial\gamma^{r}\partial\gamma^{s}\right)\nonumber\\
=&\dfrac{(-i)}{2\sqrt{2}}g^{ij}\Gamma^{k}_{ir}\tensor{\varphi}{^m_j_l}\tensor{R}{^l_m_s_k}\partial\gamma^{s}\partial\gamma^{r}+\dfrac{(-i)}{2\sqrt{2}}g^{ij}\Gamma^{l}_{js}\tensor{\varphi}{^n_l_m}\tensor{R}{^m_n_r_i}\partial\gamma^{r}\partial\gamma^{s}\nonumber\\
=&0,\nonumber
\end{align*}
the last equality is zero by Lemma \ref{lemmaG2Curvatureidentitywiththreeform}.
\end{itemize}
Then we conclude that $[{G_{+}}_{\lambda} \Phi_{-}]=0$, similarly we obtain $[{G_{-}}_{\lambda} \Phi_{+}]=0$.

\subsection{\texorpdfstring{$[{G_{+}}_{\lambda} \Phi_{+}]$}{}}\label{G+ lambda Phi+}

$G=c^{i}\beta_{i}+\partial\gamma^{i}b_{i},$
\begin{align*}
\Phi_{+}=&\frac{1}{12\sqrt{2}}\varphi_{ijk} c^{i}c^{j}c^{k}+ \frac{1}{4\sqrt{2}}\varphi_{ijk}g^{il}c^{j}c^{k}b_{l} + \frac{1}{4\sqrt{2}}\varphi_{ijk} g^{il}g^{jm}c^{k} b_{l}b_{m}\\
&+ \frac{1}{12\sqrt{2}}\varphi_{ijk} g^{il}g^{jm}g^{kn}b_{l}b_{m}b_{n}+\frac{1}{2\sqrt{2}}\varphi_{ijk}\Gamma^{i}_{mn}g^{jm}\partial \gamma^{n}g^{kl}b_{l}\\
&+\frac{1}{2\sqrt{2}}\varphi_{ijk}\Gamma^{i}_{mn}g^{jm}\partial \gamma^{n}c^{k}.
\end{align*}
We can use the $\lambda$-brackets performed in section \ref{G+withPhi-}.

\subsubsection{Coefficient of \texorpdfstring{$\lambda^{2}$}{}}

We realize inmediatly that the coefficient of $\lambda^{2}$ is zero because there are not $\lambda^{2}$ terms.

\subsubsection{Coefficient of \texorpdfstring{$\lambda$}{}}
\begin{itemize}
\item \fbox{Coefficient of $c^{i}c^{j}$:} \hspace{0.1in} $0$, we use that $\nabla\varphi=0.$

\item \fbox{Coefficient of $c^{i}b_{j}$:} \hspace{0.1in} $0$, we use that $\nabla\varphi=0.$

\item \fbox{Coefficient of $b_{i}b_{j}$:}

$\dfrac{1}{4\sqrt{2}}\tensor{\varphi}{^l^m^n_{,n}}b_{l}b_{m}=0$.

As $\nabla\varphi=0$ (\ref{Phicovariantlyconstant}) we have,
\begin{align*}
\tensor{\varphi}{^l^m^n_{,n}}=-\tensor{\varphi}{^a^m^n}\Gamma^{l}_{an}-\tensor{\varphi}{^l^a^n}\Gamma^{m}_{an}-\tensor{\varphi}{^l^m^a}\Gamma^{n}_{an}=0,
\end{align*}
to conclude that the last equality is zero we used the symmetries of $\varphi$ and (\ref{localcoordinateassumption}).

\item \fbox{Coefficient of $\gamma^{s}$:}

$\dfrac{1}{2\sqrt{2}}\left(\tensor{\varphi}{_l^m^n}\Gamma^{l}_{ms}\right)_{,n}\partial\gamma^{s}=0.$

As $\nabla\varphi=0$ (\ref{Phicovariantlyconstant}) we have
\begin{align*}
\left(\tensor{\varphi}{_l^m^n}\Gamma^{l}_{ms}\right)_{,n}=&\tensor{\varphi}{_l^m^n_{,n}}\Gamma^{l}_{ms}+\tensor{\varphi}{_l^m^n}\left(\Gamma^{l}_{ms}\right)_{,n}\\
=&\left(\tensor{\varphi}{_a^m^n}\Gamma^{a}_{ln}-\tensor{\varphi}{_l^a^n}\Gamma^{m}_{an}-\tensor{\varphi}{_l^m^a}\Gamma^{n}_{an}\right)\Gamma^{l}_{ms}+\tensor{\varphi}{_l^m^n}\left(\Gamma^{l}_{ms}\right)_{,n}\\
=&\tensor{\varphi}{_a^m^n}\Gamma^{a}_{ln}\Gamma^{l}_{ms}+\tensor{\varphi}{_l^m^n}\left(\Gamma^{l}_{ms}\right)_{,n}\\
=&\tensor{\varphi}{_l^m^n}\left(\Gamma^{l}_{ms}\right)_{,n}+(-1)\tensor{\varphi}{_l^m^n}\left(\Gamma^{l}_{mn}\right)_{,s}+\tensor{\varphi}{_l^m^n}\Gamma^{l}_{na}\Gamma^{a}_{ms}\\
&+(-1)\tensor{\varphi}{_l^m^n}\Gamma^{l}_{sa}\Gamma^{a}_{mn}\\
=&\tensor{\varphi}{_l^m^n}\tensor{R}{^l_m_n_s}\\
=&0.
\end{align*}
That the last equality is zero follows by Lemma \ref{lemmaG2Curvatureidentitywiththreeform}.
\end{itemize}
Then $[{G_{+}}_{\lambda} \phi_{+}]=K_{+}$ and similarly $[{G_{-}}_{\lambda} \phi_{-}]=K_{-}$.

\subsection{\texorpdfstring{$[{L_{+}}_{\lambda}\Phi_{+}]$}{}}\label{L+ lambda Phi+}
\begin{align*}
L=&\partial\gamma^{i}\beta_{i}-\dfrac{1}{2}c^{i}\partial b_{i}+\dfrac{1}{2}\partial c^{i}b_{i},
\end{align*}
\begin{align*}
\Phi_{+}=&\frac{1}{12\sqrt{2}}\varphi_{ijk} c^{i}c^{j}c^{k}+ \frac{1}{4\sqrt{2}}\varphi_{ijk}g^{il}c^{j}c^{k}b_{l}\\
\quad&+ \frac{1}{4\sqrt{2}}\varphi_{ijk} g^{il}g^{jm}c^{k} b_{l}b_{m}+\frac{1}{12\sqrt{2}}\varphi_{ijk} g^{il}g^{jm}g^{kn}b_{l}b_{m}b_{n}\\
\quad&+\frac{1}{2\sqrt{2}}\varphi_{ijk}\Gamma^{i}_{mn}g^{jm}\partial \gamma^{n}g^{kl}b_{l}+\frac{1}{2\sqrt{2}}\varphi_{ijk}\Gamma^{i}_{mn}g^{jm}\partial \gamma^{n}c^{k}.
\end{align*}
The only brackets that can produce $\lambda^{2}$ are:
\begin{align*}
[{c^{i}\partial b_{i}}_{\lambda}\tensor{\varphi}{^l_m_n}c^{m}c^{n}b_{l}]=\left(2\tensor{\varphi}{^i_i_n}c^{n}\right)\lambda^{2}+ \text{terms linear in $\lambda$},
\end{align*}
\begin{align*}
[{c^{i}\partial b_{i}}_{\lambda}\tensor{\varphi}{^l^m_n}c^{n}b_{l}b_{m}]=\left(2\tensor{\varphi}{^l^i_i}b_{l}\right)\lambda^{2}+ \text{terms linear in $\lambda$},
\end{align*}
\begin{align*}
[{\partial c^{i}b_{i}}_{\lambda}\tensor{\varphi}{^l_m_n}c^{m}c^{n}b_{l}]=\left(2\tensor{\varphi}{^i_i_n}c^{n}\right)\lambda^{2}+ \text{terms linear in $\lambda$},
\end{align*}
\begin{align*}
[{\partial c^{i} b_{i}}_{\lambda}\tensor{\varphi}{^l^m_n}c^{n}b_{l}b_{m}]=\left(2\tensor{\varphi}{^l^i_i}b_{l}\right)\lambda^{2}+ \text{terms linear in $\lambda$}.
\end{align*}
Then the coefficient of $\lambda^{2}$ is zero. We conclude that:\\

$[{L_{+}}_{\lambda}\phi_{+}]=(\partial+\frac{3}{2}\lambda)\phi_{+},$ similarly we obtain $[{L_{-}}_{\lambda}\phi_{-}]=(\partial+\frac{3}{2}\lambda)\phi_{-}.$

\bibliography{bibliografia}

\begin{bibdiv}
\begin{biblist}

\bib{Blumenhagen92}{article}{
      author={Blumenhagen, R},
       title={Covariant construction of {$N=1$} super {W}-algebras},
        date={1992},
     journal={Nuclear Physics B},
      volume={381},
      number={3},
       pages={641\ndash 669},
}

\bib{Borcherds86}{article}{
      author={Borcherds, Richard~E.},
       title={Vertex algebras, {K}ac-{M}oody algebras, and the {M}onster},
        date={1986},
     journal={Proc. Nat. Acad. Sci. U.S.A.},
      volume={83},
      number={10},
       pages={3068\ndash 3071},
}

\bib{Ben-Zvi-Heluani-Szczesny}{article}{
      author={Ben-Zvi, David},
      author={Heluani, Reimundo},
      author={Szczesny, Matthew},
       title={Supersymmetry of the chiral de {R}ham complex},
        date={2008},
     journal={Compositio Mathematica},
      volume={144},
       pages={503\ndash 521},
}

\bib{CHNP15}{article}{
      author={Corti, Alessio},
      author={Haskins, Mark},
      author={Nordstr{\"o}m, Johannes},
      author={Pacini, Tommaso},
       title={{$G_2$}-manifolds and associative submanifolds via semi-{F}ano
  3-folds},
        date={2015},
        ISSN={0012-7094},
     journal={Duke Math. J.},
      volume={164},
      number={10},
       pages={1971\ndash 2092},
         url={http://dx.doi.org/10.1215/00127094-3120743},
      review={\MR{3369307}},
}

\bib{BoerNaqviShomer}{article}{
      author={de~Boer, Jan},
      author={Naqvi, Asad},
      author={Shomer, Assaf},
       title={The topological {$G_2$} string},
        date={2008},
        ISSN={1095-0761},
     journal={Adv. Theor. Math. Phys.},
      volume={12},
      number={2},
       pages={243\ndash 318},
         url={http://projecteuclid.org/euclid.atmp/1210083226},
      review={\MR{2399312}},
}

\bib{dandrea}{article}{
      author={D'Andrea, A.},
      author={Kac, V.~G.},
       title={Structure theory of finite conformal algebras},
        date={1998},
     journal={Selecta Mathematica},
      volume={4},
      number={3},
       pages={377\ndash 418},
}

\bib{DeSoleKac05}{article}{
      author={De~Sole, Alberto},
      author={Kac, Victor~G.},
       title={Freely generated vertex algebras and non-linear {L}ie conformal
  algebras},
        date={2005},
     journal={Comm. Math. Phys.},
      volume={254},
      number={3},
       pages={659\ndash 694},
}

\bib{DeSole-Kac06}{article}{
      author={De~Sole, Alberto},
      author={Kac, Victor~G.},
       title={Finite vs affine {W}-algebras},
        date={2006},
     journal={Japanese Journal of Mathematics},
      volume={1},
      number={1},
       pages={137\ndash 261},
}

\bib{Ekstrand-Heluani-Kallen-Zabzine13}{article}{
      author={Ekstrand, Joel},
      author={Heluani, Reimundo},
      author={Källén, Johan},
      author={Zabzine, Maxim},
       title={Chiral de {R}ham {C}omplex on {R}iemannian {M}anifolds and
  {S}pecial {H}olonomy},
        date={2013},
     journal={Communications in Mathematical Physics},
      volume={318},
      number={3},
       pages={575\ndash 613},
}

\bib{OFarrill97}{article}{
      author={Figueroa-O'Farrill, José~M.},
       title={A note on the extended superconformal algebras associated with
  manifolds of exceptional holonomy},
        date={1997},
     journal={Phys. Lett.},
      volume={B392},
       pages={77\ndash 84},
}

\bib{OFarrillSchrans91}{article}{
      author={Figueroa-O'Farrill, José~M.},
      author={Schrans, Stany},
       title={{Extended superconformal algebras}},
        date={1991},
     journal={Phys.Lett.},
      volume={B257},
       pages={69\ndash 73},
}

\bib{OFarrillSchrans92}{article}{
      author={Figueroa-O'Farrill, José~M.},
      author={Schrans, Stany},
       title={{The {C}onformal bootstrap and super {W} algebras}},
        date={1992},
     journal={Int.J.Mod.Phys.},
      volume={A7},
       pages={591\ndash 618},
}

\bib{Heluani09}{article}{
      author={Heluani, Reimundo},
       title={Supersymmetry of the chiral de {R}ham complex 2: {C}ommuting
  sectors},
        date={2009},
     journal={International Mathematics Research Notices},
      volume={2009},
       pages={953\ndash 987},
}

\bib{HowePapadopoulos91}{article}{
      author={Howe, P.S.},
      author={Papadopoulos, G.},
       title={A note on holonomy groups and sigma models},
        date={1991},
     journal={Physics Letters B},
      volume={263},
      number={2},
       pages={230 \ndash  232},
}

\bib{Howe-Papadopoulos93}{article}{
      author={Howe, P.~S.},
      author={Papadopoulos, G.},
       title={Holonomy groups and {$W$}-symmetries},
        date={1993},
        ISSN={0010-3616},
     journal={Comm. Math. Phys.},
      volume={151},
      number={3},
       pages={467\ndash 479},
         url={http://projecteuclid.org/euclid.cmp/1104252234},
      review={\MR{1207260}},
}

\bib{HeluaniRodriguez15}{article}{
      author={Heluani, Reimundo},
      author={Rodr{\'{\i}}guez~D{\'{\i}}az, L{\'a}zaro~O.},
       title={The {S}hatashvili-{V}afa {$G_2$} superconformal algebra as a
  quantum {H}amiltonian reduction of {$D(2,1;\alpha)$}},
        date={2015},
        ISSN={1678-7544},
     journal={Bull. Braz. Math. Soc. (N.S.)},
      volume={46},
      number={3},
       pages={331\ndash 351},
         url={http://dx.doi.org/10.1007/s00574-015-0094-x},
}

\bib{Joyce07}{book}{
      author={Joyce, Dominic~D},
       title={Riemannian holonomy groups and calibrated geometry},
      series={Oxford Graduate Texts in Mathematics},
   publisher={Oxford Univ. Press, UK},
        date={2007},
}

\bib{Kac96}{article}{
      author={Kac, Victor},
       title={Vertex operator algebras for beginners},
organization={AMS},
        date={1998},
     journal={University Lecture Series},
      volume={10},
}

\bib{Karigiannis09}{article}{
      author={Karigiannis, Spiro},
       title={Flows of {$G_{2}$} {S}tructures, {I}},
        date={2009},
     journal={The Quarterly Journal of Mathematics},
      volume={60},
      number={4},
       pages={487\ndash 522},
}

\bib{KacWakimoto04}{article}{
      author={Kac, Victor~G.},
      author={Wakimoto, Minoru},
       title={Quantum reduction and representation theory of superconformal
  algebras},
        date={2004},
     journal={Advances in Mathematics},
      volume={185},
      number={2},
       pages={400\ndash 458},
}

\bib{Malikov-Schechtman-Vaintrob99}{article}{
      author={Malikov, Fyodor},
      author={Schechtman, Vadim},
      author={Vaintrob, Arkady},
       title={Chiral de {R}ham {C}omplex},
        date={1999},
     journal={Communications in Mathematical Physics},
      volume={204},
      number={2},
       pages={439\ndash 473},
}

\bib{Odake89}{article}{
      author={Odake, Satoru},
       title={{Extension of $N=2$ {S}uperconformal {A}lgebra and {C}alabi-{Y}au
  {C}ompactification}},
        date={1989},
     journal={Mod.Phys.Lett.},
      volume={A4},
       pages={557},
}

\bib{Cadabra07}{article}{
      author={{Peeters}, Kasper},
       title={{Cadabra: a field-theory motivated symbolic computer algebra
  system.}},
    language={English},
        date={2007},
        ISSN={0010-4655},
     journal={{Comput. Phys. Commun.}},
      volume={176},
      number={8},
       pages={550\ndash 558},
}

\bib{salamon89}{book}{
      author={Salamon, Simon},
       title={Riemannian geometry and holonomy groups},
      series={Pitman research notes in mathematics series},
   publisher={Longman Scientific \& Technical New York},
        date={1989},
}

\bib{Shatashvili-Vafa95}{article}{
      author={Shatashvili, S.L.},
      author={Vafa, C.},
       title={Superstrings and manifolds of exceptional holonomy},
        date={1995},
     journal={Selecta Mathematica},
      volume={1},
      number={2},
       pages={347\ndash 381},
}

\bib{Thielemans91}{article}{
      author={Thielemans, K.},
       title={A {M}athematica package for computing operator product
  expansions},
        date={1991},
     journal={Int. J. Mod. Phys.},
      volume={C2},
       pages={787\ndash 798},
}

\end{biblist}
\end{bibdiv}

\end{document}